\documentclass{article}

\usepackage[margin=1.4in]{geometry}

\usepackage{natbib}
\usepackage{lmodern}
\usepackage{amssymb}
\usepackage{amsmath}
\usepackage{mathtools}
\usepackage{amsthm}
\usepackage{dsfont}
\usepackage{bbm}
\usepackage{graphicx}
\usepackage{comment}
\usepackage[utf8]{inputenc}
\usepackage{hyperref}
\usepackage{nicefrac}
\usepackage{multicol}
\usepackage{stmaryrd}
\usepackage{tikz}
\usepackage{enumitem}
\usepackage{color}
\newcommand{\semantics}[1]{[\![\mbox{\em $ #1 $\/}]\!]}

\theoremstyle{definition}
\newtheorem{definition}{Definition}[section]

\newtheorem{lemma}[definition]{Lemma}
\newtheorem{fact}[definition]{Fact}
\newtheorem{proposition}[definition]{Proposition}

\newtheorem{corollary}[definition]{Corollary}
\newtheorem{theorem}{Theorem}

\newcommand\blfootnote[1]{%
  \begingroup
  \renewcommand\thefootnote{}\footnote{#1}%
  \addtocounter{footnote}{-1}%
  \endgroup
}

\tikzset{
    vertex/.style={},
    edge/.style={->}
}

\newcommand{\indep}{\perp \!\!\! \perp}

\newcommand{\prob}{\mathbb{P}}

\newcommand{\ETR}{\exists\mathbb{R}}
\newcommand{\NP}{\mathsf{NP}}
\newcommand{\SAT}{\mathsf{SAT}}
\newcommand{\ind}[1]{ \mathds{1}_{#1} }
\newcommand{\mb}[1]{ \mathbf{#1}}
\newcommand\independent{\protect\mathpalette{\protect\independenT}{\perp}}
\def\independenT#1#2{\mathrel{\rlap{$#1#2$}\mkern2mu{#1#2}}}

\usepackage{mdframed}

\title{Probing the Quantitative--Qualitative Divide in Probabilistic Reasoning}

\author{Duligur Ibeling, Thomas Icard,\\ Krzysztof Mierzewski, Milan Moss\'{e}}

\begin{document}

\maketitle

%\tableofcontents

\begin{abstract}
This paper explores the space of (propositional) probabilistic logical languages, ranging from a purely `qualitative' comparative language to a highly `quantitative' language involving arbitrary polynomials over probability terms. While talk of qualitative vs. quantitative may be suggestive, we identify a robust and meaningful boundary in the space by distinguishing systems that encode (at most) additive reasoning from those that encode additive and multiplicative reasoning. The latter includes not only languages with explicit multiplication but also languages expressing notions of dependence and conditionality. We show that the distinction tracks a divide in computational complexity: additive systems remain complete for $\mathsf{NP}$, while multiplicative systems are robustly complete for $\exists\mathbb{R}$. We also address axiomatic questions, offering several new completeness results as well as a proof of non-finite-axiomatizability for comparative probability. Repercussions of our results for conceptual and empirical questions are addressed, and open problems are discussed. 
\end{abstract}

\blfootnote{This material is based upon work supported by the National Science Foundation Graduate Research Fellowship Program under Grant No. DGE-1656518.}

%This distinction cross-cuts the previous one, and it coincides with a meaningful complexity distinction. Namely, as we show, purely additive systems robustly enjoy NP-complete satisfiability problems, while the inclusion of multiplication robustly leads to ETR-completeness (ETR being the complexity class for the existential theory of the reals). Within these two classes, systems that are commonly extolled as simple and qualitative—most notably, comparative probability and comparative conditional probability—in addition to being not only are less expressive, but also possess less desirable metalogical properties than their (equally complex) explicitly arithmetical counterparts. We focus in particular on issues of (finite) axiomatization. The result is a more complete and meaningfully regimented picture of the space of probability logics.

\tableofcontents

\section{Introduction}

For as long as probability has been mathematized, numbers and numerical calculation have been at the center. From the treatment of probability in the Port Royal \emph{Logic} in terms of ratios of frequencies, to the modern axiomatic treatment of Kolmogorov, it has always been standard to formulate probabilistic reasoning in  fundamentally \emph{quantitative} terms. 

It may be surprising that the systematic mathematical analysis of more \emph{qualitative} probabilistic notions is relatively recent---beginning in earnest only after Kolmogorov's landmark treatise \citep{Kolmogorov}---and still occupies not much more than a footnote in the history of the subject.\footnote{\cite{fine1973theories} referred to comparative probability as a `neglected concept', and while the work by Fine and others inspired substantial subsequent development, it is still arguably only a marginal part of the field.} This is despite the fact that qualitative probabilistic locutions are of ancient origin, predating the various numerical concepts by millennia (see \citealt{Franklin2001}), and despite the insistence by many that some of the qualitative notions are somehow primary. In the words of \cite{koopman1940axioms}, qualitative comparisons like `more likely than' are an expression of `the primordial intuition of probability', while the use of numbers is `a mathematical construct derived from the latter under very special conditions' (p. 269). Similar attitudes were expressed earlier  by Keynes, de Finetti, and by many more authors since.

\subsection{Why Qualitative?}

Aside from inherent interest of codifying and systematizing these putatively more basic or fundamental types of judgments, a number of other purported advantages have been adduced for qualitative formulations of probability. For theorists interested in the  \emph{measurement} of psychological states, it has been argued that comparative judgments provide a more sound empirical basis for elicitation than explicit numerical judgments. To quote \cite{Suppes1994}, `The intuitive idea of using a comparative qualitative relation is that individuals can realistically be expected to make such judgments in a direct way, as they cannot when the comparison is required to be quantitative' (p. 18).  Comparative judgments also appear to be more reliable and more stable over time, for example, compared to point estimates. Even when the elicited comparative judgments may be represented by a numerical function, such a numerical representation will be `a matter of convention' chosen for reasons of `computational convenience' \cite[p. 12]{Krantz1971}.

A second rationale for considering qualitative formulations is that they may be seen as more fundamental mathematically, in part by virtue of their flexibility. \cite{Narens1980} suggests, `The qualitative approach provides a powerful method for the scrutinization and revelation of underlying assumptions of probability theory, is a link to empirical probabilistic concerns, and is a point of departure for the formulation of alternative probabilistic concepts' (p. 143). Indeed, it is easy to construct orderings on events that satisfy some of the same principles as standard numerical probability while violating others. For instance, possibility theory violates finite additivity \citep{Dubois1988}, imprecise probabilities violate comparability \citep{Walley},  and so on. In the infinite case, qualitative orders may accommodate intuitions that elude any reasonable quantitative model (see, e.g., \citealt{diBella}).  For those who judge the Kolmogorov axiomatization to be appropriate only for a limited range of applications, this  generality and flexibility offers a theoretical advantage.

Finally, a third rationale is that qualitative systems may be in some respect simpler, a vague sentiment expressed in nearly all works on the topic. There is indeed something intuitive in the idea that reasoning about a (not necessarily total) order on a space of events is easier than reasoning about measurable functions to the real unit interval. 

\subsection{Probing the Distinction}

Whereas a distinction between quantitative and qualitative probabilistic reasoning seems to be ubiquitous, it is not entirely clear what the distinction is exactly.  In the present work we adopt a logical approach to this and related questions. By advancing our technical understanding of a large space of  probabilistic representation languages, we aim to clarify meaningful ways this distinction might be drawn, and more generally to elucidate further the relationships between specifically probabilistic and broader numerical reasoning patterns.

Uncontroversial exemplars of both qualitative and quantitative systems can be identified. At one extreme, a paradigmatically qualitative language is that of comparative probability. This language, which we call $\mathcal{L}_{\text{comp}}$,\footnote{Such a language was formulated in explicit logical terms first by \cite{Segerberg1971} and by \cite{Gardenfors1975}, but can be traced back at least to \cite{Finetti1937}.} involves only basic comparisons $\mathbf{P}(\alpha) \succsim \mathbf{P}(\beta)$ with the intuitive meaning, `$\alpha$ is at least as likely as $\beta$'. We take as a paradigmatically quantitative language one that allows comparison of arbitrary polynomials (sums and products) over probability terms, a language we call $\mathcal{L}_{\text{poly}}$.\footnote{Such a language appeared first in \cite{Scott1966AssigningPT}, and then later in \cite{fagin1990logic}.} While $\mathcal{L}_{\text{comp}}$ is uncontroversially qualitative and (over finite spaces) vastly underdetermines numerical content, $\mathcal{L}_{\text{poly}}$ is manifestly quantitative and is capable of describing probability measures at a relatively fine level of grain. In between these two extremes is a large space of probabilistic representation languages, essentially differing in how much numerical content they can encode. What are the natural classifications of this space? 

We identify one particularly robust classification based on the amount of arithmetic a system (implicitly) encodes, namely the simple distinction between additive and (additive)-multiplicative systems. Thanks to the Boolean structure of events, even $\mathcal{L}_{\text{comp}}$ can already codify substantial additive reasoning. One can also add explicit addition to this language (e.g., as in \citealt{fagin1990logic}). However, much of probabilistic reasoning appeals to notions of (in)dependence and conditionality, which seem to involve not just additive but also \emph{multiplicative} patterns. This would include (purportedly qualitative) conditional comparisons like $(\alpha|\beta) \succsim (\gamma|\delta)$, as studied by \cite{koopman1940axioms}, as well as seemingly even simpler constructs like (qualitative) \emph{confirmation}, whereby `$\beta$ confirms $\alpha$' just in case $(\alpha|\beta)\succ \alpha$ (e.g., \citealt{Carnap}). 

A first hint that this arithmetical classification is meaningful comes from the observation that the additive systems always admit an interpretation in rational numbers, which in turn facilitates a natural alternative interpretation in terms of concatenation on strings. By contrast, even the simplest multiplicative systems can easily force irrational numbers. For instance, if $(A \wedge B)\approx \neg (A \wedge B)$, then for $(A|B) \approx B$ to hold as well, $B$ must have probability $1/\sqrt{2}$. 

\subsection{Overview of Results}

One of our main results is that this distinction between additive and multiplicative systems is matched by a demarcation in computational complexity. The satisfiability problem for additive systems is seen to be  complete for $\mathsf{NP}$-time, thus no harder than the problems of Boolean satisfiability or integer programming. With any modicum of multiplicative reasoning, by contrast, the satisfiability problem becomes complete for the class $\exists\mathbb{R}$, conjectured to be harder than $\mathsf{NP}$. This classification is surprisingly robust, encompassing the most minimal languages encoding qualitative dependence notions, all the way to the largest system we consider, $\mathcal{L}_{\text{poly}}$, allowing arbitrary addition and multiplication.

Within each of these classes---the `purely' additive and the additive-multiplicative---we find a common distinction between systems that only allow comparisons between atomic probability terms and those that admit explicit arithmetical operations over terms. Indeed, on the additive side, $\mathcal{L}_{\text{add}}$ augments $\mathcal{L}_{\text{comp}}$ with the ability to sum probabilities. As just mentioned, this involves no increase in complexity. However, it does lead to a difference in reasoning principles, and in particular axiomatizability. Drawing on work of \cite{vaught1954}, \cite{kraft1959intuitive}, and \cite{fishburn1997failure}, we show that $\mathcal{L}_{\text{comp}}$ is not finitely axiomatizable. In stark contrast, we present a new finite axiomatization  of $\mathcal{L}_{\text{add}}$ that is seen to be simple and intuitive. The work on $\mathcal{L}_{\text{add}}$ is then adapted for $\mathcal{L}_{\text{comp}}$, to give a new completeness argument for the powerful \emph{polarization rule} \citep{kraft1959intuitive}, which has been used to supplant Scott's \citeyearpar{scott1964measurement} infinitary \emph{finite cancellation} schema \citep{burgess2010axiomatizing,Ding2020}. Rather than deriving the finite cancellation axioms and then appealing to Scott's representation result, we show directly how a variable elimination method can be adapted to show completeness. 

From a logical point of view, these results together suggest that disallowing explicit addition might be seen as an artificial restriction: reasoning in $\mathcal{L}_{\text{add}}$ can always be emulated within $\mathcal{L}_{\text{comp}}$,  at the expense of temporarily expanding the space of events. At no cost in computational complexity, $\mathcal{L}_{\text{add}}$ codifies the relevant reasoning principles in simple, intuitive axioms. 

A similar pattern is seen to arise in the multiplicative setting, for systems that fall into the $\exists\mathbb{R}$-complete class.
Here we consider two languages, $\mathcal{L}_{\mathrm{poly}}$ and the conditional comparative language $\mathcal{L}_{\mathrm{cond}}$ alluded to above. The latter involves comparisons of the form $\mathbf{P}(\alpha | \beta) \succsim \mathbf{P}(\gamma | \delta)$, with the interpretation `$\alpha$ is at least as likely given $\beta$ as is $\gamma$ given $\delta$'.
We give a very intuitive finite system for the former language $\mathcal{L}_{\mathrm{poly}}$ by presenting a multiplicative annex to our axiomatization for $\mathcal{L}_{\mathrm{add}}$.
Varying an argument from \cite{ibeling2020probabilistic}, where a polynomial language permitting subtraction (via unary negation) was considered,\footnote{See also \cite{perovic2008probabilistic} in which an (intuitively) even wider polynomial system with explicit rational quantities was strongly axiomatized via an infinitary proof rule.} we show this system complete. The argument turns on a \emph{Positivstellensatz} of semialgebraic geometry.
As for the latter language $\mathcal{L}_{\text{cond}}$, we review pertinent work by \cite{Domotor}.
% As for $\mathcal{L}_{\mathrm{cond}}$, we critically review existing results and obtain an explicit infinite axiomatization relying on a new polarization rule for the multiplication operator. However we leave the question of finite axiomatizability open.
% \textcolor{blue}{[If that is not done:] We critically review existing results on representation theorems pertinent to $\mathcal{L}_{\text{cond}}$, but leave open the task of finding an explicit axiomatization and the question of finite axiomatizability.}

\subsection{Roadmap}

We begin by defining several probabilistic languages, including $\mathcal{L}_{\text{comp}}$, $\mathcal{L}_{\text{add}}$, $\mathcal{L}_{\text{cond}}$, and $\mathcal{L}_{\text{poly}}$; showing that these languages form an expressivity hierarchy; and introducing the notions from computational complexity used throughout the paper (§\ref{section:space of languages}). We then consider the additive systems $\mathcal{L}_{\text{add}}$ and $\mathcal{L}_{\text{comp}}$, proving the soundness and completeness of axiomatizations of both languages (§§\ref{section: positive linear inequalities}, \ref{section: comparative probability}), discussing issues of finite axiomatizability (§\ref{section: finite axiomatizability}), and rehearsing results that characterize the complexity of reasoning in these systems (§\ref{section: additive complexity}). We turn next to the multiplicative systems, providing an axiomatization of the language $\mathcal{L}_{\text{poly}}$ (§\ref{section: polynomial probability calculus}), investigating axiomatic questions for the language $\mathcal{L}_{\text{cond}}$ (§\ref{section: conditional probability}), and characterizing the complexity of reasoning in these multiplicative systems, including a minimal logic $\mathcal{L}_{\text{ind}}$ allowing only for Boolean combinations of equality and independence statements (§\ref{section: multiplicative complexity}). In §\ref{section: Additive-Multiplicative Divide}, we summarize the results of this discussion: both in the additive and multiplicative settings, systems with explicitly `numerical' operations are more expressive and admit finite axiomatizations, while incurring no cost in complexity; this does not seem to be the case for `purely comparative' systems (as we showed in the additive case, and conjecture for the case of conditional comparative probability). Finally, in §\ref{section: Quantitative-Qualitative Distinction}, we critically discuss various ways of understanding the distinction between qualitative and quantitive probability logics, before concluding in \S\ref{section:conclusion} with some open questions.

\section{A Space of Probabilistic Representation Languages}\label{section:space of languages}

In this section, we define the syntax and semantics of the additive and multiplicative languages which are the primary focus of the paper's discussion, and we illustrate that these languages form an expressivity hierarchy. We also introduce the key notions from computational complexity that are used to characterize the satisfiability problems of these languages.

\subsection{Syntax and Semantics}

Fix a nonempty set of proposition letters $\mathsf{Prop}$, and let $\sigma(\mathsf{Prop})$ be all Boolean combinations over $\mathsf{Prop}$. We will be interested in terms $\mathbf{P}(\alpha)$ for $\alpha \in \sigma(\mathsf{Prop})$, which will be standardly interpreted as the probability of $\alpha$. We first define several sets of probability terms, next define languages of comparisons between terms, and finally provide a semantics for these languages:
\begin{definition}[Terms] Define sets of terms using the following grammars:
\begin{align*}
    \mathsf{a} \in T_{\text{uncond}} &\iff \mathsf{a} :=  \mathbf{P}(\alpha)   & \text{ for any } \alpha \in \sigma(\mathsf{Prop})\\
    \mathsf{a} \in T_{\text{cond}} &\iff \mathsf{a} :=  \mathbf{P}(\alpha |\beta)   & \text{ for any } \alpha,\beta \in \sigma(\mathsf{Prop})\\
    \mathsf{a} \in T_{\text{add}} &\iff \mathsf{a} :=  \mathbf{P}(\alpha) \; | \; \mathsf{a} + \mathsf{b}  & \text{ for any } \alpha \in \sigma(\mathsf{Prop})\\
    \mathsf{a} \in T_{\text{quad}} &\iff \mathsf{a} :=  \mathbf{P}(\alpha)\cdot \mathbf{P}(\beta)  & \text{ for any } \alpha,\beta \in \sigma(\mathsf{Prop})\\
    % \mathsf{a} \in T_{\text{rational}} &\iff  \mathsf{a} :=  r\cdot\mathbf{P}(\alpha) \; | \; \mathsf{a} + \mathsf{b}  & \text{ for any } r \in \mathbb{Q}, \alpha \in \sigma(\mathsf{Prop})\\
    \mathsf{a} \in T_{\text{poly}} &\iff \mathsf{a} :=  \mathbf{P}(\alpha) \; | \; \mathsf{a} + \mathsf{b} \; | \; \mathsf{a} \cdot \mathsf{b} & \text{ for any } \alpha \in \sigma(\mathsf{Prop}).
\end{align*}
\end{definition}

\begin{definition} Define an operator $\Lambda$ that, for each set of terms $T$, generates a language of comparisons in those terms:
\begin{align*}
    \varphi \in \Lambda(T) \iff \varphi = \mathsf{a} \succsim \mathsf{b} \,|\, \neg \varphi \,|\, \varphi \land \psi \,\,\, \text{for }\textsf{a}, \textsf{b} \in T.
\end{align*}
\end{definition}

\begin{definition}[Additive languages]
 Define $\mathcal{L}_{\text{comp}} = \Lambda(T_{\text{uncond}})$ and $\mathcal{L}_{\text{add}} = \Lambda(T_{\text{add}})$. %, and $\mathcal{L}_{\text{rational}} = \Lambda(T_{\text{rational}})$.
\end{definition}

\begin{definition}[Multiplicative languages]\label{def:multlang}

 Define $\mathcal{L}_{\text{cond}} =~ \Lambda(T_{\text{cond}})$, $\mathcal{L}_{\text{quad}} = \Lambda(T_{\text{quad}})$, and $\mathcal{L}_{\text{poly}} = \Lambda(T_{\text{poly}})$. Define two further languages:
\begin{align*}
    \varphi \in \mathcal{L}_{\text{ind}} \iff \varphi =\,& \mathsf{a} = \mathsf{b} \, | \, \mathsf{a} \indep \mathsf{b} \, |\, \neg \varphi \, | \, \varphi \land \psi & \text{ for any } \mathsf{a}, \mathsf{b} \in T_{\text{uncond}}\\
    \varphi \in \mathcal{L}_{\text{confirm}} \iff \varphi =\,& \mathbf{P}(\alpha|\beta) \succsim  \mathbf{P}(\alpha) \, | \, \mathbf{P}(\alpha) \succsim  \mathbf{P}(\alpha|\beta) \, | \,\\ &\mathbf{P}(\alpha) = \mathbf{P}(\beta)\, | \, \neg \varphi \, | \, \varphi \land \psi & \text{ for any } \alpha,\beta \in \sigma(\mathsf{Prop})
\end{align*}
\end{definition}

We assume that $\mathsf{0}$ is an abbreviation for $\mathbf{P}(\bot)$, where $\bot$ is any Boolean contradiction, and likewise $\mathsf{1}$ is an abbreviation for $\mathbf{P}(\top)$. We let $\mathsf{t} \approx \mathsf{t'}$ abbreviate $\mathsf{t} \succsim \mathsf{t'} \wedge \mathsf{t'} \succsim \mathsf{t}$; and let  $\mathsf{t} \succ \mathsf{t'}$ abbreviate $\mathsf{t} \succsim \mathsf{t'} \wedge \neg \mathsf{t'} \succsim \mathsf{t}$.

Some generally valid axioms and rules (call these principles `core' probability logic) appear in Fig. \ref{fig-base}. 
Note that reflexivity of $\succsim$ and non-negativity, $\mathbf{P}(\alpha)\succsim\mathsf{0}$, both follow from $\mathsf{Dist}$.   
These axioms and rules will be part of every system we study. Call them $\mathsf{AX}_{\textnormal{base}}$.
\begin{figure}[t]
\begin{mdframed} %\vspace{-.2in} 
\[\underline{\textbf{The Axioms and Rules of }\mathsf{AX}_{\textnormal{base}}}\] \vspace{-.3in} 
\begin{eqnarray*}
\mathsf{Lin}. & & \mbox{Transitivity and comparability of }\succsim  \\
\mathsf{Bool}. & & \mbox{Boolean reasoning} \\
\mathsf{Dist}. & & \mathbf{P}(\alpha)\succsim \mathbf{P}(\beta)\mbox{ whenever }\models \beta \rightarrow \alpha \\
\mathsf{NonDeg}. & & \neg \mathbf{P}(\bot) \succsim \mathbf{P}(\top) 
\end{eqnarray*}
\end{mdframed}
\caption{Core axioms $\mathsf{AX}_{\textnormal{base}}$ common to all systems.\label{fig-base}}
\end{figure}

\begin{definition}[Semantics]
A model is a probability space $\mathfrak{M} = (\Omega, \mathcal{F}, \mathbb{P},\semantics{\cdot})$, such that $\semantics{\cdot}: \sigma(\mathsf{Prop}) \rightarrow \wp{(\Omega)}$, with  $\semantics{A} \in \mathcal{F}$ for each $A \in \mathsf{Prop}$. It follows that $\semantics{\alpha} \in \mathcal{F}$ for all $\alpha \in \sigma(\mathsf{Prop})$. 

The denotation $\mathbf{P}(\alpha)^{\mathfrak{M}}$ of a basic probability term $\mathbf{P}(\alpha)$ is defined to be $\mathbb{P}(\semantics{\alpha})$, while the donations of complex terms $\mathsf{a} + \mathsf{b}$ and $\mathsf{a} \cdot \mathsf{b}$ are defined in the usual recursive manner. 

We define truth of basic inequality statements in the obvious way:
\begin{eqnarray*}
\mathfrak{M}\models  \mathsf{a} \succsim \mathsf{b}  & \quad \iff  \quad &
\mathsf{a}^{\mathfrak{M}} \geq \mathsf{b}^{\mathfrak{M}},
\end{eqnarray*} while Boolean combinations are evaluated as usual. In $\mathcal{L}_{\text{cond}}$, $\mathcal{L}_{\text{ind}}$, and $\mathcal{L}_{\text{confirm}}$, we specify their atomic clauses separately:
\[\begin{tabular}{l c l}
$\mathfrak{M}\models  \mathbf{P}(\alpha | \beta)\succsim \mathbf{P}(\beta | \delta)$   & \quad $\iff$  \quad \quad  & $\mathbb{P}(\semantics{\alpha \land \beta}) \mathbb{P} (\semantics{\delta}) \geq \mathbb{P}(\semantics{\beta \land \delta})  \mathbb{P}(\semantics{\beta})$ \\
& & \\
$\mathfrak{M}\models  \mathbf{P}(\alpha) \indep \mathbf{P}(\beta)$  & \quad $\iff$  \quad\quad  &  
$\mathbb{P}(\semantics{\alpha \wedge \beta}) = \mathbb{P}(\semantics{\alpha})\mathbb{P}(\semantics{\beta})$ \\
&  & \\ 
$\mathfrak{M}\models  \mathbf{P}(\alpha|\beta) \succsim  \mathbf{P}(\alpha)$  & \quad $\iff$  \quad\quad  &  $\mathbb{P}(\semantics{\alpha \wedge \beta}) \geq \mathbb{P}(\semantics{\alpha})\mathbb{P}(\semantics{\beta})$\\
&  & \\ 
$\mathfrak{M}\models  \mathbf{P}(\alpha) \succsim  \mathbf{P}(\alpha|\beta)$  & \quad $\iff$  \quad\quad  &  $\mathbb{P}(\semantics{\alpha \wedge \beta}) \leq \mathbb{P}(\semantics{\alpha})\mathbb{P}(\semantics{\beta})$.
\end{tabular}\]
Equality statements and rational terms are evaluated as expected. 
\end{definition}

\subsection{An Expressivity Hierarchy}

\begin{comment}
The languages introduced in the preceding section form an expressive hierarchy. Indeed, $\mathcal{L}^{\mathrm{comp}}$ is less expressive than both $\mathcal{L}^{\mathrm{add}}$ and $\mathcal{L}^{\mathrm{cond}}$, both of which are less expressive than the language $\mathcal{L}^{\mathrm{poly}}$. 

When we say that one language is less expressive than another, we mean that no statement in the less expressive language distinguishes two models which can be distinguished by some statement in the more expressive language.

Drawing arrows from less expressive languages to more expressive ones, the hierarchy can be shown graphically:

\begin{center}
    \begin{tikzpicture}
    \node[vertex] (a) at (0,0) {$\mathcal{L}_{\mathrm{comp}}$};
    \node[vertex] (b1) at (2,1) {$\mathcal{L}_{\mathrm{ind}}$};
    \node[vertex] (b2) at (4,1) {$\mathcal{L}_{\mathrm{confirm}}$};
    \node[vertex] (b3) at (6,1) {$\mathcal{L}_{\mathrm{cond}}$};
    \node[vertex] (b4) at (8,1) {$\mathcal{L}_{\mathrm{quad}}$};
    \node[vertex] (c1) at (5,-1) {$\mathcal{L}_{\mathrm{lin}}$};
    \node[vertex] (d) at (10,0) {$\mathcal{L}_{\mathrm{poly}}$};

    \draw[edge] (a) -- (b1);
    \draw[edge] (a) -- (c1);
    \draw[edge] (b1) -- (b2);
    \draw[edge] (b2) -- (b3);
    \draw[edge] (b3) -- (b4);
    \draw[edge] (b4) -- (d);
    \draw[edge] (c1) -- (d);
    \end{tikzpicture}
\end{center}
\end{comment}

The languages introduced in the preceding section form an expressivity hierarchy. For a formula $\varphi$ in any of these languages, let Mod$(\varphi) = \{\mathfrak{M}: \mathfrak{M} \models \varphi\}$ be the class of its models. For two languages $\mathcal{L}_{1}$ and $\mathcal{L}_{2}$, we say that $\mathcal{L}_{2}$ is \emph{at least as expressive as} $\mathcal{L}_{1}$ if for every $\varphi \in \mathcal{L}_1$ there is some $\psi \in \mathcal{L}_2$ such that Mod$(\varphi) = $ Mod$(\psi)$. We say $\mathcal{L}_{2}$ is \textit{strictly more expressive} than $\mathcal{L}_{1}$ if $\mathcal{L}_{2}$ is at least as expressive as $\mathcal{L}_{1}$ but not vice versa. Two languages are \textit{incomparable in expressivity} if neither is at least as expressive as the other. Figure \ref{EXP} illustrates the expressivity hierarchy among the languages introduced above. In particular, $\mathcal{L}_{\mathrm{comp}}$ is less expressive than both $\mathcal{L}_{\mathrm{add}}$ and $\mathcal{L}_{\mathrm{cond}}$, both of which are less expressive than the language $\mathcal{L}_{\mathrm{poly}}$. 

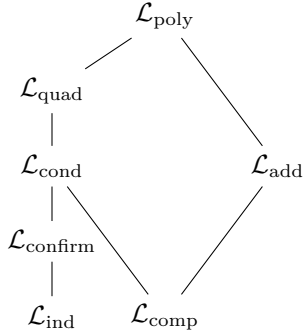
\begin{figure}
\begin{center}
    \begin{tikzpicture}
    \node (a) at (0,0) {$\mathcal{L}_{\text{comp}}$};
    \node (b1) at (-1.5,0){$\mathcal{L}_{\text{ind}}$};
       \node (b2) at (-1.5,1) {$\mathcal{L}_{\text{confirm}}$};
      \node (b3) at (-1.5,2) {$\mathcal{L}_{\text{cond}}$ };
        \node (b4) at (-1.5,3) {$\mathcal{L}_{\text{quad}}$ };
    \node (c1) at (1.5,2) {$\mathcal{L}_{\text{add}}$};
   \node (d) at (0,4) {$\mathcal{L}_{\text{poly}}$};

     \draw (a) -- (b3);
    \draw (a) -- (c1);
     \draw (b1) -- (b2);
    \draw (b2) -- (b3);
       \draw (b3) -- (b4);
       \draw (b4) -- (d);
    \draw (c1) -- (d);
    \end{tikzpicture}
    \end{center}
    \caption{The expressivity hierarchy.}\label{EXP}
    \end{figure}

In this section, we provide examples to establish the above-pictured hierarchy, with each segment between a pair of languages indicating a \textit{strict} increase in expressivity. In particular, we will show that $\mathcal{L}_{\mathrm{comp}}$ is strictly less expressive than both $\mathcal{L}_{\mathrm{add}}$ and $\mathcal{L}_{\mathrm{cond}}$, and that both of these are strictly less expressive than the language $\mathcal{L}_{\mathrm{poly}}$. In fact, in all but one case we show something stronger: namely that (fixing $\Omega$,  $\mathcal{F}$, and $\semantics{\cdot}$) there are two measures $\prob_1$ and $\prob_2$ which are indistinguishable in the less expressive language but which can be distinguished by some statement in the more expressive one. This stronger result is not possible in the case of $\mathcal{L}_{\text{add}}$ and $\mathcal{L}_{\text{poly}}$, because of the following: 
\begin{fact}
Any distinct measures $\mathbb{P}_1, \mathbb{P}_2$ are distinguishable in $\mathcal{L}_{\text{add}}$.
\end{fact}
\begin{proof}
If $\mathbb{P}_1, \mathbb{P}_2$ are distinct measures, without loss of generality, $\mathbb{P}_1(\alpha) < n/m < \mathbb{P}_2(\alpha)$ for some natural numbers $n$ and $m$. Thus $\mathbb{P}_1(\alpha)$ added to itself $m$ times is at most $\mathbb{P}_1(\top)$ added to itself $n$ times, while this is not true of $\mathbb{P}_2$; thus $\mathbb{P}_1, \mathbb{P}_2$ are distinguishable in $\mathcal{L}_{\text{add}}$.
\end{proof}

%[compactness...]

\paragraph{Additive systems.} First, we observe that $\mathcal{L}_{\mathrm{comp}}$ is less expressive than $\mathcal{L}_{\mathrm{add}}$. Let $\mathbb{P}_1(A) = \nicefrac{2}{3}$ and $\mathbb{P}_2(A) = \nicefrac{3}{5}$. The order on the events $A, \neg A, \top, \bot$ induced by these measures is the same, but, for instance, $\mathbb{P}_1(A) = \mathbb{P}_1(\neg A)+\mathbb{P}_1(\neg A)$, while $\mathbb{P}_2(A) \neq \mathbb{P}_2(\neg A)+\mathbb{P}_2(\neg A)$. 

To show that $\mathcal{L}_{\text{add}}$ is less expressive than $\mathcal{L}_{\text{poly}}$, we simply identify a formula $\varphi \in \mathcal{L}_{\text{poly}}$ such that there is no $\psi \in \mathcal{L}_{\text{add}}$ with Mod$(\varphi) = $ Mod$(\psi)$. For this we can take the example mentioned in the introduction: $\mathbb{P}(A \wedge B)=\mathbb{P}(\neg A \vee \neg B)$ and $\mathbb{P}(A|B) = \mathbb{P}(B)$. (This is in fact expressible already in $\mathcal{L}_{\text{cond}}$.) As mentioned above, this enforces that $\mathbb{P}(B)= 1/\sqrt{2}$, while it follows from Corollary~\ref{cor:gen} below that every formula in $\mathcal{L}_{\text{add}}$ has models in which every probability is rational.

\paragraph{Multiplicative systems.} First, we show that $\mathcal{L}_{\text{comp}}$ is no more expressive than $\mathcal{L}_{\text{ind}}$. Define the measures
\begin{align*}
    \prob_1(A \land B) = \nicefrac{25}{36},\; \prob_1(A \land \neg B) = \prob_1(\neg A \land B) = \nicefrac{5}{36},\; \prob_1(\neg A \land \neg B) = \nicefrac{1}{36};\\
    \prob_2(A \land B) = \nicefrac{27}{36},\; \prob_2(A \land \neg B) = \prob_2(\neg A \land B) = \nicefrac{4}{36},\; \prob_2(\neg A \land \neg B) = \nicefrac{1}{36}.
\end{align*}
Then $\prob_1(A \land B) = \prob_1(A) \prob_1(B)$, while $\prob_2(A \land B) \neq \prob_2(A) \prob_2(B)$, so that the measures are distinguishable in $\mathcal{L}_{\text{ind}}$. However, for $i \in \{1,2\}$ we have
\begin{multline*}
    \prob_i(A \lor B) > \prob_i(A)
    = \prob_i(B) > \prob_i(A \land B) > \prob_i(\neg A \lor \neg B) > \\ \prob_i(A \land \neg B)= \prob_i(\neg A \land B) > \prob_i(\neg A \land \neg B),
\end{multline*}
so that the measures $\prob_1$ and $\prob_2$ are not distinguishable in $\mathcal{L}_{\text{comp}}.$

Next, we show that $\mathcal{L}_{\text{ind}}$ is less expressive than $\mathcal{L}_{\text{confirm}}$. Define the measures
\begin{align*}
    \prob_1(A \land B) = \nicefrac{23}{36},\; \prob_1(A \land \neg B) = \prob_1(\neg A \land B) = \nicefrac{6}{36},\; \prob_1(\neg A \land \neg B) = \nicefrac{1}{36};\\
    \prob_2(A \land B) = \nicefrac{27}{36},\; \prob_2(A \land \neg B) = \prob_2(\neg A \land B) = \nicefrac{4}{36},\; \prob_2(\neg A \land \neg B) = \nicefrac{1}{36}.
\end{align*}
The above measures satisfy the same order satisfied by the measures in the preceding example, and $A$ and $B$ are not independent under either measure, so the measures are indistinguishable in $\mathcal{L}_{\text{ind}}$. However, $\mathbb{P}_1(A |B ) < \mathbb{P}_1(A)$, while $\mathbb{P}_2(A |B ) > \mathbb{P}_2(A)$, so that the measures are distinguishable in $\mathcal{L}_{\text{confirm}}$.

Then, note that $\mathcal{L}_{\text{confirm}}$ and $\mathcal{L}_{\text{comp}}$ are incomparable in expressivity. 
The following are $\mathcal{L}_{\text{confirm}}$-equivalent:
\begin{align*}
    \prob_1(A \land B) = \nicefrac{1}{9},\; \prob_1(A \land \neg B) = \nicefrac{1}{3},\; \prob_1(\neg A \land  B) = \nicefrac{5}{9},\; \prob_1(\neg A \land \neg B) = 0;\\
    \prob_2(A \land B) = \nicefrac{1}{9},\; \prob_2(A \land \neg B) = \nicefrac{5}{9},\; \prob_2(\neg A \land  B) = \nicefrac{1}{3},\; \prob_2(\neg A \land \neg B) = 0.
\end{align*}
These measures are, however, evidently distinguishable in $\mathcal{L}_{\text{comp}}$. A fortiori, they are distinguishable in $\mathcal{L}_{\text{cond}}$. Thus this example also shows that $\mathcal{L}_{\text{confirm}}$ is strictly less expressive than $\mathcal{L}_{\text{cond}}$.

%Next, we show that $\mathcal{L}_{\text{confirm}}$ is less expressive than $\mathcal{L}_{\text{cond}}$. The former language can express that events are unconditionally dependent, but this under-determines conditional dependence relations between events. 

After that, we show that $\mathcal{L}_{\mathrm{cond}}$ is less expressive than $\mathcal{L}_{\mathrm{quad}}$. Let $\alpha,\beta,\gamma$ be mutually unsatisfiable events. Define $\prob_1(\alpha)= \nicefrac{3}{20}, \prob_1(\beta)= \nicefrac{4}{20}, \prob_1(\gamma)=\nicefrac{13}{20}$, while $\prob_2(\alpha)=\nicefrac{3}{20} - .03, \prob_2(\beta)=\nicefrac{4}{20} - .01, \prob_2(\gamma)=\nicefrac{13}{20} + .04$. One can verify by exhaustion that all comparisons of conditional probabilities agree between $\prob_1$ and $\prob_2$, thus they are indistinguishable in $\mathcal{L}_{\mathrm{cond}}$. At the same time, there are statements in $\mathcal{L}_{\mathrm{quad}}$ in which the models differ. For example, $\prob_1(\gamma)\prob_1(\beta) <  \prob_1(\alpha)$, whereas $\prob_2(\gamma)\prob_2(\beta)  > \prob_2(\alpha)$.

Finally, we show that $\mathcal{L}_{\mathrm{quad}}$ is less expressive than $\mathcal{L}_{\mathrm{poly}}$. Defining $\mathbb{P}_1(A) = \nicefrac{2}{3}$ and $\mathbb{P}_2(A) = \nicefrac{3}{4}$ we find that for $i \in \{1,2\}$
\[
\mathbb{P}_i(A) > \mathbb{P}_i(A)^2 > \mathbb{P}_i(\neg A) > \mathbb{P}_i(A) \cdot \mathbb{P}_i(\neg A) > \mathbb{P}_i(\neg A)^2,
\]
while $\mathbb{P}_1 (A)^3 < \mathbb{P}_1(\neg A)$ and $\mathbb{P}_2 (A)^3 > \mathbb{P}_2(\neg A)$, so that the measures are not distinguishable in $\mathcal{L}_{\mathrm{quad}}$ but are distinguishable in $\mathcal{L}_{\mathrm{poly}}$.

% \paragraph{Incomparability.} Considering again the example which demonstrated that $\mathcal{L}_{\mathrm{cond}}$ is less expressive than $\mathcal{L}_{\text{quad}}$, we observe that $\prob_1, \prob_2$ can be distinguished in $\mathcal{L}_{\mathrm{lin}}$: $\prob_i(q) \geq 0.2$ for $i = 1$ but not for $i=2$, and this statement is equivalent to the statement in $\mathcal{L}_{\mathrm{lin}}$ that
% \[
% \underbrace{\prob_i(q) + .. + \prob_i(q)}_{10 \text{ times}}  \geq \prob_i(\top) + % \prob_i(\top).
% \]
% This observation, together with the observation that Luce's example is not distinguishable in $\mathcal{L}_{\mathrm{lin}}$, shows that $\mathcal{L}_{\mathrm{lin}}$ and $\mathcal{L}_{\mathrm{cond}}$ are incomparable in expressivity.

Summarizing the results of this section: 
\begin{theorem} 
Figure~\ref{EXP} describes an expressivity hierarchy; each language is less expressive than any higher-up language to which it is path-connected and incomparable in expressivity to all other languages. In words, the language $\mathcal{L}_{\text{comp}}$ is less expressive than $\mathcal{L}_{\text{add}}$, which is again less expressive than $\mathcal{L}_{\text{poly}}$. Similarly, the languages $\mathcal{L}_{\text{comp}}$, $\mathcal{L}_{\text{cond}}$, $\mathcal{L}_{\text{quad}}$, and $\mathcal{L}_{\text{poly}}$ form an expressivity hierarchy, with each language in the series less expressive than the one that follows it, as do the languages $\mathcal{L}_{\text{ind}}$, $\mathcal{L}_{\text{confirm}}$, and $\mathcal{L}_{\text{cond}}$. The languages $\mathcal{L}_{\text{ind}}$ and $\mathcal{L}_{\text{confirm}}$ are incomparable in expressivity with the languages $\mathcal{L}_{\text{comp}}$ and $\mathcal{L}_{\text{add}}$.
\end{theorem}

% To do: discuss binary product language, confirmation language, and independence language

\subsection{Complexity}

In this subsection, we introduce the ideas from complexity theory needed to state some of the paper's results. We denote by $\SAT_{\mathcal{L}}$ the satisfiability problem for $\mathcal{L}$.

\begin{definition}[Polynomial-time, deterministic reductions]\label{reductions}
A polynomial-time, deterministic reduction from one decision problem $A$ to another decision problem $B$ is a deterministic Turing machine $M$, such that $a \in A$ if and only if $M(a) \in B$, with $M(a)$ computing in a number of steps polynomial in the length of the binary input $a$. We write $A \leq B$ when there exists such a reduction..
\end{definition}

In particular, when there is a polynomial-time map from $\mathcal{L}_1$ to $\mathcal{L}_2$ which preserves and reflects satisfiability, we write $\SAT_{\mathcal{L}_1} \leq \SAT_{\mathcal{L}_2}$.

\begin{definition}[$\mathsf{NP}$-reductions]
An $\NP$-reduction from $A$ to $B$ is a nondeterministic Turing machine $M$, such that $a \in A$ if and only if at least one of the outputs $M(a)$ is in $B$, with each output $M(a)$ computing in a number of steps polynomial in the length of the binary input $a$.\footnote{Equivalently, one can think of an $\NP$-reduction as a deterministic reduction $M^\prime$, provided with a polynomial-sized ``certificate'' or ``guess,'' which specifies which (if any) of the non-deterministic paths of $M$ will lead to an output $M(a)$ in $B$: then, $M^\prime$ ``verifies'' this path, correctly producing an output $M^\prime(a) \in B$ if and only if $a \in A$.}
\end{definition}

\begin{definition}[Complexity classes]
When each member of a collection $\mathcal{C}$ of decision problems can be reduced via some deterministic, polynomial-time map to a particular decision problem $A$, one says that the problem is \textit{$\mathcal{C}$-hard}; if in addition, $A \in \mathcal{C}$, then $A$ is $\mathcal{C}$-complete. The class $\mathcal{C}$ of decision problems is called a \textit{complexity class}.
\end{definition}

\begin{definition}[Closure under $\NP$-reductions]
    A complexity class $\mathcal{C}$ is closed under $\mathsf{NP}$-reductions if whenever there is an $\NP$-reduction from $A$ to $B \in \mathcal{C}$, then in addition $A \in \mathcal{C}$.
\end{definition}

 We are concerned here with two complexity classes which are closed under $\NP$-reductions \citep{ten2013data}: $\NP$ and $\ETR$. We discuss each in turn.

\paragraph{The Class $\NP$.} The class $\NP$ contains any problem that can be solved by a non-deterministic Turing machine in a number of steps that grows polynomially in the input size. Hundreds of problems are known to be $\NP$-complete, among them Boolean satisfiability and the decision problems associated with several natural graph properties, for example possession of a clique of a given size or possession of a Hamiltonian path. See \cite{ruiz2011survey} for a survey of such problems and their relations.

\paragraph{The Class $\ETR$.} The Existential Theory of the Reals (ETR) contains all true sentences of the form
\[
\text{there exist } x_1,...,x_n \in \mathbb{R} \text{ satisfying }\mathcal{S},
\]
where $\mathcal{S}$ is a system of equalities and inequalities of arbitrary polynomials in the variables $x_1,...,x_n$. For example, one can state in ETR the existence of the golden ratio, which is the only root of the polynomial $f(x) =x^2 - x -1$ greater than one, by `there exists $x >1$ satisfying $f(x) = 0$.' The decision problem of saying whether a given formula $\varphi\in$ ETR is complete (by definition) for the complexity class $\ETR$.

The class $\ETR$ is the real analogue of $\NP$, in two senses. Firstly, the satisfiability problem that is complete for $\ETR$ features real-valued variables, while the satisfiability problems that are complete for $\NP$ typically feature integer- or Boolean-valued variables. Secondly, and more strikingly, \cite{erickson2020smoothing} showed that while $\NP$ is the class of decision problems with answers that can be verified in polynomial time by machines with access to unlimited integer-valued memory, $\ETR$ is the class of decision problems with answers that can be verified in polynomial time by machines with access to unlimited \textit{real-valued} memory.

As with $\NP$, a myriad of problems are known to be $\ETR$-complete. We include some examples that illustrate the diversity of such problems:
\begin{itemize}
    \item
    In graph theory, there is the $\ETR$-complete problem of deciding whether a given graph can be realized by a straight line drawing \citep{schaefer2013realizability}. 
    \item
    In game theory, there is the $\ETR$-complete problem of deciding whether an (at least) three-player game has a Nash equilibrium with no probability exceeding a fixed threshold \citep{bilo2017existential}.
    \item 
    In geometry, there is the $\ETR$-complete `art gallery' problem of finding the smallest number of points from which all points of a given polygon are visible \citep{abrahamsen2018art}.
     \item
     In machine learning, there is the $\ETR$-complete problem of finding weights for a neural network and some training data such that the total error is below a given threshold \citep{abrahamsen2021training}.
\end{itemize}
For discussions of further $\ETR$-complete problems, see  \cite{schaefer2009complexity} and \cite{cardinal2015computational}.\\

%\subsection{Relativization}
%Another important fact about these languages is that they are closed under relativization to (definable) submodels in the following sense. 

 The inclusions $\NP \subseteq \ETR \subseteq \mathsf{PSPACE}$ are known, where $\mathsf{PSPACE}$ is the set of decision problems solvable using polynomial space; it is an open problem whether either inclusion is strict.

\subsection{Notation}\label{section:notation}

We denote probability measures by $\mathbb{P}$ and formal logical symbols for such measures by $\mathbf{P}$. We use $A, B, C$ to denote propositional atoms; Greek minuscule $\alpha, \beta, \gamma, \delta, \epsilon, \zeta$ to denote propositional formulas over such atoms; sans-serif $\mathsf{a}, \mathsf{b}, \mathsf{c}$, etc.\ to denote terms (elements of the various $T_*$) in probabilities of such formulas; and $\varphi, \psi, \chi$ to denote formulas comparing such terms (viz.\ formulas of the $\mathcal{L}_*$). 
At various points in the paper we rely on the following definition:

\begin{definition}
For a set $\mathcal{A} \subset \mathsf{Prop}$ of proposition letters, let \[\Delta_{\mathcal{A}} = \big\{ \bigwedge_{A \in \mathcal{A}} \ell_A : \ell_A \in \{ A, \lnot A \} \text{ for each } A \big\}\] be the set of formulas providing complete state descriptions of $\mathcal{A}$. We simply write $\Delta$ where the set $\mathcal{A}$ is clear from context.
\end{definition}

Often, we will take $\mathcal{A}$ to be the set of proposition letters appearing in a formula $\varphi$, in which case we write $\Delta_\varphi$ instead of $\Delta_{\mathcal{A}}$. For example, if $\varphi$ is the formula $ \mathbf{P}(A \lor B) > \mathbf{P}(A)$, then $\Delta_\varphi = \{A\land B, \neg A \land B, A \land \neg B, \neg A \land \neg B\}$.

\section{Additive Systems}\label{section: Additive Systems}

\subsection{Positive Linear Inequalities}\label{section: positive linear inequalities}

Our first task is to provide an axiomatization of the language $\mathcal{L}_{\text{add}}$. Aside from the basic principles of $\mathsf{AX}_{\text{base}}$ (Fig. \ref{fig-base}), we have the following additivity axiom: \begin{eqnarray*}
\mathsf{Add}. & & \mathbf{P}(\alpha) \approx \mathbf{P}(\alpha \wedge \beta) + \mathbf{P}(\alpha \wedge \neg \beta)
\end{eqnarray*} We also have core axioms for dealing with addition. The system $\mathsf{AX}_{\text{add}}$ is shown in Figure \ref{fig-add}.\footnote{A version of the axiom $\mathsf{2Canc}$ appears in the textbook by  \cite{Krantz1971}, under the name \emph{double cancellation} (p. 250).}
%\begin{eqnarray*}
%\mathsf{Assoc}. & & \mathsf{a} + (\mathsf{b}+\mathsf{c}) \approx (\mathsf{a}+\mathsf{b})+\mathsf{c} \\
%\mathsf{Comm}. & & \mathsf{a}+\mathsf{b} \approx \mathsf{b}+\mathsf{a} \\ 
%\mathsf{Zero}. & & \mathsf{a} + \mathsf{0} \approx \mathsf{a} \\
%\mathsf{Dupl}. & & \mathsf{a}+\mathsf{a} \succsim \mathsf{b}+\mathsf{b} \leftrightarrow \mathsf{a} \succsim \mathsf{b}. 
%\end{eqnarray*} This essentially says that the set of terms with operation $+$ and identity element $\mathbf{P}(\bot)$ form a cancellative, commutative monoid. 
%Finally, the following are all valid because we are restricting the interpretations of terms to \emph{positive} real numbers (no negative-signs in the language):
%We have two more axioms:
%\begin{eqnarray*}
%\mathsf{Comb}. & & (\mathsf{a}+\mathsf{e}\succsim \mathsf{c}+\mathsf{f} \wedge \mathsf{b}+\mathsf{f}\succsim \mathsf{d}+\mathsf{e}) \rightarrow \mathsf{a}+\mathsf{b} \succsim \mathsf{c}+\mathsf{d} \\
%\mathsf{Contr}. & & (\mathsf{a}+\mathsf{b}\succsim\mathsf{c}+\mathsf{d} \wedge \mathsf{d} \succsim \mathsf{b}) \rightarrow \mathsf{a} \succsim \mathsf{c}
%\mathsf{Mono}. & & \mathsf{a}+\mathsf{b} \succsim \mathsf{a} \\
%\mathsf{Null}. & & \mathsf{a} + \mathsf{b} \approx \mathsf{0} \rightarrow \mathsf{a} \approx \mathsf{0} \\
%\mathsf{Elim}. & & (\mathsf{e}+\mathsf{a}\succ \mathsf{b} \wedge \mathsf{c}\succ \mathsf{e}+\mathsf{d}) \rightarrow \mathsf{a}+\mathsf{c} \succ \mathsf{b}+\mathsf{d}
%\end{eqnarray*} Adding these axioms to $\mathsf{AX}_{\textnormal{base}}$, we call %the resulting system  $\mathsf{AX}_{\text{add}}$. See Fig. \ref{fig-add}.
\begin{figure}[t]
\begin{mdframed} \[\underline{\textbf{The Axioms of }\mathsf{AX}_{\textnormal{add}}}\] \vspace{-.3in} 
\begin{eqnarray*}
\mathsf{Add}. & & \mathbf{P}(\alpha) \approx \mathbf{P}(\alpha \wedge \beta) + \mathbf{P}(\alpha \wedge \neg \beta) \\ 
\mathsf{Assoc}. & & \mathsf{a} + (\mathsf{b}+\mathsf{c}) \approx (\mathsf{a}+\mathsf{b})+\mathsf{c} \\
\mathsf{Comm}. & & \mathsf{a}+\mathsf{b} \approx \mathsf{b}+\mathsf{a} \\ 
\mathsf{Zero}. & & \mathsf{a} + \mathsf{0} \approx \mathsf{a} \\
\mathsf{2Canc}. & & (\mathsf{a}+\mathsf{e}\succsim \mathsf{c}+\mathsf{f} \wedge \mathsf{b}+\mathsf{f}\succsim \mathsf{d}+\mathsf{e}) \rightarrow \mathsf{a}+\mathsf{b} \succsim \mathsf{c}+\mathsf{d} \\
\mathsf{Contr}. & & (\mathsf{a}+\mathsf{b}\succsim\mathsf{c}+\mathsf{d} \wedge \mathsf{d} \succsim \mathsf{b}) \rightarrow \mathsf{a} \succsim \mathsf{c}
\end{eqnarray*}
\end{mdframed}
\caption{$\mathsf{AX}_{\textnormal{add}}$ is comprised of these axioms, in addition to the rules and axioms of $\mathsf{AX}_{\textnormal{base}}$.\label{fig-add}}
\end{figure}
A number of further principles are easily derivable in $\mathsf{AX}_{\text{add}}$, which we record in the following lemma, with suggestive names:
\begin{lemma}
The following all follow in $\mathsf{AX}_{\text{add}}$: \begin{eqnarray*}
\mathsf{NonNull}. & & \mathsf{a} \succsim \mathsf{0} \\
\mathsf{Refl}. & & \mathsf{a} \succsim \mathsf{a} \\
\mathsf{Mono}. & & \mathsf{a}+\mathsf{b} \succsim \mathsf{a} \\
\mathsf{1Canc}. & & \mathsf{a}+\mathsf{c}\succsim\mathsf{b}+\mathsf{c} \leftrightarrow \mathsf{a} \succsim \mathsf{b} \\
\mathsf{Dupl}. & & \mathsf{a}+\mathsf{a} \succsim \mathsf{b}+\mathsf{b} \leftrightarrow \mathsf{a} \succsim \mathsf{b} \\ 
\mathsf{Comb}. & & (\mathsf{a}\succsim\mathsf{b} \wedge \mathsf{c}\succsim\mathsf{d}) \rightarrow \mathsf{a}+\mathsf{c}\succsim\mathsf{b}+\mathsf{d} \\
\mathsf{Sub1}. & & \mathsf{e}+\mathsf{c}\approx\mathsf{a} \rightarrow (\mathsf{a}+\mathsf{d}\succsim\mathsf{b}+\mathsf{c} \leftrightarrow \mathsf{e}+\mathsf{d}\succsim \mathsf{b}) \\
\mathsf{Sub2}. & & \mathsf{e}+\mathsf{c}\approx\mathsf{a} \rightarrow (\mathsf{b}+\mathsf{c}\succsim\mathsf{a}+\mathsf{d} \leftrightarrow \mathsf{b}\succsim\mathsf{e}+ \mathsf{d}) \\
\mathsf{Elim}. & & (\mathsf{e}+\mathsf{a}\succ \mathsf{b} \wedge \mathsf{c}\succ \mathsf{e}+\mathsf{d}) \rightarrow \mathsf{a}+\mathsf{c} \succ \mathsf{b}+\mathsf{d} \\
\mathsf{Repl}. & & \mathsf{a}\approx\mathsf{b} \rightarrow (\varphi \leftrightarrow \varphi^{\mathsf{a}}_{\mathsf{b}})
\end{eqnarray*} where $\varphi^{\mathsf{a}}_{\mathsf{b}}$ is the result of replacing some instances of $\mathsf{a}$ by $\mathsf{b}$.  \label{lemma:add}
\end{lemma}
The main result of this subsection is a completeness proof for $\mathsf{AX}_{\text{add}}$. Unlike existing completeness arguments for additive probability logics (such as that in \citealt{fagin1990logic}), the proof here proceeds solely on the basis of a variable elimination argument.
\begin{theorem} $\mathsf{AX}_{\textnormal{add}}$ is sound and complete. \label{thm:add-complete}
\end{theorem}
\begin{proof} Soundness is routine. For completeness, we show that the system is strong enough to carry out a variation on the Fourier-Motzkin elimination method for solving linear inequalities. %By Proposition \ref{prop:add} it suffices to consider (un)satisfiability over $\mathbb{Q}^+$. 
%Note that for a concrete formula, 

%Given any formula, there will be only finitely many propositional atoms appearing in it. 
By $\mathsf{Dist}$ and $\mathsf{Add}$, we can assume that in every probability term $\mathbf{P}(\delta)$, the formula $\delta$ is a (canonical) complete state description over finitely many propositional atoms, or else a contradiction. Thus we can assume that for every two probability terms $\mathbf{P}(\delta)$ and $\mathbf{P}(\gamma)$ appearing in the formula, $\delta$ and $\gamma$ are logically inconsistent. Their values are therefore constrained only by the restrictions explicitly implied by the formula, and by the fact that their sum must be greater than $0$. By $\mathsf{NonDeg}$ we can assume that our formula is conjoined with the (derivable) statement $\sum_{\delta}\mathbf{P}(\delta) \succ \mathsf{0}$, since the two will be interderivable. In other words, our formula will be satisfiable iff the corresponding linear system---replacing each $\mathbf{P}(\delta)$ with a distinct variable $x$ and $\mathbf{P}(\bot)$ with $0$---has a solution.  % in which at least one of these variables $x$ is positive. %To guarantee the latter property, we can additionally assume our formula is conjoined with the inequality $\sum_{\alpha} \mathbf{P}(\alpha) \succ \mathsf{0}$. The resulting formula will certainly be equisatisfiable with the original formula, by axiom $\mathsf{NonDeg}$.
Note that it suffices to consider (un)satisfiability over solutions in $\mathbb{Q}^+$, the non-negative rationals, as we can always normalize to obtain a solution in $[0,1]$ corresponding to a probability measure. 

Our strategy will be as follows. Suppose $\varphi$ is valid. We want to show that we can transform $\neg \varphi$ into an equisatisfiable formula $\psi$, such that $\vdash_{\mathsf{AX}_{\textnormal{add}}} \neg \varphi \rightarrow \psi$. % and, because $\chi$ will also be unsatisfiable, $\vdash_{\mathsf{AX}_{\textnormal{add}}} \neg\chi$. 
Because the sentence $\psi$ will have a particularly simple form, we will be able to tell easily that its negation is derivable. It will follow at once (by Boolean reasoning) that  $\vdash_{\mathsf{AX}_{\textnormal{add}}} \varphi$. 

Assume that $\neg \varphi$ is in disjunctive normal form, and consider any disjunct, which we can assume is a conjunction of equality statements ($\approx$) and strict inequality statements ($\succ$). Pick any `variable' $x = \mathbf{P}(\delta)$. We want to show how $x$ can be eliminated from each conjunct in a way that leads to an equisatisfiable formula that is also derivable from the previous formula.  By principles $\mathsf{1Canc}$ and $\mathsf{Dupl}$ (together with $\mathsf{Lin}$, $\mathsf{Assoc}$, and $\mathsf{Comm}$) we can assume without loss a fixed $k>0$, such that each conjunct containing $x$ has one of the following forms:
\begin{multicols}{3}
\begin{enumerate}[label=(\roman*)]
\item \label{f1} $kx + \mathsf{a} \approx \mathsf{b}$
\item \label{f2} $kx +  \mathsf{a} \succ \mathsf{b}$
\item \label{f3} $\mathsf{b} \succ kx+ \mathsf{a}$
\end{enumerate} \end{multicols} \noindent where $kx$ is an abbreviation for the $k$-fold sum of $x$, and where $x$ does not appear anywhere in terms $\mathsf{a}$ or $\mathsf{b}$. If there is no $\mathsf{a}$, simply let it be $\mathsf{0}$, admissible by $\mathsf{Zero}$. To show that we can eliminate $x$ altogether, consider the following cases:
\begin{enumerate}[label=(\alph*)]
  \item \emph{There is at least one conjunct of type \ref{f1}.} In this case, principles $\mathsf{Sub1}$ and $\mathsf{Sub2}$ allow eliminating $x$ from all but one conjunct of type \ref{f1}, as well as all conjuncts of types \ref{f2} and \ref{f3}. It remains only to show that $x$ can be eliminated from the last equality $kx+\mathsf{a}\approx\mathsf{b}$. Since $x$ appears nowhere else in the conjunct, the whole formula will be equisatisfiable with the result of replacing  $kx+\mathsf{a}\approx\mathsf{b}$ with $\mathsf{b}\succsim\mathsf{a}$. Moreover, the latter is derivable from the former by transitivity of $\succsim$ and using $\mathsf{Comm}$ and  $\mathsf{Mono}$.
  \item \emph{There are conjuncts of both types \ref{f2} and \ref{f3}.} This case is handled by principle $\mathsf{Elim}$. For each pair $kx +  \mathsf{a} \succ \mathsf{b}$ and $\mathsf{c} \succ kx+ \mathsf{d}$, we include a new conjunct $\mathsf{a}+\mathsf{c} \succ \mathsf{b}+\mathsf{d}$, which does not involve $x$. The resulting formula will be equisatisfiable. 
  \item \emph{There are only conjuncts of type \ref{f2}.} Such a formula is always satisfiable (over the positive rationals), so we can simply replace each of them with any tautology. 
  \item \emph{There are only conjuncts of type \ref{f3}.} In this case the equisatisfiable transformation replaces each instance $\mathsf{b} \succ kx+ \mathsf{a}$ with $\mathsf{b} \succ \mathsf{a}$. The latter can be derived from the former by transitivity, $\mathsf{Comm}$, and $\mathsf{Mono}$. 
\end{enumerate} After the last variable has been eliminated by repeated application of the above rules, we will be left with a conjunction of (in)equalities in which every term is a sum of $\mathsf{0}$s. By $\mathsf{Zero}$, we can assume every conjunct is of the form $\mathsf{0} \succsim \mathsf{0}$ or $\neg \mathsf{0} \succsim \mathsf{0}$. Unsatisfiability implies that $\neg \mathsf{0}\succsim \mathsf{0}$ must be a conjunct. But $\mathsf{0}\succsim \mathsf{0}$ is provable by $\mathsf{Dist}$. Thus, any unsatisfiable formula is refutable in $\mathsf{AX}_{\textnormal{add}}$.
\end{proof}

%\begin{remark} The axiomatization here is rather different from that in \cite{fagin1990logic}. \dots

%Note also that, although $\mathbf{P}(\top)$ is abbreviated as $\mathsf{1}$, its interpretation is in no way fixed to be $1$. Indeed, without a constant symbol for $1$, any model of $\mathcal{L}_{\textnormal{add}}$ can be turned into one in which the denotation of every probability term is a natural number. In other words, $\mathsf{AX}_{\textnormal{add}}$ simultaneously axiomatizes quantifier-free Presburger arithmetic with a $0$-constant, \emph{but no $1$-constant}. Without the ability to enforce \emph{particular numerical constraints}, linear and integer programming collapse. This collapse will of course fail in the multiplicative setting, where $\mathbf{P}(\top)$ also becomes a multiplicative identity. 

%Another (highly qualitative?) interpretation of this system has that each term $\mathbf{P}(\alpha)$ denotes a string in a unary alphabet, while $+$ is interpreted as string concatenation and $\succsim$ is the relation of string containment. Does $\mathsf{AX}_{\textnormal{add}}$ actually axiomatize the quantifier-free theory of free ordered monoids generated by a single element, that is, the theory of structures $(\{a\}^*;\sqsupseteq,\cdot,\epsilon)$, where $\sqsupseteq$ is string containment, $\cdot$ is concatenation and $\epsilon$ is the empty string? 
%\end{remark}

\subsection{Comparative Probability}\label{section: comparative probability}

The pure comparative language $\mathcal{L}_{\text{comp}}$ is just like $\mathcal{L}_{\text{add}}$, but without explicit addition over probability terms. Thus, while we can take $\mathsf{AX}_{\text{base}}$ as a basis for axiomatization, none of the remaining axioms of $\mathsf{AX}_{\text{add}}$ are in the language, most saliently, the additivity axiom $\mathsf{Add}$.
\subsubsection{Quasi-Additivity}\label{subsection: quasi-additivity}
Early in the development of modern probability, \cite{Finetti1937} proposed an intuitive principle, subsequently called \emph{quasi-additivity} (see, e.g., \citealt{Krantz1971}):\footnote{It will be convenient to omit the explicit reference in $\mathcal{L}_{\text{comp}}$ to probability operators, simply writing $\alpha \succsim \beta$ in place of $\mathbf{P}(\alpha)\succsim \mathbf{P}(\beta)$. For $\mathcal{L}_{\text{add}}$ we do not omit it.} 
\begin{eqnarray*}
\mathsf{Quasi}. & & \alpha\succsim \beta \leftrightarrow (\alpha \wedge \neg \beta) \succsim (\beta \wedge \neg \alpha) 
\end{eqnarray*} Some authors also refer to $\mathsf{Quasi}$ as \emph{qualitative additivity}, and it has been argued that this principle constitutes the `hard core for the logic of uncertain reasoning' \cite{Gaifman2009}. 

It was famously shown in \cite{kraft1959intuitive} that $\mathsf{Quasi}$ is insufficient to guarantee a probabilistic representation, falling short of full additivity. However, there is an important sense in which this axiom does capture additivity. First, observe that $\mathsf{Quasi}$ is equivalent to the following variant, where $\delta$ is a Boolean expression incompatible with both $\gamma$ and $\theta$:
\begin{eqnarray*}
\mathsf{Quasi'}. & &   (\gamma \vee \delta) \succsim (\theta \vee \delta) \leftrightarrow \gamma\succsim \theta
\end{eqnarray*} To see this, note that $\mathsf{Quasi}$ emerges as the special case of $\mathsf{Quasi'}$ when $\gamma = (\alpha \wedge \neg \beta)$, $\theta=(\beta \wedge \neg \alpha)$, and $\delta = (\alpha \wedge \beta)$. In the other direction, letting $\alpha = (\gamma \vee \delta)$ and $\beta = (\theta \vee \delta)$, we obtain \begin{eqnarray*} (\gamma \vee \delta)\succsim (\theta \vee \delta) & \leftrightarrow & ((\gamma \vee \delta) \wedge \neg (\theta \vee \delta)) \succsim ((\theta \vee \delta) \wedge \neg( \gamma \vee \delta)) \\
%& \leftrightarrow & ((\gamma \vee \delta) \wedge \neg \theta \wedge \neg \delta) \succsim ((\theta \vee \delta) \wedge \neg \gamma \wedge \neg \delta) \\
& \leftrightarrow &  (\gamma \wedge \neg \theta) \succsim (\theta \wedge  \neg \gamma) \\
& \leftrightarrow &  \gamma \succsim \theta,
\end{eqnarray*} where the first and third equivalences are instances of $\mathsf{Quasi}$ and the second follows from $\mathsf{Dist}$.

The following observations identify a strong respect in which $\mathsf{Quasi}$ truly captures additivity. As long as we are dealing with incompatible Boolean expressions,  $\mathsf{Quasi}$ facilitates all of the same reasoning patterns as the core principles of $\mathsf{AX}_{\text{add}}$:
\begin{lemma} \label{lemma:combq} %Suppose the pairs $\{\alpha,\beta\}$,   $\{\gamma,\delta\}$, $\{\alpha,\epsilon\}$, $\{\beta,\zeta\}$, $\{\gamma,\zeta\}$, and $\{\delta,\epsilon\}$ 
Suppose $\alpha$, $\beta$, $\gamma$, $\delta$, $\varepsilon$, and $\zeta$ are all pairwise unsatisfiable. Then following patterns are derivable from $\mathsf{AX}_{\text{base}}+\mathsf{Quasi}$:\begin{eqnarray*}
\mathsf{2CancQ}. & & \big((\mathsf{\alpha\vee\varepsilon})\succsim (\mathsf{\gamma \vee \zeta}) \wedge (\mathsf{\beta\vee\zeta})\succsim (\mathsf{\delta\vee\varepsilon})\big) \rightarrow (\mathsf{\alpha\vee\beta}) \succsim (\mathsf{\gamma\vee\delta}) \\
\mathsf{ContrQ}. & & \big((\alpha\vee\beta)\succsim(\gamma\vee\delta) \wedge \delta \succsim \beta\big) \rightarrow \alpha \succsim \gamma
\end{eqnarray*}
\end{lemma}
\begin{proof} Consider first $\mathsf{2CancQ}$. If $(\alpha\vee\varepsilon)\succsim (\gamma \vee \zeta)$, then by $\mathsf{Quasi'}$ we have $(\alpha \vee \varepsilon \vee \beta) \succsim (\gamma \vee \zeta \vee \beta)$. Meanwhile, from $(\beta\vee\zeta)\succsim (\delta\vee\varepsilon)$, again using $\mathsf{Quasi'}$ we have $(\beta \vee \zeta \vee \gamma) \succsim (\delta \vee \varepsilon \vee \gamma)$. By $\mathsf{Dist}$ and transitivity of $\succsim$ we obtain $(\alpha \vee \beta \vee \varepsilon) \succsim (\gamma \vee \delta \vee \varepsilon)$, and finally by one last application of $\mathsf{Quasi'}$ we derive $(\alpha \vee \beta) \succsim (\gamma \vee \delta)$.

For $\mathsf{ContrQ}$: if $\delta \succsim \beta$ then by $\mathsf{Quasi'}$, $(\gamma \vee \delta) \succsim (\gamma \vee \beta)$. From $(\alpha \vee \beta) \succsim (\gamma \vee \delta)$ and transitivity we derive $(\alpha \vee \beta) \succsim (\gamma \vee \beta)$. From one more application of $\mathsf{Quasi'}$ we conclude $\alpha \succsim \gamma$. 
\end{proof}
Indeed, the reader is invited to check that all of the patterns recorded in Lemma \ref{lemma:add} are derivable under the same restriction. For instance, $\mathsf{1Canc}$ simply becomes $\mathsf{Quasi'}$. For another example, here is a version of $\mathsf{Dupl}$:
\begin{lemma} \label{lemma:dupq} Suppose $\alpha$, $\beta$, $\gamma$, and $\delta$ are all pairwise unsatisfiable. Then every instance of the following is derivable from $\mathsf{AX}_{\text{base}}+\mathsf{Quasi}$:\begin{eqnarray*}
\mathsf{DuplQ}. & & (\alpha \approx \beta \wedge \gamma \approx \delta) \rightarrow \big((\alpha \vee \beta)\succsim (\gamma \vee \delta) \leftrightarrow \alpha \succsim \gamma\big). 
\end{eqnarray*}
\end{lemma}
\begin{proof} Suppose $\alpha \approx \beta$ and $\gamma \approx \delta$. Then if $\alpha \succsim \gamma$, by several applications of transitivity we know that $\beta \succsim \delta$. Applying $\mathsf{2CancQ}$ from Lemma \ref{lemma:combq} (letting $\varepsilon = \zeta = \top$) we obtain $(\alpha \vee \beta)\succsim (\gamma \vee \delta)$.

In the other direction, suppose $(\alpha \vee \beta) \succsim (\gamma \vee \delta)$ but that $\alpha \succsim \gamma$ fails. Then from our assumption above, using transitivity and Boolean reasoning, $\beta \succsim \delta$ must also fail. By comparability of $\succsim$ it follows that $\delta \succsim \beta$. But then $\mathsf{ContrQ}$ from Lemma \ref{lemma:combq} implies $\alpha \succsim \gamma$, a contradiction. 
\end{proof}
Limitations arise from the case where we cannot simulate addition with disjunction, when Boolean expressions are not incompatible. 

There are two approaches to circumvent the problem, each leading to a different axiomatization of the logic of comparative probability. The first and most canonical way, emerging from classical work in the theory of comparative probability orders, relies on introducing an infinitary axiom scheme which captures precisely the probabilistic representability of a comparative probability order on a Boolean algebra of events. The second one relies on introducing the \emph{Polarization rule}, a powerful proof rule proposed by \cite{burgess2010axiomatizing}. As we show in section \ref{Polarizationsection}, adding the quasi-additivity axiom and the polarisation rule to $\mathsf{AX}_{\textnormal{base}}$ gives an alternative axiomatization for $\mathcal{L}_{\text{comp}}$. 

%Arguably, this is a superficial obstacle, one that can be circumvented by introduction of a powerful proof rule. 

%There have thus been two approaches: employ an infinitary schema explicitly enumerating all of the relevant additive facts, or add a powerful rule that---as we show here---allows simulating explicit addition. 

\subsubsection{Finite Cancellation and probabilistic representability} \label{section:scott}

%The standard way to axiomatize the logic of comparative probability \cite{Segerberg1971, Gardenfors1975} crucially relies on a representation theorem for comparative probability orders, due to \citep{kraft1959intuitive} and \citep{scott1964measurement}. The representation theorem gives necessary and sufficient conditions for a binary relation $\preceq$ on a finite Boolean algebra to be \emph{probabilistically representable}. 

The most common approach to axiomatizing comparative probability crucially relies on a representation theorem for comparative probability orders, due to \cite{kraft1959intuitive} and \cite{scott1964measurement}. Faced with the inadequacy of de Finetti's quasi-additivity principle to guarantee probabilistic representation, these authors proposed an infinite list of axioms, often called the \emph{finite cancellation} axioms. Here we follow subsequent modal-logical formulations of the axioms due to \cite{Segerberg1971,Gardenfors1975}. Given two lists of $n$ Boolean formulas $\alpha_1,\dots,\alpha_n,\beta_1,\dots,\beta_n$, we can consider the set $\Delta$ of all state descriptions, treating these Boolean formulas as atoms. We call a state description $\delta \in \Delta$ \emph{balanced} if the same numbers of $\alpha_i$'s as $\beta_i$'s is (not) negated in $\delta$. Let $\mathcal{B} \subset \Delta$ be the set of all balanced state descriptions. Now define the following abbreviation: \begin{eqnarray}\label{Balancedsequences}
(\alpha_1,\dots,\alpha_n)\equiv_0(\beta_1,\dots,\beta_n) & \overset{\textnormal{def}}{=} & \big(\bigvee_{\delta \in \mathcal{B}}\delta\big)\approx \top.
\end{eqnarray}
For all $n$, and all pairs of sequences of $n$ formulas, we have an instance of $\mathsf{FinCan}_n$: 
\begin{eqnarray*}
\mathsf{FinCan}_n. & & \big( (\alpha_1,\dots,\alpha_n)\equiv_0(\beta_1,\dots,\beta_n) \wedge \alpha_1\succsim\beta_1 \wedge \dots \wedge \alpha_{n-1}\succsim\beta_{n-1}   \big) \rightarrow \beta_n \succsim \alpha_n.
\end{eqnarray*}

Semantically, it is easy to see that the balancedness condition $(\alpha_1,\dots,\alpha_n)\equiv_0(\beta_1,\dots,\beta_n)$ amounts to the property that 
$$\sum_{i\leq n} \mathbbm{1}_{ \llbracket \alpha_{i}\rrbracket} = \sum_{i\leq n} \mathbbm{1}_{\llbracket \beta_{i}\rrbracket},$$
where $\mathbbm{1}_{\llbracket \alpha \rrbracket}$ is the indicator function of the set $\llbracket \alpha \rrbracket $. Accordingly, given a sample space $\Omega$ and a set algebra $\mathcal{F}$, we say that two sequences of events $A_1,..., A_{n}$ and $B_1,..., B_n \in \mathcal{F}$ are \emph{balanced} if $\sum_{i\leq n} \mathbbm{1}_{A_{i}} = \sum_{i\leq n} \mathbbm{1}_{B_{i}}$. A semantic formulation of the finite cancellation axioms is then the following:
\begin{eqnarray*}
\mathsf{FinCan}_n. & & \text{If $(A_{i})_{i\leq n}$ and $(B_{i})_{i\leq n}$ are balanced and $\forall i< n$, $A_{i}\succeq B_{i}$, then $B_{n}\succeq A_{n}$}.
\end{eqnarray*}
The two sequences being balanced means that every element of the sample space belongs to exactly as many $A_i$'s as $B_i$'s. Under this description, the soundness of $\mathsf{FinCan}_n$ is straightforward: balancedness ensures that the sums $\sum_{i} \prob(A_i)$ and $\sum_{i} \prob(B_i)$ are equal, since they are computed by taking the exact same sums of terms $\prob(\omega)$ for $\omega\in\Omega$: this is inconsistent with having $ \prob(A_i)\geq \prob(B_i)$ for all $i$ \emph{with some of these inequalities strict}. (Alternatively, to show soundness we can prove that $\mathsf{FinCan}_n$ is derivable in $\mathsf{AX}_{\text{add}}$: see Appendix.) 

The significance of the finite cancellation scheme stems from the following theorem. Say that a relation $\succeq$ on a Boolean algebra $\mathcal{F}$ is  \emph{probabilistically representable} if there exists a probability measure $\mathbb{P}$ on $\mathcal{F}$ such that, for all $A,B\in\mathcal{F}$, 
$$
A\succeq B \text{ if and only if } \mathbb{P}(A)\geq \mathbb{P}(B).
$$
We have: 
\begin{theorem}[\citealt{kraft1959intuitive,scott1964measurement}]
Let $(\Omega, \mathcal{F})$ a finite set algebra and $\succeq$ a binary relation on $\mathcal{F}$. There is a probability measure $\mathbb{P}$ on $(\Omega, \mathcal{F})$ representing $\succeq$ if and only if the following hold for all $A$, $B \in\mathcal{F}$: \label{SCOTTTHM}
\begin{eqnarray*}
\mathsf{Tot}. & &  \succeq \text{ is a reflexive total order;}\\
\mathsf{NonDeg}. & &  \varnothing\not\succeq\Omega;\\
\mathsf{NonTriv}. & & A\succeq \varnothing;\\
%\mathsf{Quasi}^\prime. & &  \text{If $(A\cup B)\cap C = \varnothing$, then $(A\preceq B \Leftrightarrow A\cup C \preceq B \cup C)$;}\\
\mathsf{FinCan}_n. & & \text{If $(A_{i})_{i\leqslant n}$ and $(B_{i})_{i\leqslant n}$ are balanced and $\forall i< n$, $A_{i}\succeq B_{i}$, then $B_{n}\succeq A_{n}$}.
\end{eqnarray*}
\end{theorem}
The result is proved by appeal to general results in linear algebra or linear programming (see, e.g., \citealt{scott1964measurement} or \citealt{narens07}). It is worth sketching the proof here in order to highlight the algebraic content of the finite cancellation rule, as well as to emphasize the key differences between the additive and multiplicative systems we investigate below: particularly,  the distinct tools and proof techniques involved.

To prove the right-to-left direction of Theorem \ref{SCOTTTHM}, we first formulate the task as an algebraic problem. Consider an order $\succeq$ on events in the algebra $(\Omega,\mathcal{F})$ satisfying the properties listed above. Take the vector space $\mathbb{R}^{n}$. Each event $A$ is identified with the vector $\ind{A}$ of its indicator function: that is, the vector $(v_{1},\dots,v_{n})$ where $v_{i}= \ind{A}(\omega_{i})$. Finding a measure representing $\succeq$ amounts to finding a linear functional $\Phi:\mathbb{R}^{n}\rightarrow \mathbb{R}$ with the property that $\Phi (\ind{A})\geq \Phi(\ind{B})$ iff $A\succeq B$. The map $A\mapsto \Phi(\ind{A})$ can then be seen as a (non-normalised) additive measure on $(\Omega,\mathcal{F})$. Note that the linearity of $\Phi$ ensures that $\Phi(\ind{\varnothing})=\Phi(\mathbf{0})=0$. Further, $\mathsf{NonDeg}$ and $\mathsf{NonTriv}$ ensure that $\Phi(\ind{\Omega}) > 0$ and $\Phi(\ind{A})\geq 0$ for all $A\in\mathcal{F}$. Importantly, additivity is also guaranteed: when $A\cap B\neq \varnothing$, we have $\ind{A\cup B}=\ind{A}+\ind{B}$ in $\mathbb{R}^{n}$, and by the linearity of $\Phi$ we have $\Phi(\ind{A\cup B})= \Phi(\ind{A}+\ind{B}) = \Phi(\ind{A})+\Phi(\ind{B})$. This means that we can define the desired probability measure in the obvious way: set $\prob(A):=\Phi(\ind{A})/\Phi(\ind{\Omega})$.

To find such an order-preserving linear functional, we can appeal to the following Lemma, due to \cite{scott1964measurement}:

\begin{lemma}[Scott's Lemma]\label{replemma}
Let $V$ a finite-dimensional real vector space, and $(M,\succeq)$ a finite relational structure with $M\subseteq V$ and $M$ a set of vectors with coordinates in $\mathbb{Q}$. Then there exists a linear functional $\Phi:V\rightarrow \mathbb{R}$ satisfying $\mb{w}\succeq \mb{v} \Leftrightarrow\Phi(\mb{v})\succeq\Phi(\mb{w})$ if and only if
\begin{itemize}
    \item[(a)] $\forall \mb{w},\mb{v}\in M$, $\mb{v}\succeq \mb{w}$ or $\mb{w}\succeq \mb{v}$
    \item[(b)] if $\sum^{n}_{i=1}\mb{v}_{i} = \sum^{n}_{i=1}\mb{w}_{i}$ and $\forall i<n, \mb{v}_{i}\succeq \mb{w}_{i}$ then $\mb{v}_{n}\preceq \mb{w}_{n}$.
\end{itemize}
\end{lemma}

We obtain the desired functional by applying the Lemma to the structure $(M, \succeq)$, where $M=\{\ind{A}\,|\, A\in\mathcal{F}\}$, and we lift the $\succeq$ relation to $M$ by setting $\ind{A}\succeq \ind{B}$ if and only if $A\succeq B$. In particular, note that property (b) of the Lemma, when applied to vectors of the form $\mb{v}_{i}= \ind{A_i}$ and $\mb{w}_{i} = \ind{B_{i}}$, corresponds precisely to the finite cancellation axiom scheme $\mathsf{FinCan}_n$. In this way, Theorem \ref{SCOTTTHM} is established. 

In order to get a better grasp on the algebraic content the finite cancellation axioms, it is informative to consider the following simple geometric description of the problem. We want a linear functional $\Phi$ to have the property that 
$\Phi(\ind{A}-\ind{B})\geq 0$ if $A\succeq B$, and 
$\Phi(\ind{A}-\ind{B}) < 0$ if $A\not\succeq B$. Each linear functional on $\mathbb{R}^{n}$ can be written in the form $\Phi(\mathbf{v})=\mathbf{w}^{T}\mathbf{v}$ for some vector $\mathbf{w}$. This means that, geometrically, finding a linear functional of the desired kind amounts to finding a hyperplane separating (the cones generated by) the sets $\{\ind{A}-\ind{B}\,|\, A\succeq B\}$ and $\{\ind{A}-\ind{B}\,|\, A\not\succeq B\}$. Given such a hyperplane with normal $\mathbf{w}$, we have that $\mathbf{w}^T(\ind{A}-\ind{B})\geq 0$ for all $A$, $B$ such that $A\succeq B$ (the angle between $\mathbf{w}$ and $(\ind{A}-\ind{B})$ is right or acute) and $\mathbf{w}^T(\ind{A}-\ind{B})< 0$ for all $A$, $B$ such that $B\succ A$ (the angle between $(\ind{A}-\ind{B})$ and $\mb{w}$ is obtuse). We thus need to solve the following system of linear inequalities for $\mb{w}\in\mathbb{R}^n$. 

\begin{align*}
\mb{w}^T (\ind{A}-\ind{B}) &\geq 0  \text{ for all $A,B\in\mathcal{F}$ such that $A\succeq B$}\\
\mb{w}^T (\ind{A}-\ind{B}) &< 0 \text{  for all $A,B\in\mathcal{F}$ such that $B\succ A$ } 
\end{align*}

A well-known theorem of the alternative (\citealt{Motzkin}; see also \citealt{Schrijver}) states that a system like the above fails to have a solution only if there exists an (integer-valued) \emph{certificate of infeasibility}. A certificate of infeasibility for the linear system given here translates into the existence of two balanced sequences of events which violate an instance of $\mathsf{FinCan}_n$. The finite cancellation conditions ensure the nonexistence of certificates of infeasibility for the system of linear inequalities expressed by the order $\succeq$. 

With Theorem \ref{SCOTTTHM} at hand, one can show by standard logical methods that adding the infinite axiom scheme $\mathsf{FinCan}_n$  to $\mathsf{AX}_{\text{base}}$ yields a complete axiomatization of $\mathcal{L}_{\text{comp}}$-validities (together with a simple extensionality axiom). 
\begin{theorem}[\citealt{Segerberg1971,Gardenfors1975}]
$\mathcal{L}_{\text{comp}}$ is completely axiomatized by $\mathsf{AX}_{\text{base}}$ together with the following: 
\begin{eqnarray*}
\mathsf{FinCan}_n. & & \big( (\alpha_1,\dots,\alpha_n)\equiv_0(\beta_1,\dots,\beta_n) \wedge \alpha_1\succsim\beta_1 \wedge \dots \wedge \alpha_{n-1}\succsim\beta_{n-1}   \big) \rightarrow \beta_n \succsim \alpha_n. \\
\mathsf{Ext}. & &  \big((\alpha_{1}\leftrightarrow\alpha_2 \succsim \top) \wedge (\beta_{1}\leftrightarrow\beta_2 \succsim \top)\big) \to \big((\alpha_{1}\succsim \beta_{1}) \to (\alpha_{2}\succsim \beta_{2})\big)
\end{eqnarray*}
\end{theorem}

Three points are worth noting. First, the techniques required for proving the completeness result belong entirely to the standard toolkit of linear algebra. We see how this canonical axiomatization of $\mathcal{L}_{\text{comp}}$ is based on hyperplane separation methods. $\mathcal{L}_{\text{comp}}$ can only express linear constraints on the representing probability measure; the axiom schemes ensure precisely the consistency of the systems of linear inequalities that the language can consistently express. As we will see, this will no longer hold in any of the \emph{multiplicative} systems we will consider, which can express polynomial constraints: there, proving completeness will require showing that the system is powerful enough to prove the consistency of certain systems of polynomial inequalities. Thus the study of multiplicative probability logics involves techniques from semialgebraic geometry. 

Second, in this case the linear functional in the representation theorem can in fact always be taken to be \emph{rational-valued}: in other words, no constraints expressible in $\mathcal{L}_{\text{comp}}$ can force the probability of an event to have an irrational value.  This entails that any consistent formula in $\mathcal{L}_{\text{comp}}$ has a model where the probabilities are all rational. As a consequence, one can show that  $\mathcal{L}_{\text{comp}}$ is sound an complete with respect to finite \emph{counting} models, in which $\alpha\succeq \beta$ holds exactly if more states in the model satisfy $\alpha$ than $\beta$ \citep{Hoek1996b}. This is not in general the case for multiplicative systems: as we saw above, the ability to express polynomial constraints can force \emph{irrational} probabilities for some events.

Lastly, we saw this canonical axiomatization is infinite. In the next section, we will show that one can avoid the infinite cancellation scheme by enriching our base system with a powerful proof rule. In Section \ref{section: finite axiomatizability}, we will show that, without such a strong proof rule, an infinite axiom scheme is unavoidable: within the basic rules of system $\mathsf{AX}_{\text{base}}$, the logic $\mathcal{L}_{\text{comp}}$ is not finitely axiomatizable.

%There have thus been two approaches: employ an infinitary schema explicitly enumerating all of the relevant additive facts, or add a powerful rule that---as we show here---allows simulating explicit addition. 

\subsubsection{Polarization}\label{Polarizationsection}
Intuitively, if we could only `duplicate' formulas whenever we want to add probabilities for overlapping events, this would license the same reasoning capacities as with linear inequalities. Such a proof rule was introduced by \cite{burgess2010axiomatizing}, following \cite{kraft1959intuitive}. Suppose $A$ is a proposition letter that occurs nowhere in $\alpha$ or $\varphi$. Then the polarization rule says: 
\begin{eqnarray*}
\mathsf{Polarize}. & & \mbox{ From }(\alpha \wedge A)\approx(\alpha \wedge \neg A) \rightarrow \varphi\mbox{ infer }\varphi. 
\end{eqnarray*} The soundness of $\mathsf{Polarize}$ is straightforward to show (see \citealt{Ding2020}): if $\neg \varphi$ is satisfiable, it suffices to show that $\neg \varphi$ can be satisfied together with $(\alpha \wedge A) \approx (\alpha \wedge \neg A)$. This is achieved by duplicating $\semantics{\alpha}$, the extension of $\alpha$, and ensuring that all $A$-free formulas are thereby preserved.

Completeness, however, is less straightforward. Existing treatments show how the infinitary schema $\mathsf{FinCan}_n$ can be derived from $\mathsf{Polarize}$ (cf. \citealt{burgess2010axiomatizing,Ding2020}); as we saw in \S\ref{section:scott}, completeness of the infinitary system in turn depends on additional facts from linear algebra. Here we can give a more direct argument, showing exactly how polarization, together with de Finetti's quasi-additivity axiom, recapitulates the additive reasoning for variable elimination that we gave for $\mathsf{AX}_{\text{add}}$ (Theorem \ref{thm:add-complete}). 
Let $\mathsf{AX}_{\text{comp}}$ consist of the axioms and rules of $\mathsf{AX}_{\text{base}}$, plus $\mathsf{Quasi}$ and $\mathsf{Polarize}$ (Fig. \ref{fig-comp}).
\begin{figure}[t]
\begin{mdframed} \[\underline{\textbf{The Axioms and Rules of }\mathsf{AX}_{\textnormal{comp}}}\] \vspace{-.3in} 
\begin{eqnarray*}
\mathsf{Quasi}. & & \alpha\succsim \beta \leftrightarrow (\alpha \wedge \neg \beta) \succsim (\beta \wedge \neg \alpha)  \\
\mathsf{Polarize}. & & \mbox{From }(\alpha \wedge A)\approx(\alpha \wedge \neg A) \rightarrow \varphi\mbox{ infer }\varphi
\end{eqnarray*}
\end{mdframed}
\caption{$\mathsf{AX}_{\textnormal{comp}} \overset{\textnormal{def}}{=} \mathsf{AX}_{\textnormal{base}} + \mathsf{Quasi} + \mathsf{Polarize}$. \label{fig-comp}}
\end{figure}
 
Consider some finite set $\mathcal{A}$ of propositional atoms and suppose $A \notin \mathcal{A}$. For a formula $\varphi$ over $\mathcal{A}$, define the \emph{relativization} $\varphi^A$ to be the result of replacing every inequality $\varepsilon \succsim \zeta$ in $\varphi$ with $(\varepsilon \wedge A) \succsim (\zeta \wedge A)$. Let $\pi$ be the formula: \begin{eqnarray*}
\pi & \overset{\textnormal{def}}{=} & \bigwedge_{\delta \in \Delta}(\delta \wedge A) \approx (\delta \wedge \neg A),
\end{eqnarray*} where $\Delta$ is the set of state descriptions over $\mathcal{A}$. Then we have:
 \begin{lemma}[Relativization under Polarization] $\vdash_{\mathsf{AX}_{\textnormal{comp}}} \pi \rightarrow (\varphi \leftrightarrow \varphi^A)$. \label{lemma:rel}
\end{lemma}
\begin{proof} Consider any inequality $\varepsilon \succsim \zeta$ appearing in $\varphi$. The result will follow by Boolean reasoning if we can just show that $\varepsilon \succsim \zeta \leftrightarrow (\varepsilon \wedge A) \succsim (\zeta \wedge A)$ follows from $\pi$.

By $\mathsf{Dist}$ we can assume that $\varepsilon$ and $\zeta$ are disjunctions of state descriptions over $\mathcal{A}$. First observe that $(\varepsilon \wedge A) \approx (\varepsilon \wedge \neg A)$ follows from $\pi$, and the same for $(\zeta \wedge A) \approx (\zeta \wedge \neg A)$, using $\mathsf{Dist}$ and multiple instances of $\mathsf{DuplQ}$ (from Lemma \ref{lemma:dupq}). But then by another application of $\mathsf{DuplQ}$, where $\alpha = (\varepsilon \wedge A)$, $\beta = (\varepsilon \wedge \neg A)$, $\gamma = (\zeta \wedge A)$, and $\delta = (\zeta \wedge \neg A)$, and $\mathsf{Dist}$ once again, we derive $\varepsilon \succsim \zeta \leftrightarrow (\varepsilon \wedge A) \succsim (\zeta \wedge A)$.
\end{proof}
\begin{theorem}  $\mathsf{AX}_{\text{comp}}$ is sound and complete.
\end{theorem}
\begin{proof} The proof strategy is to follow the derivation of $\varphi$ from $\mathsf{AX}_{\text{add}}$, showing how to transform this into a derivation from $\mathsf{AX}_{\text{comp}}$ using polarization. Roughly speaking, the idea is to replace sums of probability terms with probabilities of disjunctions; polarization is used to ensure that the disjuncts can be mutually incompatible, by furnishing a sufficient number of `copies' of each disjunct. As we saw above, quasi-additivity is sufficiently strong when reasoning about incompatible disjuncts.

As in the proof of Theorem \ref{thm:add-complete}, by $\mathsf{Dist}$ we assume that for every inequality $\alpha \succsim \beta$ appearing in $\varphi$, both $\alpha$ and $\beta$ are disjunctions involving the finitely many state-descriptions $\delta \in \Delta$ over the propositional atoms $\mathsf{Prop}^{\varphi}$ that occur in $\varphi$. %, and by $\mathsf{NonDeg}$ we can assume $(\bigvee_{\delta \in \Delta} \delta) \succ \mathsf{0}$ is conjoined with $\varphi$. 

%When $\varphi \in \mathcal{L}_{\text{comp}}$ is valid, so that $\neg\varphi$ is unsatisfiable, we again want to find an equisatisfiable $\chi$ whose negation is derivable, such that $\vdash_{\mathsf{AX}_{\textnormal{comp}}} \neg \varphi \rightarrow \chi$. Suppose that $\neg\varphi$ is in disjunctive normal form, and consider any disjunct, which will again be a conjunction of equality statements ($\approx$) and strict inequalities ($\succ$) between disjunctions of state descriptions.  

Because $\varphi$ is also in $\mathcal{L}_{\text{add}}$, the proof of Theorem \ref{thm:add-complete} furnishes a derivation of $\varphi$. Let $m$ be the maximum factor that appears anywhere in the derivation, that is, the largest number of times any term $\mathbf{P}(\delta)$ is added to itself. We introduce $n=\lceil\mbox{log}_2 m \rceil$ fresh proposition letters $\mathsf{Prop}^+ = \{A_1,\dots,A_n\}$. As the method below will mimic the previous derivation---and in particular will not introduce more factors---$n$ atoms, and thus $2^n \geq m$ state descriptions over those atoms, suffices. Define %$\Delta^+$ to be the state descriptions over $\mathsf{Prop}^{\varphi} \cup \mathsf{Prop}^+$, stratified so that 
%$\Delta^+ = \bigcup_{k\leq n} \Delta_k$, where 
$\Delta_{k}$ to be the state descriptions over $\mathsf{Prop}^{\varphi} \cup \{A_1,\dots,A_k\}$, so in particular, $\Delta_0 = \Delta$, and $\Delta_n$ is the set of state descriptions over $\mathsf{Prop}^{\varphi} \cup \mathsf{Prop}^+$. 

We can now relativize $\varphi$ $n$ times, producing $\varphi^* = (\dots (\varphi^{A_1})\dots)^{A_n}$. And where \begin{eqnarray} \pi^* & \overset{\textnormal{def}}{=} & \bigwedge_{k< n} \bigwedge_{\delta \in \Delta_k} (\delta \wedge A_{k+1}) \approx (\delta \wedge \neg A_{k+1}), \label{eq:pi*}
\end{eqnarray} multiple applications of Lemma \ref{lemma:rel} allow us to conclude: \begin{equation}\vdash_{\mathsf{AX}_{\textnormal{comp}}} \pi^* \rightarrow (\varphi \leftrightarrow \varphi^*). \label{eq:restrict}\end{equation} Thus, under the assumption $\pi^*$, it suffices just to derive $\varphi^*$. 

Observe furthermore that we now essentially have $m$ copies of each state description $\delta$ over $\mathsf{Prop}^{\varphi}$, each copy `tagged' by a distinct state description $\sigma$ over $\mathsf{Prop}^+$. The conjunction $\delta \wedge \sigma$ is in fact an element of $\Delta_n$, i.e., a state description over $\mathsf{Prop}^{\varphi} \cup \mathsf{Prop}^+$. It is straightforward to show that $\mathsf{AX}_{\textnormal{comp}}$ proves $\pi^* \rightarrow (\delta \wedge \sigma_i) \approx (\delta \wedge \sigma_j)$ for each $\delta \in \Delta$ and all $\sigma_i\neq \sigma_j$; that is, every pair of elements of $\Delta_n$ that agree on $\mathsf{Prop}^{\varphi}$ are provably equiprobable. Because $\sigma_i\neq\sigma_j$ we also know that all pairs $\delta \wedge \sigma_i$, $\delta \wedge \sigma_j$ are jointly unsatisfiable. 

Thus, suppose that $\varphi^*$ is valid, that $\neg \varphi^*$ is in disjunctive normal form, and consider any disjunct, a conjunction of equality statements ($\approx$) and strict inequalities ($\succ$) between disjunctions of \emph{relativized} state descriptions; that is, each (in)equality is between disjunctions of conjunctions $\delta \wedge \bigwedge_{i\leq n}A_i$, where $\delta \in \Delta$. It remains to show that each step of the variable elimination (or `$\delta$-elimination') argument from Theorem \ref{thm:add-complete} can be emulated here. 

%Rather than a $k$-fold sum of $\mathbf{P}(\delta)$, as in Theorem \ref{thm:add-complete}, here the 
Our aim is to eliminate the `variable' $x = (\delta \wedge \sigma_1)$, where $\sigma_1 = \bigwedge_{i \leq n}A_i$. Let $kx$ stand for any $k$-fold disjunction $(\delta \wedge \sigma_{i_1}) \vee \dots \vee (\delta \wedge \sigma_{i_k})$, analogous to the $k$-fold sum $\mathbf{P}(\delta)+\dots+\mathbf{P}(\delta)$ in the proof of Theorem \ref{thm:add-complete}. At the start we will always have $k=1$, but as we proceed through the elimination of variables some will appear with greater multiplicity (but again, no greater than $m$). Thus, in general, every conjunct containing $x$ will have one of the following three forms: 
\begin{multicols}{3}
\begin{enumerate}[label=(\roman*)]
\item \label{f12} $kx \vee \alpha \approx \beta$
\item \label{f22} $kx \vee \alpha \succ \beta$
\item \label{f32} $\beta \succ kx \vee \alpha$
\end{enumerate} \end{multicols} \noindent where $kx$, $\alpha$, and $\beta$ are all mutually incompatible Boolean formulas. Here we are using $\mathsf{Dist}$, $\mathsf{DuplQ}$, and $\mathsf{Quasi'}$, which also allow us to assume $k$ is the same across conjuncts containing $x$. Note also that either of $\alpha$ or $\beta$ can be the empty disjunction $\bot$. We now carry out the same case distinctions as in the proof of Theorem \ref{thm:add-complete}:
\begin{enumerate}[label=(\alph*)]
  \item \emph{There are only conjuncts of type \ref{f32}.} In this case we can simply replace each instance $\beta \succ kx \vee \alpha$ with $\beta \succ \alpha$, which results in an equisatisfiable formula. The latter can be derived from the former by $\mathsf{Dist}$ and transitivity of $\succsim$. 
  \item \emph{There are only conjuncts of type \ref{f22}.} The entire disjunct will then make no difference to satisfiability, so we can replace it with any tautology. 
  \item \emph{There are conjuncts of both types \ref{f22} and \ref{f32}.} This case is handled by the quasi-additive version\footnote{That is, $((\varepsilon \vee \alpha) \succ \beta \wedge \gamma \succ (\varepsilon \vee \delta)) \rightarrow (\alpha \vee \gamma) \succ (\beta \vee \delta)$.} of $\mathsf{Elim}$. For each pair $kx \vee  \alpha \succ \beta$ and $\gamma \succ kx\vee \delta$, we want to replace it with a new conjunct $\alpha \vee \gamma \succ \beta \vee \delta$, which does not involve $x$. But this is only guaranteed to produce an equisatisfiable result if $\alpha$ and $\gamma$ do not share disjuncts (and the same for $\beta$ and $\delta$). If $\alpha$ and $\gamma$ share a disjunct, then we let $\gamma'$ be just like $\gamma$ but with a distinct $\sigma_i$ for that disjunct. Because all such copies of the disjunct are provably equiprobable, we have $\gamma \approx \gamma'$, and thus $\gamma' \succ kx \vee \delta$. Performing any necessary analogous replacement to obtain $\delta'$ in addition, the result, $(\alpha \vee \gamma') \succsim (\beta \vee \delta')$, in place of $kx \vee  \alpha \succ \beta$ and $\gamma \succ kx\vee \delta$---again, derivable from the latter by the variant of $\mathsf{Elim}$---will give an equisatifiable transformation of the original formula, now without any appearance of $x$. 
  \item \emph{There is at least one conjunct of type \ref{f12}.} In this case, the quasi-additive versions\footnote{To wit: $(\varepsilon \vee \gamma)\approx \alpha \rightarrow ((\alpha \vee \delta)\succsim(\beta \vee \gamma) \leftrightarrow (\varepsilon \vee \delta)\succsim \beta)$ and $(\varepsilon \vee \gamma)\approx \alpha \rightarrow ((\beta \vee \gamma)\succsim(\alpha \vee \delta) \leftrightarrow \beta\succsim(\varepsilon \vee \delta))$.} of $\mathsf{Sub1}$ and $\mathsf{Sub2}$ allow eliminating $x$ from every other conjunct of type \ref{f12}, as well as all conjuncts of types \ref{f22} and \ref{f32}. As in the previous case, we may need to use duplicates of state descriptions, but this can be done in the very same manner.
  
  It remains to show that $x$ can be eliminated from the last equality $kx \vee \alpha \approx \beta$. Since $x$ appears nowhere else in the conjunct, the whole formula will be equisatisfiable with the result of replacing  $kx \vee \alpha \approx \beta$ with $\beta\succsim\alpha$. The latter is derivable from the former by $\mathsf{Dist}$ and transitivity of $\succsim$.
\end{enumerate}
Finally, after eliminating all variables, we will end up with a conjunction of statements each provably equivalent (by $\mathsf{Dist}$) to either $\bot \succsim \bot$ or $\neg \bot \succsim \bot$. Unsatisfiability of $\neg \varphi^*$ means the latter must be a conjunct, but this formula is refutable in  $\mathsf{AX}_{\text{comp}}$. 

Consequently $\varphi^*$ is derivable, and by (\ref{eq:restrict}), $\varphi$ is itself derivable, assuming $\pi^*$. That is, we have shown that  $\vdash_{\mathsf{AX}_{\text{comp}}}\pi^* \rightarrow \varphi$. Because $\varphi$ does not involve any of the new atoms in $\mathsf{Prop}^+$, we can iteratively discharge the assumption of $\pi^*$ by $\mathsf{Polarize}$, conjunct by conjunct from (\ref{eq:pi*}). 
\end{proof}

\subsection{Finite Axiomatizability}\label{section: finite axiomatizability}

We saw that the canonical axiomatization of the logic of comparative probability $\mathcal{L}_{\text{comp}}$ features infinitely many axiom schemes. The finite cancellation axioms feature a separate axiom scheme $\varphi_n (\alpha_{1},...\alpha_{k_n})$ for each $n\in \mathbb{N}$, where the $\alpha$'s range over the Boolean formulas. By contrast, observe that the system $\mathsf{AX}_{\text{add}}$ for the logic of (explicitly arithmetical) additive comparisons is finitely (scheme-)axiomatizable, in the sense that it is given by a finite set of axiom schemes over $\mathsf{AX}_{\text{base}}$: that is, the axiomatization features only finitely many axiom schemes of the form $\varphi_n (\alpha_{1},...\alpha_{k_{n}})$, where the $\alpha$'s range over the Boolean formulas, and finitely many axiom schemes $\psi_n (\mathsf{t}_{1},...\mathsf{t}_{k_{n}})$, where the $\mathsf{t}_{i}$'s range over the terms in the language. This notion of finite axiomatizability\textemdash axiomatizabilty by finitely many schemes\textemdash is the natural one to consider in our propositional setting. The standard axiomatizations of finite comparative probability structures in the literature are given by schemes of this form, with an implicit universal quantification over events. For $\mathcal{L}_{\text{comp}}$, finite (scheme-)axiomatizability in our sense corresponds to the finite axiomatizability of comparative probability structures, over finite structures,\footnote{Recall that a class of structures $\mathbb{K}$ is axiomatized by $\Gamma$ over finite structures if $\mathbb{K}=\text{Mod}(\Gamma) \cap \mathsf{Fin}$: that is, $\mathbb{K}$ is exactly the class of \emph{finite} models of $\Gamma$.} by a universal sentence in a first-order language with quantification over events (or finite axiomatizability \emph{tout court}, if we include a uniform substitution rule in our system). 

Here we show that, for the less expressive language $\mathcal{L}_{\text{comp}}$, the presence of infinitely many schemes is inevitable. As opposed to explicitly `arithmetical' system $\mathsf{AX}_{\text{add}}$, the logic of $\mathcal{L}_{\text{comp}}$ is not finitely axiomatizable in that sense. Unless we enrich the system with powerful additional inference rules like we did in Section \ref{Polarizationsection}, no finite set of axiom schemes can capture the $\mathcal{L}_{\text{comp}}$-validities.

Our proof proceeds in two steps. We first note that a finite axiomatization of $\mathcal{L}_{\text{comp}}$ would result in a finite universal axiomatization, over finite structures, of the class of comparative probability orders in the first-order language of ordered Boolean algebras. We then appeal to a variant of a theorem of \cite{vaught1954} to show that comparative probability orders are not finitely axiomatizable by a universal sentence over finite structures. The construction appeals to combinatorial results by  \cite{fishburn1997failure} on finite cancellation axioms. 

\subsubsection{Vaught's theorem and finite axiomatizability}

We work with the language $\mathcal{L}_{\mathsf{BA}}\cup\{\succsim\}$ of Boolean algebras with an additional binary relation. The signature $\mathcal{L}_{\mathsf{BA}}$ is given by $(\mb{0}, \mb{1}, \otimes, \oplus, \cdot^{\bot})$ where the constant symbols $\mb{0}$ and $\mb{1}$ stand for the bottom and top element, the binary function symbols $\otimes$ and $\oplus$ stand for the Boolean meet and join operations, respectively, and the unary function symbol $\cdot^{\bot}$ stands for Boolean complementation. Consider the following class of  $\mathcal{L}_{\mathsf{BA}}\cup\{\succsim\}$-structures.

\begin{definition}[The class $\mathsf{FCP}$]
The class $\mathsf{FCP}$ of \emph{finite comparative probability structures} consists of all finite Boolean algebras with a representable comparative probability order: i.e., all structures of the form $$\mathcal{A}=(A, \mb{0}_{\mathcal{A}},\mb{1}_{\mathcal{A}}, \otimes, \oplus, \cdot^{\bot}, \succsim)$$ 
where $(A, \mb{0}_{\mathcal{A}},\mb{1}_{\mathcal{A}}, \otimes, \oplus, \cdot^{\bot})$ is a Boolean algebra, and the order $\succsim$ on $A$ is representable by a probability measure, in that there exists some probability measure $\prob$ on $\mathcal{A}$ such that, for all $a,b\in A$, we have $
a\succsim b \,\,\text{ if and only if }\,\,\prob(a)\geq\prob(b)
$.
\end{definition}
 
 We can naturally translate each $\varphi$ in $\mathcal{L}_{\text{comp}}$ into a universally quantified $\widehat{\varphi}$ in  $\mathcal{L}_{\mathsf{BA}}\cup\{\succsim\}$: we assign a variable $p_{i}^{\ast}:= x_{i}$ to each atomic $p_{i}$, and extend it in the obvious way so that each Boolean expression $\beta$ gets assigned a corresponding term $\beta^{\ast}\in\mathcal{L}_{\mathsf{BA}}$.\footnote{We let $(\neg \beta)^{\ast} := (\beta^{\ast})^{\bot}$, $(\alpha\wedge \beta)^{\ast} := \alpha^{\ast} \otimes \beta^{\ast}$ and $(\alpha\vee \beta)^{\ast} := \alpha^{\ast} \oplus \beta^{\ast}$.} For each term $\mathbf{P}(\beta)$, we have $\mathbf{P}(\beta)^{\ast} := \beta^{\ast}$, and each formula $\varphi\in \mathcal{L}_{\text{comp}}$ is translated into a quantifier free $\varphi^{\ast}\in\mathcal{L}_{\mathsf{BA}}\cup \{\succsim\}$.\footnote{Let $
(t_1\succsim t_2)^{\ast} :=t_1^{\ast} \succsim t_2^{\ast}$, $
(\varphi\wedge \psi)^{\ast} := \varphi^{\ast} \wedge \psi^{\ast}$, $(\varphi\vee \psi)^{\ast} := \varphi^{\ast} \vee \psi^{\ast}$, and $(\neg\varphi)^{\ast}:=\neg \varphi^{\ast}$.} Now take the translation that assigns, to each  $\varphi\in\mathcal{L}_{\text{comp}}$, the formula
$$
\widehat{\varphi} := \forall x_{1}...\forall x_{n} \varphi^{\ast} 
$$
where $p_{1},...,p_{n}$ are all the atomic propositions occurring in $\varphi$. Then $\varphi$ is valid on all probability models if and only if $\mathsf{FCP}\models \widehat{\varphi}$. In particular, a scheme $\varphi(\alpha_{1},\dots,\alpha_{n})$ is valid on all probability models if and only if $\mathsf{FCP}\models \widehat{\varphi}$. Observe also that, given Theorem \ref{SCOTTTHM}, the class $\mathsf{FCP}$ is axiomatized over finite structures by universal sentences, by taking standard universal axioms for Boolean algebras and adding the translations of the infinitely many $\mathsf{FinCan}_{n}$ axiom schemes (as well as the $\mathsf{AX}_{\text{base}}$ axioms).

To show that the logic of comparative probability $\mathcal{L}_{\text{comp}}$ is not finitely axiomatizable over $\mathsf{AX}_{\text{base}}$, it suffices to show that $\mathsf{FCP}$ is not axiomatizable over finite structures by a single universal sentence in $\mathcal{L}_{\mathsf{BA}}\cup\{\succsim\}$. For suppose there was a finite collection of schemes $\Delta=\{\psi_{1},...,\psi_{k}\}$ such that every $\mathcal{L}_{\text{comp}}$-validity over probability models followed from $\Delta$. Then $\widehat{\Delta}:=\{\widehat{\psi}_{1},\dots, \widehat{\psi}_{k} \}\subset \mathcal{L}_{\mathsf{BA}}\cup\{\succsim\}$ would finitely axiomatize the universal theory of $\mathsf{FCP}$ over finite structures. But $\mathsf{FCP}$ being universally axiomatizable, this would amount to an axiomatization of $\mathsf{FCP}$ by a single universal sentence (over finite structures). We will now show that $\mathsf{FCP}$ is not axiomatizable by a universal sentence over finite structures.

We recall one useful model-theoretic definition that will be needed. 

\begin{definition}[Uniformally locally finite classes]
A class $\mathbb{K}$ of first-order structures is \emph{uniformly locally finite} if there exists a function $f:\mathbb{N}\to\mathbb{N}$ such that for any $\mathfrak{M}\in\mathbb{K}$ and any subset $\{a_{1},\dots,a_{n}\}\subseteq \text{dom}(\mathfrak{M})$, we have $|\mathfrak{M}\langle\vec{a}\rangle| \leq f(n)$, where $\mathfrak{M}\langle\vec{a}\rangle$ is the substructure of $\mathfrak{M}$ generated by $\{a_{1},\dots,a_{n}\}$.
\end{definition}

We will make use of the following (minor variant of) Vaught's characterisation of structures axiomatizable by a universal sentence \citep{vaught1954}.

\begin{proposition}\label{VaughtProp}
Let $\mathbb{K}$ a uniformly locally finite class of structures in a finite first-order signature $\mathcal{L}$.  If $\mathbb{K}$ is axiomatizable by a universal sentence over finite structures, then
\begin{itemize}
    \item[(i)] $\mathbb{K}$ is closed under substructures
    \item[(ii)] There is some $n\in\mathbb{N}$ such that, for any finite structure $\mathfrak{A}$, if every substructure $\mathfrak{B}\subseteq \mathfrak{A}$ of size at most $n$ belongs to $\mathbb{K}$, then $\mathfrak{A}\in\mathbb{K}$.
\end{itemize}
\end{proposition}
\begin{proof}
(i) is immediate. Let $\mathsf{Fin}$ the class of finite structures. For (ii), suppose $\mathbb{K}=\text{Mod}(\varphi)\cap \mathsf{Fin}$ for some $\varphi\in\mathcal{L}$ of the form $\varphi = \forall x_{1}\dots\forall x_{k} \psi(x_{1},...,x_{k})$, with $\psi$ quantifier-free. Since $\mathbb{K}$ is uniformaly locally finite, there is some $f:\mathbb{N}\to\mathbb{N}$ such that, for any $\mathfrak{A}\in\mathbb{K}$, any substructure of $\mathfrak{A}$ generated by $k$ elements has size at most $f(k)$. Take $n=f(k)$, and we show that this $n$ satisfies (ii). Take a finite structure $\mathfrak{A}\not\in\mathbb{K}$. We show $\mathfrak{A}$ contains a substructure of size at most $f(k)$ which is not in $\mathbb{K}$. We have $\mathfrak{A}\not\models \varphi$, so $\exists \bar{a}=(a_{1},...,a_{k})\in\mathfrak{A}$ with  $\mathfrak{A}\models \neg\psi[\bar{a}]$. Consider $\mathfrak{A}{\langle\bar{a}\rangle}$, the substructure of $\mathfrak{A}$ generated by $\{a_{1},...,a_{k}\}$. By definition of $f$, we have $|\mathfrak{A}{\langle\bar{a}\rangle}|\leq f(k)$. Since $\neg\psi$ is quantifier-free, it is preserved under substructures, so $\mathfrak{A}{\langle\bar{a}\rangle}\models\neg\psi[\bar{a}]$, hence  $\mathfrak{A}{\langle\bar{a}\rangle}\not\models \varphi$. So $\mathfrak{A}{\langle\bar{a}\rangle}\not\in\mathbb{K}$, which establishes (ii).  
\end{proof}

Note that the class $\mathsf{FCP}$ of finite Boolean algebras with a representable comparative probability order is locally finite, with the uniform bound on the size of substructures given by $f(n)=2^{2^n}$. We will use Proposition \ref{VaughtProp} to show that $\mathsf{FCP}$ is not axiomatizable over finite structures by a universal formula. 

\subsubsection{The logic of comparative probability is not finitely axiomatizable}

In order to show that $\mathsf{FCP}$ is not axiomatizable by a universal formula, we appeal to a combinatorial analysis of cancellation axioms in the setting of a restricted class of ordered Boolean algebras, which corresponds to one of the earliest comparative probability structures introduced by \cite{Finetti1937}.

\begin{definition}[de Finetti orders]
A \emph{de Finetti order} is a structure $(\mathcal{B},\succsim)$, where $\mathcal{B}$ is a Boolean algebra and $\succsim$ a binary relation on $\mathcal{B}$ satisfying the following: 
\begin{eqnarray*}
\mathsf{Tot}. & &  \succeq \text{ is a reflexive total order;}\\
\mathsf{NonDeg}. & &  \neg (\mb{0} \succeq \mb{1});\\
\mathsf{NonTriv}. & & \forall x (x\succsim \mb{0});\\
%\mathsf{Quasi}. & & \forall x_1 x_{2} \Big( x_1\succsim x_{2} \leftrightarrow \big( x_1 \otimes x^{\bot}_{2} \succsim x_{2}\otimes  x^{\bot}_1 \big)\Big).
\mathsf{Quasi}. & & \forall x_1 x_{2} \Big( x_1\succsim x_{2} \leftrightarrow \big( x_1 \otimes x^{\bot}_{2} \succsim x_{2}\otimes  x^{\bot}_1 \big)\Big).
\end{eqnarray*}
A \emph{linear} de Finetti order is one where the strict binary relation $\succ$, defined by $a\succ b$ iff $\neg (b \succsim a)$, is a linear order. A de Finetti order on $n$ atoms is one where the underlying Boolean algebra $\mathcal{B}$ is a finite algebra generated by $n$ atoms.
\end{definition}

Note that the definition of de Finetti orders simply characterizes, in the language of ordered Boolean algebras, comparative orders satisfying the quasi-additivity axiom discussed in Section \ref{subsection: quasi-additivity}.\footnote{Recall that, in the plain set-theoretic language of comparative probability orders, quasi-additivity is equivalent to the statement that for any $A$, $B$ and $C$ such that $(A\cup B)\cap C=\varnothing$, we have $A\succsim B$ if and only if $A\cup C\succsim B\cup C$.} The notion of two sequences $(a_{1},\dots, a_{n})$ and $(b_{1},\dots, b_{n})$ of elements from $\mathcal{B}$ being \emph{balanced} transfers naturally to the setting of Boolean algebras.  Given $b\in\mathcal{B}$ and a $\mathcal{B}$-atom $c$, say that $b$ \emph{lies above} $c$ if $c\otimes b = c$. Now define $(a_{1},\dots, a_{n})$ and $(b_{1},\dots, b_{n})$ to be balanced if, for every atom $c$ in the algebra, the number of $a_{i}$'s above $c$ equals the number of $b_{i}$'s above $c$. We write $(a_{1},\dots, a_{n})\equiv_{0} (b_{1},\dots, b_{n})$ to express that the two sequences are balanced.\footnote{It follows from the formulation in balancedness \eqref{Balancedsequences} in Section \ref{subsection: quasi-additivity} that the statement $(a_{1},\dots, a_{n})\equiv_{0} (b_{1},\dots, b_{n})$ can in fact be expressed in $\mathcal{L}_{\mathsf{BA}}\cup\{\succsim\}$.} Consider the following property $\mathsf{S}_{k}$:
\begin{itemize}
    \item[($\mathsf{S}_{k}$)] If $(a_{i})_{i\leq N} \equiv_{0} (b_{i})_{i\leq N}$ are balanced sequences with \emph{at most $k$ distinct pairs} $(a_{i}, b_{i})$, then it is not the case that $a_{i}\succ b_{i}$ for all $i\leq N$.
\end{itemize}

\noindent ($\mathsf{S}_{k}$) is a variant of the finite cancellation axiom $\mathsf{FinCan}_{k}$. Note that $k$ here counts the number of \emph{distinct premises} $a_{i}\succ b_{i}$ in the antecedent, and not the number of premises. Observe also that a linear de Finetti order satisfying $S_k$ for \emph{all} $k\in\mathbb{N}$ also satisfies all instances of finite cancellation $\mathsf{FinCan}_{k}$. Thus, by Theorem \ref{SCOTTTHM}, linear de Finetti orders satisfying all axioms ($\mathsf{S}_{k}$) belong to $\mathsf{FCP}$: they are probabilistically representable. 

One can use the ($\mathsf{S}_{k}$) axioms to measure `\emph{how much}' finite cancellation an order needs to satisfy in order to be probabilistically  representable. Given a fixed bound on the size of a finite algebra (say, at most $n$ atoms), we can ask: is there some $k$ such that every de Finetti order of this size satisfying $S_k$ is representable? \cite{Fishburn1996, fishburn1997failure}  investigates such bounds. He defines: 
\begin{align*}
f(n):= \min\big\{k\in\mathbb{N}\,|\, & \text{every linear de Finetti order on $n$ atoms}  \\
& \text{that satisfies }  \mathsf{S}_{k} \text{ is representable} \big\}
\end{align*}
Known bounds on $f(n)$ are given by the following:
\begin{proposition}[\citealt{kraft1959intuitive,Fishburn1996}]\label{PropBounds}
For all $n$, $f(n)\leq n+1$.
\end{proposition}
\begin{proposition}[\citealt{fishburn1997failure}]\label{FISHBURN} 
For any $m\geq 6$, there exists a linear de Finetti order on a Boolean algebra with $m$ atoms that fails $\mathsf{S}_{m-1}$, but satisfies $\mathsf{S}_{m-2}$. 
\end{proposition}
We now use these bounds to prove our desired result.
\begin{proposition}
The class $\mathsf{FCP}$ is not axiomatizable by a universal sentence over finite structures.
\end{proposition}
\begin{proof}
We show that condition (ii) of Proposition \ref{VaughtProp} fails for $\mathsf{FCP}$. That is, we show that for any $n\in \mathbb{N}$, there exists some finite structure $\mathfrak{A} = (\mathcal{A},\succsim)\not\in\mathsf{FCP}$ such that every one of its substructures $\mathfrak{B}=(\mathcal{B},\succsim\restriction  \mathcal{B})$ of size at most $n$ is in $\mathsf{FCP}$. Given sufficiently large $n$, take a linear de Finetti order $\mathfrak{A}$ as given by Proposition \ref{FISHBURN} with $m\geq \log_{2}n + 3$. Then $\mathfrak{A}$ is not representable, as it fails $\mathsf{S}_{m-1}$. Now take any of its substructures $(\mathcal{B},\succsim\restriction  \mathcal{B})$ of size at most $n \leq 2^{m-3}$. Such an algebra has at most $m-3$ atoms, and it is evidently a linear de Finetti order (note that linear de Finetti orders are axiomatizable by a universal sentence, hence preserved under substructures). Since $(\mathcal{A},\succsim)$ satisfies $\mathsf{S}_{i}$ for all $i\leq (m-2)$, so does the substructure $(\mathcal{B},\succsim\restriction \mathcal{B})$: for any violation of $\mathsf{S}_{i}$ in $(\mathcal{B},\succsim\restriction \mathcal{B})$ would also hold in $(\mathcal{A},\succsim)$.\footnote{Each inequality $a_{i} \succ b_{i}$ is obviously preserved under substructures; so whether any instance of $\mathsf{S}_{i}$ holds only depends on whether all the elements $a_{i}$, $b_{i}$ involved are indeed in the subalgebra.} But now $\mathcal{B}$ is an algebra on at most $m-3$ atoms, and by Proposition \ref{PropBounds} $f(m-3)\leq m-2$. This means that any linear De Finetti order on $m-3$ atoms that satisfies $\mathsf{S}_{m-2}$ is representable, since it then also satisfies $\mathsf{S}_{f(m-3)}$. So any $\mathfrak{B}\subset \mathfrak{A}$ of size at most $n$ is representable.
\end{proof}
From this we can conclude: 
\begin{theorem}
$\mathcal{L}_{\text{comp}}$ is not finitely axiomatizable over $\mathsf{AX}_{\text{base}}$.
\end{theorem}

\subsection{Complexity of additive systems}\label{section: additive complexity}

In this section, we recall well-known results that characterize the complexity of reasoning in the additive systems $\mathcal{L}_{\text{comp}}$ and $\mathcal{L}_{\text{add}}$.

\begin{theorem}\label{additive-complexity}
$\mathsf{SAT}_{\text{comp}}$ and $\mathsf{SAT}_{\text{add}}$ are $\mathsf{NP}$-complete.
\end{theorem}

We rehearse a proof by Fagin, Halpern, and Megiddo \citeyearpar{fagin1990logic} to show that $\mathsf{SAT}_{\mathrm{add}}$ is $\mathsf{NP}$-complete. This requires two lemmas:

\begin{lemma}\label{linear-small-model}
If there exists a non-negative solution to a system of $m$ linear inequalities with integer coefficients each of length at most $\ell$, then the system has a non-negative solution with at most $m$ nonzero entries, and where the size of each entry is $O(m \ell + m \log (m))$.
\end{lemma}

\begin{proof}

We begin by transforming the system of linear inequalities into a linear program. Let $x_1,...,x_n$ denote the variables appearing in the system. For each non-strict inequality constraint $\sum_{j} a_{i,j} x_j \leq b_j$, one can introduce a slack variable $x_{n+j}$ and define the equality constraint $\sum_{j} a_{i,j} x_j + x_{n+j} = b_j$. Similarly, for a strict inequality constraint $\sum_{j} a_{i,j} x_j < b_j$, one can introduce $x_{n+j}$ and write $\sum_{j} a_{i,j} x_j + x_{n+j} = b_j$, adding to this the constraint that $x_{n+j} \geq x_0$. This gives rise to system of linear constraints $\textbf{A}\textbf{x} = \textbf{b}$ in the variables $x_0,....,x_{n+m}$, which can be placed in the following linear program:
\begin{align*}
    \text{maximize } &x_0\\
    \text{subject to } &\textbf{A}\textbf{x} = \textbf{b}, \textbf{x} \geq 0
\end{align*}
Because the original system has a non-negative solution, the above linear program has a solution for which the objective function $x_0$ is positive. It is well-known (see, for example, Ch. 8 of \citealt{chvatal1983linear}) that the simplex algorithm, which traverses the vertices of a convex polytope associated with the system of inequalities $\textbf{A}\textbf{x} \leq \textbf{b}$, will discover an optimal, non-negative solution $\textbf{x}^*$ to the system $\textbf{A}\textbf{x} = \textbf{b}$, and in this case it follows from optimality that $x_0^* >0$. Thus $x_1^*,...,x_n^*$ provide a non-negative solution to the original system of inequalities. The simplex algorithm explores solutions to the linear program which lie at vertices by successively setting $n - m$ variables to 0 and setting the remaining $m$ variables so as to satisfy the linear constraints, so the solution $\textbf{x}^*$ has at most $m$ variables positive.

Following the presentation in \cite{Williamson2008}, we now bound the size of the non-zero entries of $\textbf{x}^*$. We observe that deleting the zero entries of $\textbf{x}^*$ and the corresponding rows of $\textbf{b}$ and columns of $\textbf{A}$ gives vectors $\bar{\textbf{x}}$ and $\bar{\textbf{b}}$ and a matrix $\bar{\textbf{A}}$ such that $\bar{\textbf{A}} \bar{\textbf{x}} = \bar{\textbf{b}}.$ By Cramer's rule,
\[
\bar{x}_j = \frac{\text{det}(\bar{\textbf{A}}_j)}{\text{det}(\bar{\textbf{A}})},
\]
where $\bar{\textbf{A}}_j$ is the result of replacing the $j^{th}$ column of $\bar{\textbf{A}}$ with $\bar{\textbf{b}}$. It suffices to show that one can express $\text{det}(\bar{\textbf{A}})$ using at most $O(m \ell + m \log (m))$ bits. Recall that
\[
\text{det}(\bar{\textbf{A}}) = \sum_{\sigma \in S_m} (-1)^{N(\sigma)} a_{1 \sigma(1)} \cdot\cdot\cdot a_{m \sigma(m)},
\]
where $S_m$ is the set of permutations of $[m] = \{1,...,m\}$ and \[N(\sigma) = \{i, j \in [m] : i < j \text{ and } \sigma(i) > \sigma(j)\}\] is the set of indices inverted by $\sigma$.

Each entry $a_{i \sigma(i)}$ of $\bar{\textbf{A}}$ has size at most $\ell$. Thus each term in the above sum has size at most $m \cdot \ell$. Noting that $|S_m| = m!$, relabel the terms in the above sum and define $y_i$ such that
\[
\text{det}(\bar{\textbf{A}}) = \sum_{ i \in [m!]} y_i.
\]
Group the $y_i$ in pairs arbitrarily and sum them to produce a sequence $y_i^\prime$ which is half as long:
\[
\text{det}(\bar{\textbf{A}}) = \sum_{ i \in [m! /2]} y_i^\prime.
\]
The size of each sum $y_i^\prime$ is at most one greater than that of its summands. Repeating this process $\log (m!) \leq \log (m^m) = m \cdot \log (m)$ times produces a single number of size at most $m \cdot \ell + m \cdot \log (m)$.
\end{proof}

\begin{definition}
Let $|\varphi|$ denote the number of symbols required to write $\varphi$, and let $||\varphi||$ denote the length of the longest rational coefficient of $\varphi$, written in binary.
\end{definition}

\begin{lemma}\label{additive-small-model}
Suppose $\varphi \in \mathcal{L}_{\text{add}}$ is satisfiable. Then $\varphi$ has a model which assigns positive probability to at most $|\varphi|$ events $\delta \in \Delta_\varphi$, where the probability assigned to each such $\delta$ is a rational number with size $O(|\varphi| \cdot ||\varphi|| + |\varphi|\cdot \log(|\varphi|) )$.
\end{lemma}

\begin{proof}
Since $\varphi$ is satisfiable, it has a model $\mathfrak{M}$ which is also a model of some disjunct $\psi$ appearing in the disjunctive normal form of $\varphi$. Pushing all sums in $\psi$ to one side, we observe that $\psi$ is a conjunction of formulas of the form
\[
\sum_{i} \mathbf{P}(\epsilon_i) \succsim b \,\,\,\text{ or }\,\,\, \sum_{i} \mathbf{P}(\epsilon_i) \succ b.
\]
Let $\psi(\Delta_\varphi)$ be the result of replacing each instance of $ \mathbf{P}(\epsilon_i)$ in $\psi$ with the sum $$\sum_{\substack{\delta \in \Delta_\varphi\\ \models\delta \rightarrow \epsilon_i}}  \mathbf{P}(\delta)$$ and adding the constraint $\sum_{\delta \in \Delta_\varphi} \mathbf{P}(\delta) = 1$. Then a model of $\psi(\Delta_\varphi)$ is a model of $\psi$ and so a model of $\varphi$, and the formula $\psi(\Delta_\varphi)$ simply describes a system of at most $|\varphi|$ linear inequalities, so the result follows immediately from Lemma~\ref{linear-small-model}.
\end{proof}

\begin{proof}[Proof of Theorem~\ref{additive-complexity}]
Since $\mathcal{L}_{\text{comp}} \subseteq \mathcal{L}_{\text{add}} $, it follows that $\mathsf{SAT}_{\text{comp}} \leq \mathsf{SAT}_{\text{add}}$. It thus suffices to show that $\mathsf{SAT}_{\text{comp}}$ is $\mathsf{NP}$-hard and that $\mathsf{SAT}_{\text{add}}$ is in $\mathsf{NP}$.

The Cook-Levin Theorem \cite{cook1971complexity} states that satisfiability for Boolean formulas $\alpha$ is $\mathsf{NP}$-complete. This is equivalent to the problem of deciding whether $\mathbf{P}(\alpha) \succ \mathsf{0}$ is satisfiable, which is an instance of the problem $\mathsf{SAT}_{\text{comp}}$, showing that the latter is $\mathsf{NP}$-hard.

Consider now the task of determining whether $\varphi \in \mathcal{L}_{\text{add}}$ is satisfiable. Using Lemma~\ref{additive-small-model}, we request a small model as a certificate and confirm that it satisfies $\varphi$. 
\end{proof}

\section{Multiplicative Systems}\label{section: Multiplicative Systems}

\subsection{Polynomial probability calculus}\label{section: polynomial probability calculus}
%Todo: list axioms, proof of completeness (as a Lemma: prove in full detail that replacement of equivalents is admissible).
The paradigmatic multiplicative language is $\mathcal{L}_{\textnormal{poly}}$ (Definition~\ref{def:multlang}), which adds a binary multiplication operator (denoted $\cdot$). Encompassing all comparisons between polynomial functions of probabilities, this language is sufficiently rich to express, e.g., conditional probability and (in)dependence.
% Recall that for the polynomial language $\mathcal{L}_{\textnormal{poly}}$ (Definition~\ref{def:multlang}) we add a binary multiplication operator (denoted $\cdot$).
% and unary negation $-$.
% \begin{remark}
% Is there a convenient axiomatization if we omit unary negation (as in the additive case)? One difficulty is proving the non-negativity of sums of squares, e.g. $\mathsf{a}\cdot \mathsf{a} + \mathsf{b}\cdot \mathsf{b} \succsim \mathsf{a}\cdot\mathsf{b}+ \mathsf{a}\cdot\mathsf{b}$.
% \end{remark}
The system $\mathsf{AX}_{\text{poly}}$ (Figure~\ref{fig:poly}) annexes axioms capturing multiplication and its interaction with addition to $\mathsf{AX}_{\text{add}}$, and our principal result here is its completeness.

%(add remark on how this does not admit natural string, etc. interpretations---after proof?)

\begin{figure}[t]
\begin{mdframed} \[\underline{\textbf{The Axioms of }\mathsf{AX}_{\textnormal{poly}}}\] \vspace{-.3in} 
\begin{eqnarray*}
\mathsf{Assoc}. & & \mathsf{a} \cdot (\mathsf{b}\cdot\mathsf{c}) \approx (\mathsf{a}\cdot\mathsf{b})\cdot\mathsf{c} \\
\mathsf{Comm}. & & \mathsf{a}\cdot\mathsf{b} \approx \mathsf{b}\cdot\mathsf{a} \\ 
\mathsf{Zero}. & & \mathsf{a} \cdot \mathsf{0} \approx \mathsf{0} \\
\mathsf{One}. & & \mathsf{a} \cdot \mathsf{1} \approx \mathsf{a} \\
% \mathsf{Canc1}. & & \mathsf{a}\cdot\mathsf{c}\succsim\mathsf{b}\cdot\mathsf{c} \land \mathsf{c} \succ \mathsf{0} \rightarrow \mathsf{a} \succsim \mathsf{b} \\
% \mathsf{Canc2}. & & \mathsf{a} \succsim \mathsf{b} \rightarrow \mathsf{a}\cdot\mathsf{c}\succsim\mathsf{b}\cdot\mathsf{c} \\ % valid by positive interpretation
\mathsf{Canc}. & & \mathsf{c} \succ \mathsf{0} \rightarrow (\mathsf{a} \cdot \mathsf{c} \succsim \mathsf{b} \cdot \mathsf{c} \leftrightarrow \mathsf{a} \succsim \mathsf{b})\\
\mathsf{Dist}. & & \mathsf{a}\cdot(\mathsf{b} + \mathsf{c}) \approx \mathsf{a}\cdot \mathsf{b} + \mathsf{a} \cdot \mathsf{c} \\
% \mathsf{NoZeroDiv}. & & \mathsf{a} \cdot \mathsf{b} \approx \mathsf{0} \rightarrow \mathsf{a} \approx \mathsf{0} \lor \mathsf{b} \approx \mathsf{0} \\ % derivable? yes, using Canc1.
\mathsf{Sub}. & & \mathsf{a} \succsim \mathsf{b} \land \mathsf{c} \succsim \mathsf{d} \rightarrow \mathsf{a} \cdot \mathsf{c} + \mathsf{b} \cdot \mathsf{d} \succsim \mathsf{a} \cdot \mathsf{d} + \mathsf{b}\cdot \mathsf{c}
% \mathsf{Dupl}. & & \mathsf{a}+\mathsf{a} \succsim \mathsf{b}+\mathsf{b} \leftrightarrow \mathsf{a} \succsim \mathsf{b}.  % this is derivable.
\end{eqnarray*}
% WTS a*a >= b*b -> a >= b
% suppose a*a >= b*b
% then a*a + a*b >= b*b + b*a
% a*(a+b) >= b*(a+b)
% if a+b > 0 done. otherwise 
% a+b ~ 0
% now it should be provable that a = b = 0, by completeness of add.
% WTS a >= b -> a*a >= b*b
% trivial but Canc2 twice.
% WTS 0 >= a*b -> 0>= a or 0>= b.
\end{mdframed}
\caption{$\mathsf{AX}_{\textnormal{poly}}$ is comprised of the axioms here, in addition to the rules and axioms of $\mathsf{AX}_{\textnormal{add}}$ (where schematic term variables run over terms of $\mathcal{L}_{\textnormal{poly}}$).\label{fig:poly}}
\end{figure}

% Finally, the following are all valid because we are restricting the interpretations of terms to \emph{positive} real numbers (no negative-signs in the language): % is this going to work for the polynomials?
% \begin{eqnarray*}
% \mathsf{Canc2}. & & \mathsf{a} \succsim \mathsf{b} \rightarrow \mathsf{a}\cdot\mathsf{c}\succsim\mathsf{b}\cdot\mathsf{c}
% % \mathsf{Sub1}. & & \mathsf{e}+\mathsf{c}\approx\mathsf{a} \rightarrow (\mathsf{a}+\mathsf{d}\succsim\mathsf{b}+\mathsf{c} \leftrightarrow \mathsf{e}+\mathsf{d}\succsim \mathsf{b}) \\
% % \mathsf{Sub2}. & & \mathsf{e}+\mathsf{c}\approx\mathsf{a} \rightarrow (\mathsf{b}+\mathsf{c}\succsim\mathsf{a}+\mathsf{d} \leftrightarrow \mathsf{b}\succsim\mathsf{e}+ \mathsf{d}) \\
% % \mathsf{Elim}. & & (\mathsf{e}+\mathsf{a}\succ \mathsf{b} \wedge \mathsf{c}\succ \mathsf{e}+\mathsf{d}) \rightarrow \mathsf{a}+\mathsf{c} \succ \mathsf{b}+\mathsf{d}
% \end{eqnarray*}

% \begin{lemma}
% 	The following rule is admissible: if $\vdash \tT_1 \equiv \tT_2$ then $\vdash \varphi(\tT_1) \equiv \varphi(\tT_2)$, where $\varphi(\tT_1)$ is any term in which $\tT_1$ appears as a subterm and $\varphi(\tT_2)$ is the same with $\tT_1$ replaced by $\tT_2$.
% \end{lemma}
% \begin{proof}

% \end{proof}

\begin{theorem}
$\mathsf{AX}_{\text{poly}}$ is sound and complete.
\end{theorem}
\begin{proof}
Soundness is again straightforward.
As for completeness, we first obtain a normal form.
We assume, without loss ($\mathsf{Bool}$), that $\varphi$ is a conjunction of literals.
Let $\mathsf{Prop}^{\varphi} \subset \mathsf{Prop}$ be the finite set of letters appearing in $\varphi$, with $\Delta_\varphi$ the set of complete state descriptions of $\mathsf{Prop}^{\varphi}$ (see \S{}\ref{section:notation}). Then:
\begin{lemma}\label{lem:xiocvioxcvix}
Replacement of equivalents $\mathsf{Repl}$ (see Lemma~\ref{lemma:add}) for $\mathcal{L}_{\textnormal{poly}}$ is derivable in $\mathsf{AX}_{\textnormal{poly}}$.
% Given terms $\mathsf{a}, \mathsf{b}, \mathsf{c}$ let $\mathsf{c}_{\mathsf{b}}$ be the result of replacing any occurrences of $\mathsf{a}$ as a subterm in $\mathsf{c}$ with $\mathsf{b}$. Then $\mathsf{a} \approx \mathsf{b} \vdash \mathsf{c} \approx \mathsf{c}_{\mathsf{b}}$.% is admissible in $\mathsf{AX}_{\text{poly}}$.
\end{lemma}
\begin{proof}
It suffices (by induction, and both versions of $\mathsf{Comm}$) to derive $\mathsf{a} \approx \mathsf{b} \rightarrow \mathsf{a} +\mathsf{c} \approx \mathsf{b} + \mathsf{c}$ and $\mathsf{a} \approx \mathsf{b} \rightarrow \mathsf{a} \cdot \mathsf{c} \approx \mathsf{b} \cdot \mathsf{c}$.
The first follows by completeness of $\mathsf{AX}_{\text{add}}$. As for the second: by $\mathsf{NonNull}$ case on $\mathsf{c} \succ \mathsf{0}$ and $\mathsf{c} \approx \mathsf{0}$; use $\mathsf{Canc}$ in the former case and $\mathsf{Zero}$ in the latter case.
\end{proof}
\begin{lemma}\label{lem:oxcivxcovi}
For any $\epsilon \in \sigma(\mathsf{Prop}^{\varphi})$, we have
$\mathsf{AX}_{\text{poly}} \vdash \mathbf{P}(\epsilon) \approx \mathsf{0} + \sum_{\substack{\delta \in \Delta_\varphi \\ \models \delta \rightarrow \epsilon}} \mathbf{P}(\delta)$.
\end{lemma}
\begin{proof}
Using $\mathsf{Repl}$ recursively and instances of $\mathsf{Add}$ (where $\epsilon$ stands for $\alpha$ and a single letter from $\mathsf{Prop}^{\varphi}$ stands for $\beta$) we can show that $\mathbf{P}(\epsilon) \approx \sum_{\delta \in \Delta_\varphi} \mathbf{P}(\epsilon\land \delta)$. Since formulas in $\Delta_\varphi$ are complete state descriptions for the letters appearing in $\epsilon$, we have propositionally $\models (\epsilon \land \delta) \leftrightarrow \delta$ if $\models \delta \rightarrow \epsilon$, and $\models (\epsilon \land \delta) \leftrightarrow \bot$ otherwise. Applying $\mathsf{Dist}$, $\mathsf{Repl}$, and $\mathsf{Zero}$ (from $\mathsf{AX}_{\text{add}}$) we obtain the final result.
\end{proof}
Note that the order and associativity of the sum above generally does not matter, as can be easily seen using (additive) $\mathsf{Comm}$ and $\mathsf{Assoc}$. The same holds for products, so we will simply omit this information where convenient.
Finally, we have a normal form:
\begin{lemma}
Let $\varphi$ be a conjunction of literals. Then there is a collection of polynomial terms $\{\mathsf{a}_i, \mathsf{b}_i, \mathsf{a}'_{i'}, \mathsf{b}'_{i'} \}_{\substack{1 \le i \le m\\ 1 \le i' \le m'}}$ using only the basic probability terms $\{\mathbf{P}(\delta)\}_{\delta \in \Delta_\varphi} \cup \{ \mathsf{0} \}$ such that 
\begin{align}\label{eqn:xcjhwekrnmqnwjkd}
    \mathsf{AX}_{\text{poly}} \vdash \varphi \leftrightarrow \bigwedge_{\delta \in \Delta_\varphi} \mathbf{P}(\delta) \succsim \mathsf{0} \land \sum_{\delta \in \Delta_\varphi} \mathbf{P}(\delta) \approx \mathsf{1} \land \bigwedge_{i=1}^m \mathsf{a}_i \succsim \mathsf{b}_i \land \bigwedge_{i'=1}^{m'} \mathsf{a}'_{i'} \succ \mathsf{b}'_{i'}.
\end{align}
\end{lemma}
\begin{proof}
The first two conjuncts on the right in \eqref{eqn:xcjhwekrnmqnwjkd} are clearly derivable (with or without $\varphi$).
Literals in $\varphi$ are of the form $\mathsf{a} \succsim \mathsf{b}$ or $\lnot (\mathsf{a} \succsim \mathsf{b})$, the latter being $\mathsf{b} \succ \mathsf{a}$; thus the literals in the formula give the terms $\{\mathsf{a}_i, \mathsf{b}_i, \mathsf{a}'_{i'}, \mathsf{b}'_{i'} \}_{\substack{1 \le i \le m\\ 1 \le i' \le m'}}$ in \eqref{eqn:xcjhwekrnmqnwjkd}. By
$\mathsf{Repl}$ (Lemma~\ref{lem:xiocvioxcvix}) and Lemma~\ref{lem:oxcivxcovi}, we may assume without loss that any of these terms uses only basic terms $\{\mathbf{P}(\delta)\}_{\delta \in \Delta_\varphi} \cup \{\mathsf{0}\}$.
% From Lemmas~\ref{lem:oxcivxcovi} and \ref{lem:xiocvioxcvix}: see \cite[Lemma~8]{ibeling2020probabilistic}.
\end{proof}
The conjuncts in \eqref{eqn:xcjhwekrnmqnwjkd} translate into a simultaneous system of polynomial inequalities in the indeterminates $\{\mathbf{P}(\delta)\}_{\delta \in \Delta_\varphi}$.
It is clear that $\varphi$ is satisfiable iff this system has a solution, so let us apply a well-known characterization \citep{Stengle1974}\footnote{Note that the variant here (see ibid., Theorem~4) assumes a polynomial ring over a subfield of the real closed field.} of the (un)feasible sets of related systems:
\begin{theorem}[Real Positivstellensatz]\label{thm:psatz}
Let $R = \mathbb{Q}[x_1, \dots, x_n]$ and $F, G, H \subset R$ be finite. Let $\mathrm{cone}(G)$ be the closure of $G \cup\{s^2 : s \in R\}$ under $+$ and $\times$ and let $\mathrm{ideal}(H) = \{ \sum_{h \in H} a_h h : a_h \in R \text{ for each } h \}$. Then either the system $\{f \neq 0, g \ge 0, h = 0 : (f, g, h) \in (F, G, H) \}$ has a solution over $\mathbb{R}^n$, or there is a polynomial {certificate of infeasibility}: there are $g \in \mathrm{cone}(G)$, $h \in \mathrm{ideal}(H)$, $n \in \mathbb{N}$ such that
\begin{equation*}
    g+h+f^{2n} = 0
\end{equation*}
where $f = \prod_{f' \in F} f'$. \qed
\end{theorem}
Note that the well-known Farkas' lemma of linear programming (obtained via Fourier-Motzkin elimination) is a special case of 
this more general theorem of the alternative, in which the certificate can always be taken to have a certain restricted form.
The following variant is more directly applicable for our purpose:
\begin{corollary}\label{cor:intpsatz}
Suppose above that $R = \mathbb{Z}[x_1, \dots, x_n]$. % (and let $\mathrm{cone}(G)$, $\mathrm{ideal}(H)$ have the same definitions but with coefficients over $\mathbb{Z}$).
Then either the given system has a solution over $\mathbb{R}^n$, or there are $g \in \mathrm{cone}(G)$, $h \in \mathrm{ideal}(H)$, $n \in \mathbb{N}$, $d \in \mathbb{Z}^+$ such that
\begin{equation}\label{eqn:psatzcert}
    g + h + d f^{2n} = 0.
\end{equation}
\end{corollary}
\begin{proof}
In the latter alternative of Theorem~\ref{thm:psatz}, take the certificate and multiply by $d$, setting it to the least common denominator of coefficients in $g$, $h$.
\end{proof}
The normal form \eqref{eqn:xcjhwekrnmqnwjkd} yields $F, G, H \subset \mathbb{Z}^+[\{\mathbf{P}(\delta)\}_\delta]$ for Corollary \ref{cor:intpsatz}.
Note that by translating the conjunct $\sum_{\delta} \mathbf{P}(\delta) \approx 1$ as two inequalities we may take $H = \varnothing$; the strict inequality $\mathsf{a}'_{i'} \succ \mathsf{b}'_{i'}$ gives two inequalities $a'_{i'} - b'_{i'} \ge 0 \in G$ and $a'_{i'} - b'_{i'} \neq 0 \in F$. Here $a'_{i'}, b'_{i'}$ denote the (informal) polynomials translated in the obvious fashion from the respective formal terms $\mathsf{a}'_{i'}, \mathsf{b}'_{i'}$; this convention will be used hereafter.

When translating a polynomial $a$ backward to its formal equivalent denoted $\mathsf{a}$, we will assume $\mathsf{a}$ is a sum of so-called \emph{normal monomials}.
Fixing some total order $\prec_{\Delta_\varphi}$ on $K = \{ \mathbf{P}(\delta) \}_{\delta \in \Delta_\varphi} \cup \{ \mathsf{0}, \mathsf{1} \}$, call a term $\mathsf{t}$ over $K$ a normal monomial if: (1) only multiplication $\cdot$ appears within it; (2) all multiplication is left-associative; (3) for any base terms $\mathbf{P}(\delta)$, $\mathbf{P}(\delta')$ in $\mathsf{t}$, if $\delta \prec_{\Delta_\varphi} \delta'$, then $\mathbf{P}(\delta)$ appears to the left of $\mathbf{P}(\delta')$; (4) $\mathsf{1}$, that is, $\mathbf{P}(\top)$, appears exactly once as a factor in $\mathsf{t}$, and is leftmost; (5) $\mathsf{0}$, that is, $\mathbf{P}(\bot)$, appears once if ever, and if it appears then no other term from $K$ appears.
It is easy to see that any term has an equivalent that is a sum of normal monomials by applying $\mathsf{Dist}$ toward fulfilling (1), $\mathsf{Assoc}$ for multiplication toward (2), $\mathsf{Comm}$ toward (3), $\mathsf{One}$ toward (4), and $\mathsf{Zero}$ toward (5).
Fixing some order on normal monomials, and applying the completeness of $\mathsf{AX}_{\mathrm{add}}$ and again $\mathsf{Zero}$ (for $\mathsf{AX}_{\text{add}}$), we see that we can assume without loss that the sum is in said order and is left-associative, and further that the normal monomial containing $\mathsf{0}$ appears exactly once in it and appears leftmost. We call such a sum the \emph{normal monomial form} of a given term.

Now, suppose we have a certificate $c = g + d f^{2n} = 0$ as in \eqref{eqn:psatzcert} for this system.
Given a nonzero polynomial $a$ let $a^+, a^-$ be the unique polynomials with strictly positive coefficients such that $a = a^+ - a^-$; let $a^+ = a^- = 0$ if $a = 0$.
Let $j = d f^{2n}$, $c_+ = g^+ + j^+$, and $c_- = g^- + j^-$. %; note that at least one of $c_+$ and $c_-$ is nonzero, and therefore both are since $c_+ - c_- = c = 0$.
We claim that $\vdash \eqref{eqn:xcjhwekrnmqnwjkd} \rightarrow \mathsf{c}_+ \succ \mathsf{c}_- \land \mathsf{c}_+ \approx \mathsf{c}_-$.
One can in fact show $\vdash \mathsf{c}_+ \approx \mathsf{c}_-$: assuming normal monomial form without loss, the same normal monomial terms, with the same multiplicities, must appear as summands in $\mathsf{c}_+$ and $\mathsf{c}_-$, because otherwise $c_+ \neq c_-$ (since $\mathsf{0}$ appears exactly once in a monomial term for both sums, it cannot make the difference).
% so, by completeness of $\mathsf{AX}_{\mathrm{add}}$, we have that $\vdash \mathsf{c}'_- \approx \mathsf{c}'_+$ and hence $\vdash \mathsf{c}_- \approx \mathsf{c}_+$.

We claim that $\vdash \eqref{eqn:xcjhwekrnmqnwjkd} \rightarrow \mathsf{g}^+ \succeq \mathsf{g}^-$ while $\vdash \eqref{eqn:xcjhwekrnmqnwjkd} \rightarrow \mathsf{j}^+ \succ \mathsf{j}^-$; again by completeness of $\mathsf{AX}_{\mathrm{add}}$, this gives $\vdash \eqref{eqn:xcjhwekrnmqnwjkd} \rightarrow \mathsf{c}_+ \succ \mathsf{c}_-$.
We make use of the following result.
\begin{lemma}\label{lem:mwqeqhweqhjweqw}
Let $T = \{\mathsf{t}_1, \dots, \mathsf{t}_n\}$ and $T' =\{\mathsf{t}'_1, \dots, \mathsf{t}'_n\}$ be sets of terms and let $S = \{ \prod_{i=1}^n \mathsf{s}_i\}_{\substack{\mathsf{s}_1 \in \{\mathsf{t}_1, \mathsf{t}'_1\} \\ \dots \\ \mathsf{s}_n \in \{\mathsf{t}_n, \mathsf{t}'_n\}}}$.
Let $S^+ \subset S$ and $S^- \subset S$ be the sets of monomial products that have an even and odd number respectively of factors from $T'$. Then
\begin{equation*}
    \mathsf{AX}_{\text{poly}} \vdash \mathsf{t}_1 \succsim \mathsf{t}'_1 \land \dots\land \mathsf{t}_n \succsim \mathsf{t}'_n \rightarrow \sum_{\mathsf{s} \in S^+} \mathsf{s} \succsim \sum_{\mathsf{s} \in S^-} \mathsf{s}
\end{equation*}
as is the analogous rule when every $\succsim$ is replaced by a strict $\succ$.
\end{lemma}
\begin{proof}
Straightforward by inducting on $n$: apply $\mathsf{Sub}$ and $\mathsf{Dist}$, e.g. $\mathsf{a} \cdot \mathsf{c} + \mathsf{b} \cdot \mathsf{d} \succsim \mathsf{a} \cdot \mathsf{d} + \mathsf{b}\cdot \mathsf{c} \land \mathsf{e} \succsim \mathsf{f} \rightarrow \mathsf{a} \cdot \mathsf{c} \cdot \mathsf{e} + \mathsf{b} \cdot \mathsf{d}\cdot \mathsf{e} + \mathsf{a} \cdot \mathsf{d} \cdot \mathsf{f} + \mathsf{b}\cdot \mathsf{c} \cdot \mathsf{f} \succsim \mathsf{a} \cdot \mathsf{c} \cdot \mathsf{f} + \mathsf{b} \cdot \mathsf{d} \cdot \mathsf{f} + \mathsf{a} \cdot \mathsf{d} \cdot \mathsf{e} + \mathsf{b}\cdot \mathsf{c} \cdot \mathsf{e}$.
\end{proof}
To see that $\vdash \eqref{eqn:xcjhwekrnmqnwjkd} \rightarrow \mathsf{g}^+ \succeq \mathsf{g}^-$, note that since $g \in \mathrm{cone}(G)$, it is a sum of terms of the form $g' = k^2 g_1 \dots g_l$, where $g_1, \dots, g_l \in G$ and $k, g_1, \dots, g_l \neq 0$. It thus suffices to show $\vdash \eqref{eqn:xcjhwekrnmqnwjkd} \rightarrow (\mathsf{g}')^+ \succeq (\mathsf{g}')^-$ for any such term $g'$.
Given our construction of $G$ from \eqref{eqn:xcjhwekrnmqnwjkd}, note that the inequality $\mathsf{g}^+_i \succeq \mathsf{g}^-_i$ appears in \eqref{eqn:xcjhwekrnmqnwjkd} for each $g_i$.
Referring to Lemma~\ref{lem:mwqeqhweqhjweqw}, note that if we define $p = (t_1 - t'_1) \dots (t_n - t'_n) \neq 0$, then $p^+ = \sum_{s \in S^+} s$ and $p^- = \sum_{s \in S^-} s$; by casing on whether $\mathsf{k}^+ \succeq \mathsf{k}^-$ or $\mathsf{k}^- \succeq \mathsf{k}^+$, we can apply the Lemma, finding in either case that $\vdash \eqref{eqn:xcjhwekrnmqnwjkd} \rightarrow (\mathsf{g}')^+ \succeq (\mathsf{g}')^-$.

Now, to see that $\vdash \eqref{eqn:xcjhwekrnmqnwjkd} \rightarrow \mathsf{j}^+ \succ \mathsf{j}^-$ it suffices to show that $\vdash \eqref{eqn:xcjhwekrnmqnwjkd} \rightarrow (\mathsf{f}^{2n})^+ \succ (\mathsf{f}^{2n})^-$. If $F = \varnothing$ is empty, then $f = \prod_{f' \in F} f' = 1$ and this is simple; otherwise, for each strict inequality $f \in F$ we have $f = a'_i - b'_i$ for some constraint $\mathsf{a}'_i \succ \mathsf{b}'_i$ in \eqref{eqn:xcjhwekrnmqnwjkd}. If any such $f = 0$ then by employing normal monomial form, we find $\vdash \eqref{eqn:xcjhwekrnmqnwjkd} \rightarrow \bot$; otherwise,
apply the strict variant of Lemma~\ref{lem:mwqeqhweqhjweqw}.
\end{proof}

\subsection{Comparative conditional probability}\label{section: conditional probability}

Recall the language of comparative conditional probability $\mathcal{L}_{\text{cond}}$ (Definition~\ref{def:multlang}), given by 
$$
    \varphi \in \mathcal{L}_{\text{cond}}\iff \varphi = \mathbf{P}(\alpha|\beta) \succsim \mathbf{P}(\delta|\gamma) \, | \, \neg \varphi \, | \, \varphi \land \psi  \text{ for any } \alpha,\beta,\delta, \gamma \in \sigma(\mathsf{Prop}).
$$
This language allows us to reason about comparisons of conditional probabilities. Reasoning about such comparisons plays an important role in a variety of settings. A salient example is probabilistic confirmation theory, where conditional probabilities are interpreted as a measure of the comparative support that some evidence confers to a hypothesis: a comparison of the form $P(H_1|E_1) \succsim P(H_2|E_2)$ is interpreted as the statement that evidence $E_2$ confirms hypothesis $H_2$ at least as much as evidence $E_1$ confirms hypothesis $H_1$. More generally, the focus on conditional probabilities is often motivated by the view---found, for instance, in \cite{Keynes1921} and \cite{Renyi}---that (numerical) conditional probabilities are more fundamental than unconditional probability judgments: similarly, one may want to treat comparisons of conditional probability as more fundamental than comparisons of unconditional probabilities.\footnote{This possibility is suggested, for example, in \cite{Konek}. See also \cite{koopman1940axioms}.}

Although one may be tempted, at first sight, to view the logic of comparative conditional probability as only a minor extension of the logic of (unconditional) comparative probability axiomatized above, it is important to note that moving to $\mathcal{L}_{\text{cond}}$ involves a substantial jump in expressivity. The language $\mathcal{L}_{\text{cond}}$ allows to express non-trivial quadratic inequality constraints\footnote{Consider again the example mentioned in the introduction, which consists of the formula $(A \wedge B)\approx \neg (A \wedge B) \wedge (A|B) \approx B$. This holds in model $(\Omega,\mathcal{F}, \mathbb{P}, \llbracket\cdot\rrbracket)$ only if
\begin{align*}
  \mathbb{P}(\llbracket A \wedge B \rrbracket) &=  1- \mathbb{P}(\llbracket A \wedge B \rrbracket)\\
   \mathbb{P}(\llbracket A \wedge B \rrbracket) &=  \mathbb{P}(\llbracket B \rrbracket)^{2}
\end{align*}
which forces the irrational solution $\mathbb{P}(\llbracket B \rrbracket)= \nicefrac{1}{\sqrt{2}}$.} and, as such, it belongs more naturally to the family of multiplicative systems. As we will now discuss, this comes with a rather significant shift in the complexity of its satisfiability problem as well as its axiomatization. In Section \ref{section: multiplicative complexity} we show that $\mathcal{L}_{\text{cond}}$ and all multiplicative systems we study in this paper have a $\exists \mathbb{R}$-complete satisfiability problem. Before this, we discuss some challenges involved in proving a completeness theorem for $\mathcal{L}_{\text{cond}}$. These difficulties can be traced back to certain delicate questions concerning the canonical representation theorem for conditional comparative probability orders, due to \cite{Domotor}. To the best of our knowledge, Domotor's proposed axiomatization is the only known such representation result for finite probability spaces that does not depend on imposing additional richness constraints on the underlying space. (See \citealt[\S 5.6]{Krantz1971} for an overview of other approaches.) However, as we explain, an important step in Domotor's proof appears to require further justification. Moreover, from the perspective of the additive-multiplicative distinction explored in this paper, Domotor's proof strategy is not fully satisfactory, as it does little to clarify the exact algebraic content of the axioms involved. In light of these observations, we may not be able to rely on Domotor's proposed axiomatization of conditional probability orders to obtain a completeness result for $\mathcal{L}_{\text{cond}}$. The question of axiomatizing $\mathcal{L}_{\text{cond}}$ is thus left open. 

\subsubsection{Conditional probability and quadratic probability structures}

In order to a give a complete axiomatization for the logic of comparative conditional probability, a natural first step is to identify necessary and sufficient conditions for the representation of a comparative order between pairs of events. Suppose we represent comparative conditional judgments with a quaternary relation $\succeq$ on a finite Boolean algebra of events $\mathcal{F}$. We write $A|B \succeq C|D$ when the relation $\succeq$ obtains for the quadruple $(A,B,C,D)$. Such a relation is probabilistically representable if there exists a probability measure $\mathbb{P}$ on $\mathcal{F}$ such that for any $A,B,C,D \in \mathcal{F}$, we have
$$
A|B \succeq C|D \text{ if and only if } \mathbb{P}(A|B) \geq  \mathbb{P}(C|D), 
$$
where $\mathbb{P}(X|Y) := \mathbb{P}(X\cap Y)/\mathbb{P}(Y)$. What properties must a quaternary order $\succeq$ satisfy in order to be representable in this way?

Throughout the literature, the answer to this question is credited to \cite{Domotor}, who proposes necessary and sufficient conditions for such a quaternary order to be representable by a probability measure. %This constitutes the canonical reference for finite conditional probability in the qualitative probability literature: to the authors' knowledge, it is the only known proposal for axiomatizing conditional probability orders without recourse to additional richness assumptions on the underlying space (which sacrifice generality for a simpler proof of representation). 
\begin{definition}[\cite{Domotor}]\label{FQCP}
A \emph{finite qualitative conditional probability structure} (FQCP), is a triple $(\Omega, \mathcal{F},\succeq)$ where $\Omega$ is a finite set, $\mathcal{F}$ a field of sets over $\Omega$ and $\succeq$ a quaternary relation on $\mathcal{F}$, which additionally satisfies the following properties for all $A$, $B$, $C$, $D\in\mathcal{F}$:
\begin{eqnarray*}
%\mathsf{NonZero}. & & \text{$A\,|\,B\succeq C\,|\,D$ holds only if $ B\,|\,\Omega  \succ  \varnothing\,|\,\Omega$ and $ C\,|\,\Omega  \succ  \varnothing\,|\,\Omega$};\\
\mathsf{NonZero}. & & \text{$A\,|\,B\succeq C\,|\,D$ holds only if $B\,|\,\Omega \succ \varnothing\,|\,\Omega$ and $D\,|\,\Omega \succ \varnothing\,|\,\Omega$}
\end{eqnarray*}
Accordingly, in the following, in any expression $A\,|\,B\succeq C\,|\,D$ it is assumed that we are quantifying over $B, D\in\mathcal{F}_{0}$, where $\mathcal{F}_{0}:=\{X\in\mathcal{F}\,:\, X\,|\,\Omega \succ \varnothing\,|\,\Omega \}$:
\begin{eqnarray*}
\mathsf{Tot}. & & \text{either $A\,|\,B \succeq C\,|\,D$ or $C\,|\,D \succeq A\,|\,B$ for all $B$, $D\in \mathcal{F}_{0}$;}\\
\mathsf{NonDeg}. & &  \Omega \,|\,\Omega \succ \varnothing\,|\,\Omega ; \\
\mathsf{NonTriv}. & &  B\,|\,C\succeq \varnothing\,|\,A ;\\
\mathsf{Inter}. & &  A\cap B \,|\,B \succeq A\,|\,B;  \\
\mathsf{FinCanCond}_n. & & \text{if $(A_{i},B_i)_{i\leq n}$ and $(C_{i},D_i)_{i\leq n}$ are balanced\footnotemark\,and $\forall i< n$, $A_{i}\,|\,B_i\succeq C_{i}\,|\,D_i$}, \\
 & & \text{then $C_{n}\,|\,D_n \succeq A_{n}\,|\,B_n$}; \\
\mathsf{MultCan}_n. & & \text{for any permutation $\pi$ on $\{1\dots n\}$:} \\
& &  \text{if }  \textstyle \big(A_k \,\big|\, \bigcap_{0\leq i <k} A_i\big)  \,\,\succeq\,\,   \big(B_{\pi(k)} \,\big|\, \bigcap_{0\leq i < \pi(k)} B_i\big)  \text{ for all $0<k\leq n$, }\\
& & \text{then }
\textstyle \big(\bigcap_{0< i \leq n} A_i \,|\, A_0\big) \succeq \big(\bigcap_{0 <i \leq n} B_{\pi(i)} \,|\,B_0\big); \\
& & \text{Moreover, in the last two schemes above, if $\succ$ holds for any comparison} \\
& & \text{in the antecedent, then $\succ$ also holds in the conclusion.}
\end{eqnarray*}
\end{definition}\footnotetext{In the setting of conditional comparative probability, we say that the two sequences (of pairs of events) $(A_{i},B_i)_{i\leq n}$ and $(C_{i},D_i)_{i\leq n}$ are \emph{balanced} if and only if  $\sum_{i\leqslant n} \ind{A_{i}|B_i} = \sum_{i\leqslant n} \ind{C_{i}|D_i}$, where $\ind{A|B}:\Omega\to\{0,1\}$ is a partial characteristic function,  given by $$\ind{A|B}(\omega) = \begin{dcases} 
1 & \text{if $\omega\in A\cap B$}; \\
0, & \text{if $\omega\in (\Omega\setminus A)\cap B$};  \\
\text{undefined} & \text{if $\omega \not\in B$.}
\end{dcases} $$ The sum $\sum_{i\leqslant n} \ind{A_{i}|B_i}$ is undefined whenever   one of the terms is undefined.}

Domotor's proposed axiomatization of conditional comparative probability relies on an axiomatization of finite \emph{quadratic probability structures}. The strategy consists in proving a representation theorem for quadratic probability structures, giving necessary and sufficient conditions for a binary relation $\succeq$ on $\mathcal{F}^{2}$  to be representable as a product of probabilities, in the sense that there exists a probability measure $\prob$ such that 
\begin{equation}
A\times B \succeq C\times D \,\,\text{if and only if}\,\,
\prob(A)\cdot \prob(B) \geq \prob(C)\cdot \prob(D).    \label{prodrep}
\end{equation}

The representation theorem for conditional comparative probability is based on the representation result for quadratic probability structures: the essential idea is that one can express each comparison $A\,|\,B\succeq C\,|\,D$ as a comparison of products of the form 
\begin{equation}
(A\cap B) \times D \succeq^{\prime} (C\cap D) \times B. \label{condquad}
\end{equation}
The method for constructing of a probability measure that represents these product inequalities in the sense of \eqref{prodrep} yields a probability measure that represents the conditional probability comparisons given by $\succeq$. 

It is informative to sketch the reasoning in a little more detail, in order to highlight both the differences with the representation argument for the unconditional case ($\mathcal{L}_{\text{comp}}$), as well as several points in Domotor's argument that require clarification. Moreover, the axioms that Domotor proposes for quadratic probability structures give a clearer motivation for Definition \ref{FQCP}. They are given below. 

\begin{definition}[\cite{Domotor}] \label{FQPS}
A \emph{finite quadratic probability structure} (FQPS), is a triple $(\Omega, \mathcal{F},\succeq)$ where $\Omega$ is a finite set, $\mathcal{F}$ a field of sets over $\Omega$ and $\succeq$ a binary relation on $\mathcal{F}^{2}$, which satisfies the following properties for all $A$, $B$, $C$, $D\in\mathcal{F}$: 

\begin{itemize}
    \item[$\mathsf{Q1}.$] $\Omega\times \Omega \succ \varnothing\times \Omega;$
 \item[$\mathsf{Q2}.$] $B\times C\succeq  \varnothing\times A;$
 \item[$\mathsf{Q3}.$] $A\times B \succeq B\times A;$
  \item[$\mathsf{Q4}.$] $A\times B \succeq C\times D \text{ or }  C\times D\succeq A\times B;$
   \item[$\mathsf{Q5}_{n}.$] for any permutations $\pi$ and $\tau$ on $\{1\dots n\}$, \\
   if $A_{i}\times B_{i}\succ \varnothing\times \Omega$  \text{and}  $A_{\pi(i)} \times B_{\tau(i)}\succeq A_i \times B_i  \text{ for all }i<n,$ \\
   \text{then}
$A_n \times B_n\succeq A_{\pi(n)} \times B_{\tau(n)} ; $
   \item[$\mathsf{Q6}_{n}.$] $\text{If $(A_{i} \times B_i)_{i\leq n}$ and $(C_{i}\times D_i)_{i\leq n}$ are balanced,\footnotemark and $\forall i< n$, $A_{i}\times B_i\succeq C_{i}\times D_i$},$ \\
 \text{then $C_{n}\times D_n \succeq A_{n}\times B_n$};
 \item[] \text{Moreover, in $\mathsf{Q5}_n$ and $\mathsf{Q6}_n$, if $\succ$ holds for any comparison in the antecedent,} \\
 \text{then $\succ$ also holds in the conclusion.}
\end{itemize}\footnotetext{$(A_{i} \times B_i)_{i\leq n}$ and $(C_{i}\times D_i)_{i\leq n}$ are balanced whenever $\sum_{i\leq n} \ind{A_{i} \times B_i} = \sum_{i\leq n} \ind{C_{i} \times D_i} $.}
\begin{comment}
\begin{eqnarray*}
\mathsf{Q1}. & &  \Omega\times \Omega \succ \varnothing\times \Omega;\\
\mathsf{Q2}. & & B\times C\succeq  \varnothing\times A  ; \\
\mathsf{Q3}. & & A\times B \succeq B\times A; \\
\mathsf{Q4}. & &  A\times B \succeq C\times D \text{ or }  C\times D\succeq A\times B;\\
\mathsf{Q5}_n. & & \text{for any permutations $\pi$ and $\tau$ on $\{1\dots n\}$:} \\
& &  \text{if }  A_i \times B_i \succeq A_{\pi(i)} \times B_{\tau(i)} \text{ for all }i<n, \\
& & \text{then }
A_{\pi(n)} \times B_{\tau(n)}\succeq A_n \times B_n ; \\
\mathsf{Q6}_n. & & \text{If $(A_{i} \times B_i)_{i\leq n}$ and $(C_{i}\times D_i)_{i\leq n}$ are balanced, and $\forall i< n$, $A_{i}\times B_i\succeq C_{i}\times D_i$}, \\
 & & \text{then $C_{n}\times D_n \succeq A_{n}\times B_n$}; \\
& & \text{Moreover, in $\mathsf{Q5}_n$ and $\mathsf{Q6}_n$, if $\succ$ holds for any comparison in the antecedent,} \\
& & \text{then $\succ$ also holds in the conclusion.}
\end{eqnarray*}
\end{comment}
\end{definition}

Note that the axiom scheme $\mathsf{Q6}_n$ is a version of the finite cancellation axiom for $\mathcal{L}_{\text{comp}}$, applied to product sets of the form $A\times B$. Moreover, under the translation \eqref{condquad} above, the axiom scheme $\mathsf{Q6}_n$ corresponds to the scheme $\mathsf{FinCanCond}_n$ for conditional comparative probability. 

By contrast, the axiom scheme $\mathsf{Q5}_n$ captures a version of \emph{multiplicative} cancellation. Given a set of $n$ inequalities $A_{i}\times B_{i} \succeq C_i \times D_i$, say that an event is \emph{cancelled} if it has the same number of occurrences on the left hand side of these inequalities (as $A_i$ or $B_{i}$) as on the right hand side (among $C_i$, $D_i$). In the presence of the symmetry axiom $\mathsf{Q3}$, the axiom $\mathsf{Q5}_n$ asserts that whenever we have sequence of $n$ inequalities where all events are cancelled \emph{except} for $A$ and $B$ on the left-hand side, and $C$ and $D$ on the right-hand side, then we can conclude $A\times B \succeq C\times D$. This is exactly what we would obtain by multiplying the probabilities of all left-hand side terms on the left, and the probabilities all right-hand side terms on the right, and then cancelling any terms occurring on each side. Under the translation $\eqref{condquad}$, $\mathsf{Q5}_n$ corresponds to the axiom $\mathsf{MultCan}_n$ from Definition \ref{FQCP}. 

\subsubsection{Domotor's argument}

The necessity of axioms (schemes) $\mathsf{Q1}$-$\mathsf{Q6}_n$ is easily verified. Domotor (\citeyear[p. 66]{Domotor}) provides an argument to the effect that these axioms are also sufficient for the order $\succeq$ to be product-representable by a probability measure, in the sense of condition \eqref{prodrep}. We will now give a brief description of the proof strategy. This will serve, first, to emphasize the algebraic nature of the problem, which involves proving the consistency of certain polynomial constraints and thus calls for techniques beyond the standard linear algebra involved in the purely additive setting of Theorem \ref{SCOTTTHM}. Secondly, as we will see, there is a step in the argument which suggests that the proof of sufficiency is incomplete as it stands. % Calls for a more explicitly algebraic proof

Suppose we are given a finite quadratic probability structure $(\Omega, \mathcal{F}, \succeq)$ as in Definition \ref{FQPS}: without loss of generality, we will assume we are working with the full powerset algebra $\mathcal{F}=\mathcal{P}(\Omega)$. We want to show that the relation $\succeq$ is product-representable in the sense of \eqref{prodrep}. Consider the space $\mathbb{R}^{n\times n}$ containing all indicator functions $\ind{A\times B}$ of products $A\times B$. Take the space of indicator functions as vectors in $\mathbb{R}^{n\times n}$, where each $\ind{A\times B}$ is the vector $\mb{x}$ where $x_{ij}=1$ exactly if $(\omega_{i}, \omega_{j})\in A\times B$ (that is, we list all elements of $\Omega\times\Omega$ in lexicographic order). Lift the order $\succeq$ from $\mathcal{F}$ to $M=\{\ind{A\times B}\,|\, A\times B\in\mathcal{F}\}$. Now we apply Lemma \ref{replemma}: note that axioms $\mathsf{Q4}$ and $\mathsf{Q6}_n$ give us precisely conditions (a) and (b) in the statement of the Lemma. This gives us the existence of a linear functional $\hat{\Phi}: \mathbb{R}^{n\times n}\to \mathbb{R}$ with $\hat{\Phi}(\ind{A\times B}) \geq \hat{\Phi}(\ind{C\times D})$ whenever $A\times B\succeq C\times D$. This functional is of the form 
\begin{align}
    \hat{\Phi}(\mb{x})= \mathbf{a}^{T}\mb{x}= (a_{11}, a_{12},\dots, a_{1n}, a_{21},\dots a_{nn})  & \begin{bmatrix}
           x_{11} \\
           x_{21} \\
           \vdots \\
           x_{nn}
         \end{bmatrix}
        \end{align}
The existence of such a linear functional already entails that there is a measure $\mu$ on $\mathcal{P}(\Omega\times \Omega)$ that respects the ordering on Cartesian products $\succeq$, by defining $\mu(A\times B):= \hat{\Phi}(\ind{A\times B})/{\hat{\Phi}(\ind{\Omega\times \Omega})}$. But we need to ensure that there a measure $\mu$ representing $\succeq$ which corresponds to the product of a measure $\prob$ on $\mathcal{P}(\Omega)$: i.e. such that $\mu(A\times B)= \prob(A)\cdot \prob(B)$.

We represent the linear functional $\hat{\Phi}$ as a bilinear functional $\Phi:\mathbb{R}^{n}\times \mathbb{R}^{n}\rightarrow \mathbb{R}$. It can be represented as a matrix 
$\Phi(\mathbf{x},\mathbf{y})= \mathbf{x}^{T}\mathbf{M}_{\Phi}\mathbf{y} $ where $\mathbf{M}_{\Phi}$ is a $n\times n$ matrix: intuitively, we want the $(i,j)$-th entry to represent the probability of $\{\omega_{i}\}\times\{\omega_{j}\}$. We write is as 
$$
\Phi(\mathbf{x},\mathbf{y})= \mathbf{x}^{T}\mathbf{M}_{\Phi}\mathbf{y} = (x_{1},\dots,x_{n})\begin{pmatrix}
a_{11} & \dots & a_{1n} \\
\vdots & & \vdots \\
a_{n1} & & a_{nn}
\end{pmatrix} \begin{bmatrix}
           y_{1} \\
           y_{2} \\
           \vdots \\
           y_{n}
         \end{bmatrix}
$$
We then have $\Phi(\ind{A}, \ind{B})= \hat{\Phi}(\ind{A\times B})$. Now, in order for $\Phi$ to give rise a product measure as required, we want to know whether it can be decomposed as the product of linear functionals $f_1$ and $f_2$, so that we can write $\Phi(\ind{A},\ind{B})=f_1(\ind{A})f_2(\ind{B})$. This would suffice, as the symmetry axiom $\mathsf{Q3}$ ensures that $\Phi(\ind{A}, \ind{B})=\Phi(\ind{B}, \ind{A})$ ($\mathbf{M}_{\Phi}$ is symmetric), which would ensure that $f_1=f_2$. Then setting $\prob(A) := f_1(\ind{A}) / f_1(\ind{\Omega})$ would give the desired measure (here $\mathsf{Q1}$ and $\mathsf{Q2}$ ensure it is a non-degenerate measure). A general condition for a bilinear functional being decomposable in this way given by the following standard characterization:\footnote{
A $n\times n$ matrix $\mathbf{M}$ has rank 1 if and only if it can be written in the form $\mathbf{M}= \mathbf{u}\mathbf{v}^{T}$ for two vectors $\mathbf{u}, \mathbf{v}$. In one direction, suppose $\mathbf{M}_{\Psi}$ has rank one, and so is of the form $\mathbf{u}\mathbf{v}^{T} $. Then \begin{align*}
    \Psi( \mb{x},\mb{y}) &= \mb{x}^{T}(\mb{u}\mb{v}^{T})\mb{y} \\
    &= (\mb{x}^{T}\mb{u})(\mb{v}^{T}\mb{y})\\
    &= (\mb{u}^{T}\mb{x})(\mb{v}^{T}\mb{y}),
\end{align*}
so that we define $f_{1}(\mb{x}):=\mb{u}^{T}\mb{x}$ and $f_{2}(\mb{x}):=\mb{v}^{T}\mb{x}$ for the decomposition into linear functionals. Conversely, given two linear functionals generated by respective vectors $\mb{u}$ and $\mb{v}$ in the same fashion, the matrix $\mb{u}\mb{v}^{T}$ generates a decomposable bilinear functional.}
\begin{proposition}
A bilinear functional $\Psi: \mathbb{R}^{n}\times \mathbb{R}^{n}\rightarrow \mathbb{R}$ can be written as $\Psi(\mathbf{x},\mathbf{y})= f_{1}(\mathbf{x})f_{2}(\mathbf{y})$ for two linear functionals  $f_{1}, f_{2}:\mathbb{R}^{n}\rightarrow\mathbb{R}$ if and only if  $\text{rank}(\mathbf{M}_{\Psi})=1$. We then say $\Psi$ is a rank-1 functional.
\end{proposition}

%The exists $\Phi$ as above, whose existence follows from Lemma \ref{replemma}. 

This is where Domotor's argument seems to face a difficulty (or perhaps an omission in the presentation). The proof began by establishing the existence of the bilinear function $\Phi$ which represents the order $\succeq$. At this point the proof proceeds to argue that $\Phi$\textemdash a generic order-preserving linear functional obtained from Lemma \ref{replemma}\textemdash has rank 1.\footnote{See \citep[p.68]{Domotor}.} Thus the strategy seems to be to argue that any such functional representing $\succeq$ has rank 1. This, however, is not the case, as can be seen by a simple example. Say that two functionals $\Phi$ and $\Psi$ are order-equivalent on $\mathcal{P}(\Omega)\times\mathcal{P}(\Omega)$ if we have $\succeq_{\Phi}=\succeq_{\Psi}$, where we define $A\times B\succeq_{\Phi} B\times C$ if and only if $\Phi(\ind{A},\ind{B})\geq  \Phi(\ind{C}, \ind {D})$. We can have two such functionals that are order-equivalent with only one of them having rank 1. For instance, in the case $\Omega=\{\omega_{1}, \omega_{2}\}$, we can take functionals given by matrices
 $$
 \mathbf{M}_{\Phi} = \begin{pmatrix}
1/16 & 3/16 \\
3/16 & 9/16 
\end{pmatrix}
\text{ and }
 \mathbf{M}_{\Psi} = \begin{pmatrix}
1/12 & 3/12\\
3/12 & 5/12
\end{pmatrix}
 $$
Here $\Phi$ is a rank-1 functional, whose order $\succeq_{\Phi}$ is representable by the probability measure $\prob(\{\omega_{1}\})=\nicefrac{1}{4}$, $\prob(\{\omega_{2}\})=\nicefrac{3}{4}$. However, $\Psi$ is evidently not rank-1. Yet the order $\succeq_{\Psi}$ agrees with $\succeq_{\Phi}$, and thus satisfies axioms $\mathsf{Q1}$-$\mathsf{Q6}_n$. The axioms $\mathsf{Q1}$-$\mathsf{Q6}_n$ thus cannot guarantee in general that any linear functional that represents the order has rank 1.

Clarifying this step of the argument (and indeed, determining if this difficulty may be due to a presentational ambiguity, rather than a logical gap) is complicated by the fact that the author does not spell out the decomposability argument in detail. Instead, the argument very briefly appeals to an unstated result in geometry of webs \citep{Aczel}, from which, given the multiplicative cancellation axiom $\mathsf{Q5}_{n}$, decomposability is inferred.\footnote{Unfortunately, we were unable to trace the original article by Acz\'{e}l, Pickert and Rad\'{o}.} 

Returning to the task at hand: we know that there exists a bilinear functional $\Phi$ which represents the order $\succeq$. In order to ensure the order is representable via products of probabilities, we want to show that the axioms guarantee the existence of a bilinear functional \emph{of rank 1} that is order-equivalent to $\Phi$. 

\subsubsection{The (semi-)algebraic perspective on the representation problem}

% ensure that the semi-algebraic constraints expressible by $\mathcal{L}_{\text{cond}}$-consistent sets of formulas are indeed feasible. 

Whether or not the approach via web geometry can be made to succeed in showing the sufficienty of Domotor's axioms, there is a sense in which it not the most natural from the perspective of investigating the additive-multiplicative divide in probabilistic logics. Recall that Scott's representation theorem (Theorem \ref{SCOTTTHM}), and thus the standard completeness proof for $\mathcal{L}_{\text{comp}}$, directly relates the finite cancellation axioms to certificates of inconsistency for linear inequality systems. In the same way, it is of intrinsic interest to pursue a representation theorem for conditional comparative probability that would directly reveal the algebraic content of Domotor's proposed axioms (and axiom $\mathsf{Q5}_n$ in particular, which captures the multiplicative behaviour of conditional probability orders). 

It is thus worth considering a direct algebraic formulation of the problem. Each order $\succeq$ satisfying the axioms $\mathsf{Q1}$-$\mathsf{Q6}_n$ generates a system of polynomial inequalities. For every set $A\subset\Omega$, the polynomial corresponding to $A$ is given by 
$$p_{A}(x_{1},...,x_{n}) = \sum^{n}_{i=1}\ind{A}(\omega_{i}) x_{i}.$$ Now for each inequality $A\times B \succeq C\times D$ in the order, we add a constraint 
$$
p_{A}(\overline{x})p_{B}(\overline{x})- p_{C}(\overline{x})p_{D}(\overline{x})\geq 0,
$$
and similarly for strict inequalities. We need to show the resulting system is consistent whenever $\succeq$ satisfies the axioms $\mathsf{Q1}$-$\mathsf{Q6}_n$. 

Note that this system of inequalities is given by quadratic forms: each polynomial $p_{A}(\overline{x})p_{B}(\overline{x})- p_{C}(\overline{x})p_{D}(\overline{x})$ is homogeneous of degree 2. Formulated in this way, it is clear that proving a representation result would amount to showing that the axioms ensure that the semi-algebraic sets defined by quadratic forms of this type are indeed nonempty. By analogy to the case of $\mathcal{L}_{\text{comp}}$, here it is natural to approach this problem via the  Positivstellsatz (Theorem \ref{thm:psatz}). As we mentioned above, the Positivstellensatz is a semi-algebraic analogue of hyperplane separation theorems, like those used in proving completeness for $\mathcal{L}_{\text{comp}}$. It establishes that there exist certificates of infeasibility of a given form for any infeasible system of polynomials. Just as, in the additive case, any certificate of infeasibility was shown to correspond to a failure of a finite cancellation axiom, so too, one may hope, a failure of the Domotor axioms (and the $\mathsf{Q5}_n$ axiom in particular) can always be extracted from a Positivstellensatz certificate. 

%Show that infeasibility certificates for systems of that kind must have a particular form; then show that from such infeasibility certificates, we can deduce a failure of one of the $\mathsf{Q5}_n$ axioms.
 
%Perhaps, due to the specific structure of the constraints, this strategy may involve proving a form of the Positivstellensatz specialised to inequality systems of that specific kind; which may be of independent interest. 

% SHould the geometric way succeed, its connection to the PSATZ approach might iin itself be informative. 

We leave a definite solution to the representation problem for conditional probability\textemdash and completeness for $\mathcal{L}_{\text{cond}}$\textemdash for further work: as a first step, we show in the Appendix that the axioms for quadratic probability structures are indeed sufficient for representation in the special case of $|\Omega|=2$. The observations in this section motivate further work on determining the adequacy of Domotor's axioms: this is not only to ensure that the canonical axiomatization for conditional probability orders on finite spaces is correct as it stands, but also because, as we suggested above, it is of intrinsic interest to pursue an alternative, algebraically transparent proof of the representation result.

\subsection{Complexity of multiplicative systems}\label{section: multiplicative complexity}

The main result of this section finds a uniform complexity for all of our multiplicative systems:

\begin{theorem}\label{etr-completeness}
$\mathsf{SAT}_{\text{ind}}$, $\mathsf{SAT}_{\text{confirm}}$, $\mathsf{SAT}_{\text{cond}}$, and $\mathsf{SAT}_{\text{poly}}$ are all $\exists\mathbb{R}$-complete.
\end{theorem}

To show the above theorem, we borrow the following lemma from \cite{abrahamsen2018art}:

\begin{lemma}\label{etr-inv}
Fix variables $x_1,...,x_n$, and set of equations of the form $x_i + x_j = x_k$ or $x_i x_j = 1$, for $i,j,k \in [n]$. Let $\exists\mathbb{R}$-inverse be the problem of deciding whether there exist reals $x_1,...,x_n$ satisfying the equations, subject to the restrictions $x_i \in [\nicefrac{1}{2},2]$. This problem is $\exists\mathbb{R}$-complete.
\end{lemma}

We omit the proof of Lemma~\ref{etr-inv}, but at a high level, it proceeds in two steps. First, one shows that finding a real root of a degree-4 polynomial with rational coefficients is $\exists\mathbb{R}$-complete, and then repeatedly performs variable substitutions to get the constraints $x_i + x_j = x_k$ and $x_i x_j = 1$. Second, one shows that any such polynomial has a root within a closed ball about the origin, and then shifts and scales this ball to contain exactly the range $[\nicefrac{1}{2},2]$.

\begin{proof}[Proof of Theorem~\ref{etr-completeness}]
Since 
\begin{equation*}
    \mathsf{SAT}_{\text{ind}} \leq \mathsf{SAT}_{\text{confirm}} \leq \mathsf{SAT}_{\text{cond}} \leq \mathsf{SAT}_{\text{poly}},
\end{equation*}
it suffices to show that $\mathsf{SAT}_{\text{ind}}$ is $\exists\mathbb{R}$-hard and that $\mathsf{SAT}_{\text{poly}}$ is in $\exists\mathbb{R}$. To show the former, we extend an argument given by \cite{mosseibelingicard2021}, and to show the latter, we repeat the proof given by \cite{mosseibelingicard2021} (see also \cite{ibeling2020probabilistic}).

Let us first show that $\mathsf{SAT}_{\text{poly}}$ is in $\exists\mathbb{R}$. Suppose that $\varphi \in \mathcal{L}_{\text{poly}}$ is satisfied by some model $\mathbb{P}$. Using the fact that $\exists\mathbb{R}$ is closed under $\mathsf{NP}$-reductions (\cite{ten2013data}; cf. Definition~\ref{reductions}) it suffices to provide  an $\NP$-reduction of $\varphi$ to a formula $\psi \in \mathsf{ETR}$. Let $E$ contain all $\epsilon$ such that $\mathbf{P}(\epsilon)$ appears in $\varphi$. Then consider the system of equations
\begin{align*}
    \sum_{\delta \in \Delta_{\varphi}} \mathbf{P}(\delta) &= 1\\
    \sum_{\substack{\delta \in \Delta_{\varphi}\\ \delta \models \epsilon}}\mathbf{P}(\delta) &= \mathbb{P}(\epsilon) \text{ for } \epsilon \in E.
\end{align*}
When plugged in for $\mathbb{P}$, the measure $\mathbb{P}$ satisfies the above system, so by Lemma~\ref{linear-small-model}, this system is satisfied by some model $\mathbb{P}_{\text{small}}$ assigning positive (small-sized) probability to a subset $\Delta_{\text{small}} \subseteq \Delta_\varphi$ of size at most $|E| \leq |\varphi|$. Let the set $\Delta_{\text{small}}$ and the model $\mathbb{P}_{\text{small}}$ be the certificate of the $\NP$-reduction.\footnote{Equivalently, imagine the reduction proceeding will all possible choices of $\Delta_{\text{small}}$ and $\mathbb{P}_{\text{small}}$. We will show that some such choice produces a satisfiable formula $\psi$ if and only if $\varphi$ is satisfiable.} The reduction proceeds by replacing each $\epsilon \in E$ in $\varphi$ with the $\delta \in \Delta_{\text{small}}$ which imply it, and then checking whether $\mathbb{P}_{\text{small}}$ is a model of the resulting formula $\psi$. (The size constraint on $E$ ensures that $\psi$ can be formed in polynomial time.)  If $\varphi$ is satisfiable, $\Delta_{\text{small}}$ and $\mathbb{P}_{\text{small}}$ exist, so the reduction of $\varphi$ successfully produces a satisfiable formula $\psi$. Conversely, the success of the reduction with the witnesses $\Delta_{\text{small}}$ and $\mathbb{P}_{\text{small}}$ ensures the satisfiability of $\varphi$, since a model of $\psi$ is a model of $\varphi$.

Let us now show that $\mathsf{SAT}_{\text{ind}}$ is $\ETR$-hard. To do this, consider an $\exists \mathbb{R}$-inverse problem instance $\varphi$ with variables $x_1,...,x_n$. It suffices to find a polynomial-time deterministic reduction to a $\SAT_{\mathrm{ind}}$ instance $\psi$. We first describe the reduction and then show that it preserves and reflects satisfiability.

Corresponding to the variables $x_1,...,x_n$, define fresh events $\delta_1,...,\delta_{n} \in \sigma(\mathrm{Prop})$. Define fresh, disjoint events $\delta_1^\prime,...,\delta_n^\prime$. Let $\psi$ be the conjunction of the constraints
\begin{align*}
    \frac{1}{n} \succsim \mathbf{P}(\delta_i) &\succsim  \frac{1}{4n} & \text{ for }i=1,...,n\\
    \mathbf{P}(\delta_i) \independent \textbf{P}(\delta_j) \land \mathbf{P}(\delta_i \land \delta_j) &= \frac{1}{4n^2} & \text{for } x_i \cdot x_j = 1 \text{ in } \varphi\\
    \textbf{P}(\delta_i^\prime) = \textbf{P}(\delta_i)\land \textbf{P}(\delta_j^\prime) = \textbf{P}(\delta_j)\land\textbf{P}(\delta_i^\prime \lor \delta_j^\prime) &= \mathbf{P}(\delta_k) &\text{for } x_i + x_j = x_k \text{ in } \varphi.
\end{align*}

The formula $\psi$ is not yet a formula in $\mathcal{L}_{\text{ind}}$, since it features constraints of the form $~\mathbf{P}(\alpha)~\succsim~\nicefrac{1}{N} $ and $\nicefrac{1}{N}~\succsim~\mathbf{P}(\alpha) $. For any constraint of the form $ \mathbf{P}(\alpha)~\succsim~\nicefrac{1}{N}$, replace $\nicefrac{1}{N}$ with $\mathbf{P}(\epsilon_N)$, replace $\alpha$ with $\alpha \lor \epsilon_N$, and require that the fresh events $\epsilon_1,...,\epsilon_N$ are disjoint with $\textbf{P}(\lor_i \epsilon_i)=1$ and $\textbf{P}(\epsilon_i) = \textbf{P}(\epsilon_j)$ for $i =1,...,N$. Similarly, for any constraint of the form $\nicefrac{1}{N}~\succsim~\mathbf{P}(\alpha) $, replace $\nicefrac{1}{N}$ with $\mathbf{P}(\epsilon_N^\prime)$, replace $\alpha$ with $\alpha \land \epsilon_N^\prime$, and require that the fresh events $\epsilon_1^\prime,...,\epsilon_N^\prime$ are disjoint with $\textbf{P}(\lor_i \epsilon_i^\prime)=1$ and $\textbf{P}(\epsilon_i^\prime) = \textbf{P}(\epsilon_j^\prime)$ for $i =1,...,N$.

This completes our description of the reduction. The map $x_i \mapsto x_i/2n$ sends satisfying solutions of $\varphi$ to those of $\psi$, and the inverse map $\mathbb{P}(\delta_i) \mapsto \mathbb{P}(\delta_i) \cdot 2n$ sends satisfying solutions of $\psi$ to those of $\varphi$. Further, the operations performed are simple, and the introduced events $\delta_i, \delta_i^\prime, \epsilon_i$ and the constraints containing them are short, so the reduction is polynomial-time.
\end{proof}

Some authors have suggested that probability logic with conditional independence terms may be a useful compromise between additive languages built over linear inequalities and the evidently more complex polynomial languages (see, e.g., \citealt{Ivanovska}). However, the above result shows that even allowing simple independence statements among events (not to mention \emph{conditional} independence statements among \emph{sets of variables}) results in $\exists\mathbb{R}$-hardness. Thus while probability logic with conditional independence seems on its face to offer a compromise, it in fact introduces (at least) the complexity of the maximally algebraically expressive languages considered in this paper.

The above result shows that reasoning in the seemingly simpler systems $\mathcal{L}_{\text{ind}}$ and $\mathcal{L}_{\text{confirm}}$ are just as complex as $\mathcal{L}_{\text{poly}}$, because the former systems allow for the expression of independence statements. We conclude with two observations relating to the above result. First,  a minimal extension of $\mathcal{L}_{\text{comp}}$, discussed by \cite{fine1973theories}, remains $\mathsf{NP}$-complete, even though it includes some mention of conditional probability:

\begin{definition}
Fix a nonempty set of proposition letters $\mathsf{Prop}$. The language $\mathcal{L}_{\text{same cond}}$ is defined:
\begin{align*}
    \varphi \in \mathcal{L}_{\text{same cond}} &\iff \varphi = \mathbf{P}(\alpha|\beta) \succsim  \mathbf{P}(\alpha^\prime|\beta) \, | \, \neg \varphi \, | \, \varphi \land \psi & \text{ for any } \alpha, \alpha^\prime, \beta \in \sigma(\mathsf{Prop})
\end{align*}
\end{definition}

\begin{fact}
$\mathsf{SAT}_{\text{same cond}}$ is $\mathsf{NP}$-complete.
\end{fact}

\begin{proof}
Hardness follows immediately from Theorem~\ref{additive-complexity} and the observation that $\mathcal{L}_{\text{same cond}}$ is at least as expressive as $\mathcal{L}_{\text{comp}}$. To show completeness, take any $\varphi \in \mathcal{L}_{\text{same cond}}$, and let $\psi$ be the result of replacing each term $\mathbf{P}(\alpha | \gamma)$ in $\varphi$ with $\mathbf{P}(\alpha \land \gamma)$. We claim that $\varphi$ and $\psi$ are equisatisfiable. Indeed, for any measure $\mathbb{P}$, we have:
\[
 \mathbb{P}(\alpha \land \gamma) \geq \mathbb{P}(\beta \land \gamma) \iff  \mathbb{P}(\alpha | \gamma) \geq \mathbb{P}(\beta| \gamma).
\]
Thus the inequalities mentioned in $\varphi$ hold precisely when those mentioned in $\psi$ hold.
\end{proof}
% \begin{remark} \cite{fine1973theories} considers $\mathsf{SAT}_{\text{same cond}}$ in section IIE of his \emph{Theories of Probability}. He gives a simple axiomatization with one new axiom on top of the finite cancellation scheme for each conditioning event individually: $$\bot \succsim (\gamma \wedge \neg \delta) \rightarrow \big(  (\alpha|\gamma) \succsim (\beta|\gamma) \leftrightarrow (\alpha \wedge \gamma | \delta) \succsim (\beta \wedge \gamma|\delta) \big) .$$ % $$\bot \succsim (\gamma \wedge \neg \beta) \rightarrow \big(  (\alpha|\gamma) \succsim (\delta|\gamma) \leftrightarrow (\alpha \wedge \gamma | \beta) \succsim (\delta \wedge \gamma|\beta) \big) .$$
% \end{remark}

Second, whereas the above theorem characterizes the complexity of reasoning about independence among events, the following result due to \cite{lang2002conditional} characterizes the complexity of reasoning about independence among two random variables:

\begin{theorem}[\cite{lang2002conditional}]
Determining independence of two random variables is complete for the complexity class $\mathsf{coDP}$, the complement of $\mathsf{DP}$, the class of all languages $\mathcal{L}$ such that $\mathcal{L} = \mathcal{L}_1 \cap \mathcal{L}_2$, where $\mathcal{L}_{1}$ is in $\mathsf{NP}$ and $\mathcal{L}_{2}$ is in $\mathsf{coNP}$ (the complement of $\mathsf{NP}$).
\end{theorem}

\noindent \cite{lang2002conditional} also characterize the complexity of several other tasks concerning the (conditional) independence of sets of random variables.

\section{Summary: The Additive-Multiplicative Divide}\label{section: Additive-Multiplicative Divide}

We identified an important dividing line in the space of probability logics, based on the distinction between purely additive and multiplicative systems. Additive systems can encode reasoning about systems of linear inequalities; multiplicative systems can encode reasoning about polynomial inequality systems that include at least quadratic constraints. The distinction between additive and multiplicative systems robustly tracks a difference in computational complexity: while the former are $\mathsf{NP}$-complete, the latter are $\exists\mathbb{R}$-complete. As a consequence, inference involving (implicitly) multiplicative notions is inherently more complex (assuming $\mathsf{NP}\neq \exists\mathbb{R}$): this applies to various intuitively `qualitative' systems for reasoning 
about independence ($\mathcal{L}_{\text{ind}}$), confirmation ($\mathcal{L}_{\text{confirm}}$), or comparisons of conditional probability ($\mathcal{L}_{\text{cond}}$). 

As we saw, proving completeness for the additive and multiplicative systems involves different methods. While completeness for additive systems relies on linear algebra (and hyperplane separation theorems or variable elimination methods), the natural mathematical setting for multiplicative systems is semialgebraic geometry (and completeness relies on versions of the real Positivstellensatz). %algebraic content of axiomatizations?

Importantly, both in the additive and multiplicative settings, systems with explicitly `numerical' operations are more expressive and admit finite axiomatizations, while incurring no cost in complexity. By contrast, even the most paradigmatically `qualitative' logic of unconditional comparative probability ($\mathcal{L}_{\text{comp}}$) is not finitely axiomatizable.\footnote{A similar phenomenon occurs in other non-numerical systems for probabilistic reasoning. See, for example, results on the non-finite axiomatizability of conditional independence for discrete random variables \citep{Studeny} and for Gaussian random variables \citep{Sullivant}.} Thus, from a logical perspective, there is little to be gained from restricting attention, in applications of probability logics, to syntactically `qualitative' systems without arithmetical operations.

\begin{figure}[t]
\begin{center}
    \begin{tikzpicture}
      \node[vertex] (b0) at (-2,0.25) {(additive)};
       \node[vertex] (b1) at (-2,1.75) {(multiplicative)};
        \node[vertex] (b3) at (4,0.25) {$\mathsf{NP}$-complete};
         \node[vertex] (b4) at (4,1.75) {$\exists \mathbb{R}$-complete};

           \node[vertex] (c1) at (-3,1) {};
         \node[vertex] (c2) at (5,1) {};
         
            \node[vertex] (d1) at (1,3) {};
         \node[vertex] (d2) at (1,-1) {};
         
             \node[vertex] (d3) at (-2,2.5) {\textcolor{blue}{\small (no explicit arithmetic)}};
         \node[vertex] (d4) at (3.5,2.5) {\textcolor{blue}{\textcolor{blue}{\small (explicit arithmetic)}}};
      
    \node[vertex] (a1) at (0,0) {$\mathcal{L}_{\mathrm{comp}}$};
    \node[vertex] (a2) at (0,2) {$\mathcal{L}_{\mathrm{cond}}$};
    \node[vertex] (a3) at (2,0) {$\mathcal{L}_{\mathrm{add}}$};
    \node[vertex] (a4) at (2,2) {$\mathcal{L}_{\mathrm{poly}}$};
    
    \draw[edge] (a1) -- (a2);
    \draw[edge] (a1) -- (a3);
    \draw[edge] (a2) -- (a4);
    \draw[edge] (a3) -- (a4);
    
     \draw[thick, dashed] (c1) -- (c2);
 \draw[dashed, blue] (d1) -- (d2);
    \end{tikzpicture}
\end{center}
\caption{The additive-multiplicative divide.}
\end{figure}
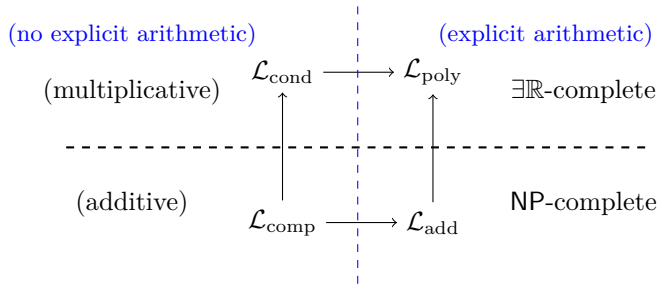

These results also illustrate how ease of elicitation and complexity of inference might come apart. As we noted, the use of comparative probability is sometimes motivated by the view that `qualitative' judgments are more intuitive, or easier to elicit, than explicitly `quantitative' ones. While these claims are somewhat difficult to substantiate (see the discussion below in Section \ref{section:empirical}), our results give a concrete sense in which intuitions about the ease of elicitation of certain comparative judgments do not reflect the complexity of \emph{inference} involving these judgments. Consider, for example, the case of $\mathcal{L}_{\text{cond}}$ and $\mathcal{L}_{\text{add}}$.  While at first blush there may be something more immediate about comparisons of conditional probabilities that are not explicitly numerical, as opposed to comparisons expressible in $\mathcal{L}_{\text{add}}$ (`is $A$ twice as likely as $B$?'), our results suggest that reasoning with the former is more complex than reasoning with the latter. 

While the distinction between additive and multiplicative systems is an informative dividing line that is useful in classifying the landscape of probability logics, investigating this divide also illustrates that the very distinction between qualitative and quantitative  reasoning remains somewhat elusive. Certainly, \emph{prima facie} natural ways to formulate the distinction, based on the simplicity and intuitiveness of comparative judgments, or on the explicit presence of arithmetical operators, do not seem to capture any clear or robust distinction that tracks properties of logical interest such as complexity, expressivity, and axiomatizability. We are thus left with the question of how to give concrete substance to the often invoked, but never delineated, qualitative/quantitative distinction: are there any structural properties of inference that are characteristic of \emph{qualitative} reasoning in probabilistic contexts? We turn to this question next.

\section{What is the Quantitative-Qualitative Distinction?}\label{section: Quantitative-Qualitative Distinction}

Before concluding we briefly consider the larger conceptual questions with which we began. How might we understand the prevalent distinction between quantitative and qualitative formulations of probabilistic principles and inference patterns, particularly in light of the landscape of systems we have explored in the present work? We begin by entertaining several suggestions from the literature. %\textcolor{blue}{Should we mention Aristotle or Kant somewhere? Quantity/Quality. Church 1942 (Dictionary of Philosophy) notes that traditional propositions were distinguished by quality---affirmative or negative---and quantity---particular, singular, or universal. Does this have anything to do with our topic here?}

\subsection{Previous Suggestions} As briefly discussed in the introduction, gestures toward a distinction between qualitative and quantitative formulations of probability can be found throughout the literature, going back at least as far as \cite{Keynes1921}. In the seminal work by de Finetti on comparative probability, the express goal was to `start out with only qualitative notions' before `one arrives at
a quantitative measure of probability' \citep[p. 101]{Finetti1937}. However, while the distinction often arises as informal motivation, there has been less explicit discussion of what exactly the distinction might be. A survey of the literature reveals two families of proposals. \emph{Syntactic} proposals locate the distinction in formal syntax, whereas \emph{semantic} proposals focus instead on the variety of models a system admits. 
\subsubsection{Syntactic Proposals}
The passage quoted above from \cite{Finetti1937} invites a view on which the distinction tracks whether \emph{numbers}, or more generally \emph{arithmetical concepts}, appear explicitly in our formal language.  Comparative probability languages like $\mathcal{L}_{\text{comp}}$ and perhaps also $\mathcal{L}_{\text{cond}}$, on this view, are typically qualitative, since there are no numerical terms or operations. In a recent paper,  \cite{Delgrande} state this view  clearly: \begin{quote}
 `What distinguishes qualitative from quantitative probability (truth valued) logics is that qualitative probability logics do not employ quantities or arithmetic operations in the syntax, and the informal reading of the qualitative probability formulas do not require a quantitative interpretation.'
\end{quote} On this picture, comparative judgments do not involve any explicit reference to numbers, so such systems would count as qualitative. By contrast, a language like $\mathcal{L}_{\text{add}}$ would presumably be considered quantitative, since it involves an explicit addition-like operator.\footnote{Notably, the system introduced by \cite{Delgrande} also employs an `addition-like' operator, allowing for a simple finite axiomatization of what is essentially our language $\mathcal{L}_{\text{add}}$. The target in this paper rather appears to be the extension from \cite{fagin1990logic} that includes explicit constant terms for integers.}

While there may be extra-logical reasons to focus attention on qualitative systems in this sense, such restrictions come at a logical price. At no increase in reasoning complexity, the presence of arithmetical functions affords simple finite axiomatizations, as well as greater expressivity. We return to potential extra-logical, viz. empirical, motivations below in \S\ref{section:empirical}.

\subsubsection{Semantic Proposals}
An alternative way of drawing the distinction appeals not to the syntax, but to the semantics of the system. The quotation from \cite{Delgrande} also gestures at such a view, according to which qualitative systems do not \emph{require} a quantitative interpretation. The suggestion seems to be that qualitative systems are sufficiently general that they also admit interpretations that do not involve numbers in any explicit way.

All of the systems we studied here are interpreted relative to a probability space $(\Omega,\mathcal{F},\mathbb{P})$. But some of them also admit alternative interpretations. Take, for instance, $\mathcal{L}_{\text{add}}$. Consider any totally ordered commutative monoid $(M;\oplus,\sqsupseteq)$ that also satisfies the following two conditions:
\begin{enumerate}
    \item Double Cancellation: if $a \oplus e \sqsupseteq c\oplus f$ and $b \oplus f \sqsupseteq d \oplus e$, then $a\oplus b \sqsupseteq c \oplus d$;
    \item Contravenience: if $a \oplus b \sqsupseteq c \oplus d$ and $d \sqsupseteq b$, then $a \sqsupseteq c$.
\end{enumerate}
Then our earlier results imply:
\begin{corollary} \label{cor:gen}
Consider any mapping $\semantics{\cdot}$ from $\sigma(\mathsf{Prop})$ to  a totally ordered commutative monoid $(M;\oplus,\sqsupseteq)$. If $\semantics{\cdot}$ also satisfies the following three conditions: 
\begin{enumerate}
    \item Monotonicity: $\semantics{\alpha} \sqsupseteq \semantics{\beta}$, whenever $\models \beta \rightarrow \alpha$,
    \item Non-triviality: $\semantics{\bot} \not\sqsupseteq \semantics{\top}$, and 
    \item Additivity: $\semantics{\alpha \vee \beta} = \semantics{\alpha} \oplus \semantics{\beta}$, whenever $\models \beta \rightarrow \neg \alpha$
\end{enumerate} then we will obtain completeness for the system $\mathsf{AX}_{\text{add}}$ (and hence also $\mathsf{AX}_{\text{comp}}$).
\end{corollary}
 The assumptions above are enough to guarantee soundness of $\mathsf{AX}_{\text{base}}$, as well as axioms $\mathsf{Add}$, $\mathsf{2Canc}$ and $\mathsf{Contr}$. The remaining three axioms of $\mathsf{AX}_{\text{add}}$ follow from the fact that $M$ is a totally ordered commutative monoid. Completeness will follow immediately, since we always have a countermodel in the non-negative rationals $(\mathbb{Q}^+;+,\geq)$, by Theorem \ref{thm:add-complete}. The same obviously applies to the smaller system,  $\mathsf{AX}_{\text{comp}}$.\footnote{This observation mirrors the classic result of \cite{Gurevich}, showing that every two (totally) ordered Abelian groups have the same existential (or universal) first-order theory.}

However, there are also other commutative monoids that do not explicitly involve numbers but still satisfy Double Cancellation and Contravenience. For example, we could take $M = \{a\}^*$ to be the set of all finite strings over a unary alphabet (the `free monoid' over $\{a\}$), with $\oplus$ string concatenation and $\sqsupseteq$ the relation of string containment. Since $x \mapsto kx$ for any positive integer $k$ is an embedding of $(\mathbb{Q}^+;+,\geq)$ to itself, we can always find countermodels in $(\mathbb{N};+,\geq)$, which is in turn isomorphic to $(\{a\}^*;\oplus,\sqsupseteq)$, so we have:
\begin{corollary} $\mathsf{AX}_{\text{add}}$ (and  $\mathsf{AX}_{\text{comp}}$) is complete with respect to  interpretations in $(\{a\}^*;\oplus,\sqsupseteq)$.
\end{corollary}
On this alternative interpretation, the `probability' of an event is taken to be simply a string in unary.

To the extent that the multiplicative systems do not enjoy such alternative interpretations, our division appears to harmonize with this distinction. We leave as an open question whether systems like $\mathsf{AX}_{\text{poly}}$ can be interpreted in models that are not (isomorphic to some sub-semiring of) the real numbers or the unit interval.

A more radical way of drawing the distinction is to insist not that a system possess non-numerical models in order to be qualitative, but that the system have no straightforward models that \emph{are} numerical. For instance, in the literature on uncertain reasoning, systems that license inferences like $A,B  \; |\!\!\!\sim A \wedge B$ have been deemed qualitative, since they correspond to intuitive, non-numerical patterns, but are incompatible with straightforward probabilistic interpretations (e.g., according to which $A$ is accepted just in case $\mathbb{P}(A)>\theta$ for some threshold $\theta$). For instance, \cite{HawthorneMakinson} give voice to this perspective when they write:
\begin{quote} `Broadly speaking, there are two ways of approaching the formal analysis
of uncertain reasoning: quantitatively, using in particular probability relationships, or by means of qualitative criteria. As is widely recognized, the
consequence relations that are generated in these two ways behave quite
differently.' \end{quote} Systems of non-monotonic reasoning can often be given quantitative probabilistic interpretations (see, e.g., \citealt{Pearl1989}), and there are various ways of ameliorating the inferential tensions in these contexts \citep{Leitgeb,Mierzewski}. But such resolution usually comes at the expense of quantitative granularity typical of numerical probabilistic reasoning. In any case, on this way of drawing the distinction, none of the systems we have studied in the present work would count as qualitative, even the basic comparative system $\mathsf{AX}_{\text{comp}}$.

\subsection{Empirical Issues} \label{section:empirical}
Some of the motivations for less quantitative formulations of probability come not from issues of logic and complexity, but rather from empirical concerns. A common thought is that eliciting comparative judgments may be in some way more tractable. Relatedly, numerous authors have suggested that quantitative judgments may not always be empirically meaningful in the same way that qualitative judgments may be. %We consider each of these in turn.

The intuition behind this motivation is clear enough. Comparative judgments introspectively appear easier to make than numerical comparisons, and echoing the earlier suggestions by Keynes, Koopman and others, some more recent authors have concluded that they are in some sense more psychologically `real' (e.g., \citealt{Stefansson}).  Indeed, it has long been appreciated that binary comparative judgments in general can be more stable or reliable than absolute judgments, even to the extent that some recent researchers advocate replacing the latter with the former to mitigate noise in judgment \citep{Kahneman}. Such judgments play a central role in algorithms for probabilistic inference as well, under the assumption that probabilistic comparisons---especially between hypothesis that are `nearby' in a larger state space---are relatively easy \citep[p. 887]{sanborn2016bayesian}.

Similar arguments about the ease and naturalness of qualitative judgments have been marshalled for conditional independence relationships, purportedly arising from even more fundamental qualitative causal intuitions \citep[p. 21]{Pearl2009}.

%\textcolor{blue}{A mention here that MCMC (e.g., Metropolis-Hastings) relies on the assumption that comparative probabilistic judgments---in this case between the probability of a current value and that of a new proposal---are relatively easy?}

\subsubsection{Direct Measurements of Probability}

For the specific problem of eliciting subjective probabilities, not only is stability across time and contexts important---it is also significant that the whole pattern of attitudes be consistent with at least one probabilistic representation in the first place. It was pointed out already by \cite{Suppes1974} that verifying this in the purely comparative setting, even for a small number of basic events, involves a combinatorial explosion of pairs to check. Worse yet, we now have overwhelming evidence (e.g., from the long line of work starting with \citealt{TverskyKahneman1974}) that ordinary judgments about comparative probability routinely violate even the most basic axioms, such as $\mathsf{Dist}$. Similar experimental patterns confront the logic of (conditional) independence (e.g., \citealt{Rehder}). These considerations, at the very least, put pressure on any claim to the effect that qualitative judgments enjoy special empirical tractability. 
%Contemporary methods for probability elicitation in the behavioral sciences 

Once we abandon the ambition of eliciting coherent, fully specified patterns of judgments, the numerical/non-numerical distinction again begins to appear somewhat arbitrary. Contemporary methods for probability elicitation tend to be partial, and they handle numerical judgments in a similar way to their treatment of purely comparative judgments (e.g., \citealt{Lambert}). In the behavioral sciences, probabilities are commonly measured on continuous sliding scales, or on 7-point Likert scales (e.g., `extremely unlikely' to `extremely likely'), often with a background assumption that such responses will be noisy reflections of an underlying psychological mechanism, estimable from many samples of the population (see, e.g., \citealt{LuceSuppes,Icard2016,sanborn2016bayesian} for discussion). From this perspective, comparative judgments may  tend to be robust because they are often relatively insensitive to random perturbations. At the same time, we might expect a claim that `$A$ and $B$ are \emph{equally likely}' to be more fickle than a claim like, `$C$ is \emph{more than twice as likely} as $D$'. So there again may be nothing distinctive  about non-numerical comparisons in this regard. 

\subsubsection{Indirect Measurement via Preference}

A prominent way of thinking about subjective probability takes it to be not a primitive notion, but rather derivative from an agent's preferences over `gambles' or `acts'.  On this alternative view such preferences, as  revealed in choice behavior, form the empirical basis of probability attributions. As long as the pattern of choices satisfies certain sets of axioms, a representation in terms of (fully quantitative, and typically unique) probabilities together with utilities is guaranteed (e.g., as in Savage's \citeyear{Savage} classic axiomatization). These axioms tend to be quite strong, and there have been countless criticisms of them, on both descriptive and normative grounds. Though interesting variants and weakenings have been proposed (e.g., \citealt{Joyce1999,GaifmanLiu}), it appears that the assumptions required will be at least as demanding as those made in the purely probabilistic case. 

Indeed, in the preferential setting there is a precise sense in which comparative probability judgments emerge as a special case of uncertain gambles. We would say that an agent considers $\alpha$ more likely than $\beta$ if they prefer a gamble that returns a good outcome in case of $\alpha$ to one that returns the same good outcome in case of $\beta$.  Axioms can be stated to guarantee that the resulting order will be probabilistically representable (see \citealt[\S 5.2.4]{Krantz1971}), though, yet again, even the most basic of these have been questioned. For instance, in Ellsberg's \citeyearpar{Ellsberg} celebrated counterexample to Savage's sure-thing principle, people tend to prefer a gamble on $A$ to one on $B$, while simultaneously preferring a gamble on $B \vee C$ to one on $A \vee C$, where $C$ is incompatible with $A$ and with $B$. This leads to a blatant violation of quasi-additivity ($\mathsf{Quasi'}$).

Nonetheless, if one is willing to weaken the axioms required to guarantee probabilistic representability (or simply disregard their systematic violation), these basic gambles can be elaborated to extract probability judgments with greater numerical content, including ratio comparisons (typical of $\mathcal{L}_{\text{add}}$ and beyond). The basic idea, following \cite{Ramsey} (see also \citealt{DavidsonSuppes,ElliotMind}, \emph{inter alia}), is to begin with a way of eliciting meaningful utilities for outcomes, and then to use these utilities to measure probability judgments. For instance, if it can be determined that an agent judges outcome $O_2$ be to at least twice as desirable as $O_1$, then someone who prefers a gamble returning $O_1$ if $\alpha$, to one returning $O_2$ if $\beta$, might be taken to judge $\alpha$ more than \emph{twice as likely} as $\beta$.

Could such methods be employed to determine not just additive constraints on probabilities, but also multiplicative constraints? Supposing we could establish that the utility of $O_2$ were greater than that of $O_1$ \emph{squared}, for instance, we might conclude that $\alpha$ has probability greater than the square of $\beta$'s probability. The problem with this suggestion is that most approaches to utility, including those that descend from \cite{Ramsey}, assume that utilities are meaningful only up to linear transformation, so such conclusions would not be well defined. Thus, even in this context of empirical questions, we see that the natural dividing line may not be numerical or quantitative constraints per se, but rather additive versus multiplicative constraints.

\section{Conclusion and Open Questions} \label{section:conclusion}

Through a broad distinction between additive and multiplicative formalisms for probabilistic reasoning, we have explored the landscape of probability logics with respect to fundamental questions about expressivity, computational complexity, and axiomatization. What emerges is a remarkably robust divide that cross-cuts some tempting ways of dividing the space based on intuitions about quantitative versus qualitative representation and reasoning. In addition to the technical contributions summarized above in \S\ref{section: Additive-Multiplicative Divide}, we also canvassed some of the empirical considerations that have motivated special attention to comparative and other `purely qualitative' judgments. At present it is not clear that the latter enjoy any distinctive empirical status, and if anything, the relevant empirical boundaries may better track the additive-multiplicative distinction that has been our focus here. 

It is certainly not an aim of this paper to discourage the use and exploration of qualitative probability. On the contrary, as we have discussed, such systems present rich opportunities for systematic logical investigation. Moreover, a number of recent authors have found such languages useful for stating and evaluating epistemological principles in a general and relatively neutral manner (e.g., \citealt{narens07,Eva2019,LiuBJPS,Mayo}, \emph{inter multa et alia}). This is especially evident in settings where agreement with a probability measure is not assumed, or is even precluded (e.g., \citealt{Dubois1988,diBella}). These considerations are rather different from many of the original motivations that  prompted study of such systems, viz. simplicity and empirical tractability. From a logical perspective---concerning complexity, axiomatizability, and expressivity---and potentially also from an empirical perspective, prohibition on the use of simple numerical primitives appears Procrustean. 

A number of significant technical questions and directions on this subject merit further exploration. In addition to the outstanding issue of substantiating Domotor's \citeyearpar{Domotor} argument for representation of comparative conditional probability (represented here as $\mathcal{L}_{\text{cond}}$), we also mention the following directions:
\begin{itemize}
    \item Many authors have argued that the comparability of $\succsim$ should be rejected, on both normative and empirical grounds \citep{Keynes1921,fine1973theories,Walley,Gaifman2009,LiuBJPS}. All of the systems presented in this paper could alternatively be interpreted relative to \emph{sets} of probability measures. Logical systems for such settings have received considerable investigation (see, e.g., \citealt{Ding2021} for an overview). It seems plausible that many of the results here would extend with little change; however, our proof methods often employed normal forms whose exact character depended on comparability. Working through this setting in detail would be worthwhile. 
    \item It is easy to imagine natural extensions of the languages considered here, including varieties of explicit quantification. For instance, \cite{Speranski} studies probability logics extending (what we here called) $\mathcal{L}_{\text{add}}$ with quantification over events, while \cite{Abadi} explore a range of \emph{first order} extensions of $\mathcal{L}_{\text{add}}$ and (what we called) $\mathcal{L}_{\text{poly}}$ with quantification over both field terms and (a second sort of) objects. Generally speaking, these extensions result in a significant complexity increases, leading to undecidability (in the best cases), $\Pi^1_1$-hardness, and even $\Pi^1_{\infty}$-hardness. However, there are ways of curtailing this complexity (e.g., by restricting to bounded finite domains), and it could be enlightening to investigate our more refined space of languages in those settings. 
    \item A potentially more modest, but quite useful extension to any of these systems would be to add \emph{exponentiation}, e.g., the  function $e^x$. As such systems would also encode logarithms, this would allow reasoning about (conditional) entropy as well. It is unclear at present whether adding exponentiation to $\mathcal{L}_{\text{poly}}$ would even result in a decidable system. By a theorem of \cite{Wilkie}, a positive answer to this question would also settle Tarski's \citeyearpar{Tarski1951} well known `exponential function problem' which has been open since the 1940s. Whether the problem may be easier for some weaker systems we have considered here remains to be seen. Moreover, it may be feasible to address questions of axiomatizability without solving that outstanding open problem. 
 %   \item \textcolor{blue}{Dynamic logic (e.g., PDEL)? Causality?}
\end{itemize}
These are just some of the questions that will need to be answered before we have a fully comprehensive understanding of the space of natural probabilistic languages.

%\section{Questions}

%\begin{itemize}
%    \item Quantifer-free theory of addition: finite axiomatization?
%    \item Psatz calculus; modular presentation
%    \item Scott axioms derivation from fully additive systems (normal form?)
%\end{itemize}

\bibliographystyle{apalike}
\bibliography{refs}{}

% \newpage
\appendix

\section*{Appendix: Finite Cancellation and $\mathsf{AX}_{\textnormal{add}}$} \label{section:fincan}

To help clarify the relationship between $\mathsf{AX}_{\textnormal{add}}$ and the standard axiomatization of comparative probability, one can verify the soundness of the finite cancellation scheme by deriving it in our additive system:

\begin{proposition} Each instance of $\mathsf{FinCan}_n$ is derivable in $\mathsf{AX}_{\textnormal{add}}$.
\end{proposition}
\begin{proof} We use all of the axioms of $\mathsf{AX}_{\textnormal{add}}$, and we appeal to $\mathsf{Assoc}$, $\mathsf{Comm}$, and the rules and axioms of $\mathsf{AX}_{\textnormal{base}}$ (such as Boolean reasoning) without mention. At a high level, the idea of the proof is simply to combine all the inequalities together (using $\mathsf{2Canc}$) and then cancel a number of terms from both side (using $\mathsf{Contr}$). 

The first observation is that $(\alpha_1,\dots,\alpha_n)\equiv_0(\beta_1,\dots,\beta_n)$ implies that $\mathbf{P}(\delta)\approx \mathsf{0}$ for every \emph{unbalanced} state description $\delta$ over these $2n$ formulas. By $\mathsf{Dist}$ and $\mathsf{Add}$, $(\alpha_1,\dots,\alpha_n)\equiv_0(\beta_1,\dots,\beta_n)$ allows us to obtain  $\sum_{\delta \in \mathcal{B}}\mathbf{P}(\delta) \approx \sum_{\delta\in\Delta}\mathbf{P}(\delta)$. By multiple applications of $\mathsf{Contr}$---and in particular its consequence,  $\mathsf{1Canc}$---we can cancel all of the summands from the left side, ending (by another application of $\mathsf{Contr}$) with $\mathsf{0}\approx \sum_{\delta \in \Delta-\mathcal{B}}\mathbf{P}(\delta)$. By yet further applications of $\mathsf{Contr}$ we conclude that $\mathbf{P}(\delta)\approx\mathsf{0}$ for each unbalanced $\delta \in \Delta-\mathcal{B}$.

Let $\Delta_{\alpha_i}$ be those state descriptions with $\alpha_i$ occurring positively (i.e., not negated), and likewise for $\Delta_{\beta_i}$. 
By $\mathsf{Add}$ we know that $\mathbf{P}(\alpha_i) \approx \sum_{\delta \in \Delta_{\alpha_i}}\mathbf{P}(\delta)$, and likewise for each $\delta$. By the previous observation that $\mathbf{P}(\delta)\approx\mathsf{0}$ for unbalanced sequences, we can conclude that $\mathbf{P}(\alpha_i) \approx \sum_{\delta \in \mathcal{B}_{\alpha_i}}\mathbf{P}(\delta)$, with $\mathcal{B}_{\alpha_i} = \Delta_{\alpha_i} \cap \mathcal{B}$. So each remaining conjunct $\alpha_i \succsim \beta_i$ in the antecedent can be rewritten as $\sum_{\delta \in \mathcal{B}_{\alpha_i}}\mathbf{P}(\delta) \succsim \sum_{\delta \in \mathcal{B}_{\beta_i}}\mathbf{P}(\delta)$. By axiom $\mathsf{2Canc}$ we can  combine all of these sums together, producing an inequality of the form:
\begin{eqnarray} \sum_{\delta \in \mathcal{B}_{\alpha_1}}\mathbf{P}(\delta) + \dots + \sum_{\delta \in \mathcal{B}_{\alpha_{n-1}}}\mathbf{P}(\delta) & \succsim & \sum_{\delta \in \mathcal{B}_{\beta_1}}\mathbf{P}(\delta) + \dots + \sum_{\delta \in \mathcal{B}_{\beta_{n-1}}}\mathbf{P}(\delta). \label{combin}
\end{eqnarray} The crux of the proof is now to consider which terms will be cancelled from both sides (once again appealing to $\mathsf{Contr}$). Any state description $\delta\in \mathcal{B}_{\beta_n}$ will appear once more on the left than on the right, so $\mathbf{P}(\delta)$ will still remain as a summand on the left. All other terms on the left will be cancelled. Similarly for the right, we will end up with the sum of terms for state descriptions in $\mathcal{B}_{\alpha_n}$, so that (\ref{combin}) becomes simply \begin{eqnarray*} \sum_{\delta \in \mathcal{B}_{\beta_n}}\mathbf{P}(\delta)  & \succsim & \sum_{\delta \in \mathcal{B}_{\alpha_n}}\mathbf{P}(\delta),
\end{eqnarray*} or, in other words, by the observations above, our desired consequent $\beta_n \succsim \alpha_n$.
\end{proof} 

The next observation shows that de Finetti's original proposal captures the first three levels of finite cancellation:
\begin{proposition} $\mathsf{Quasi}$ and $\mathsf{FinCan}_3$ are equivalent over $\mathsf{AX}_{\text{base}}$. 
\end{proposition}
%\begin{proof}[Proof sketch] To derive $\mathsf{Quasi}$ from $\mathsf{FinCan}_3$---and in fact, from $\mathsf{FinCan}_2$---simply note that $\big(\alpha,(\beta \wedge \neg \alpha)\big) \equiv_0 \big(\beta,(\alpha \wedge \neg \beta)\big)$ and $\big((\alpha \wedge \neg \beta),\beta\big) \equiv_0 \big((\beta \wedge \neg \alpha),\alpha\big)$ are valid by $\mathsf{Dist}$, so two instances of $\mathsf{FinCan}_2$ give $\alpha \succsim \beta \rightarrow (\alpha \wedge \neg \beta) \succsim (\beta \wedge \neg \alpha)$ and $(\alpha \wedge \neg \beta) \succsim (\beta \wedge \neg \alpha) \rightarrow \alpha \succsim \beta$.
%In the other direction, 
%\end{proof}

\section*{Appendix: Quadratic Qualitative Probability Structures for $n=2$}

%However, we have been unable to confirm the details of Domotor's representation theorem for FAQQPs (or to provide an alternative proof). %Specifically, the part about ``a system of symmetric hyperbolas'' seems cryptic, and has been untractable.

\begin{proposition}
The Domotor axioms $\mathsf{Q1}$--$\mathsf{Q4}$, the $n=2$ instance of $\mathsf{Q5}$, and $\mathsf{Q6}$ are complete for models with a binary event space $\Omega = \{0, 1\}$.
\end{proposition}
\begin{proof}
We aim to show there is a positive symmetric bilinear functional $\Phi: \mathbb{R}^2 \times \mathbb{R}^2 \to \mathbb{R}$ representing the order that is additionally rank $1$, and therefore factors as a product of measures. Q1--Q4, Q6 give everything except rank $1$.
The matrix of such a rank $1$ is, up to normalization,
\begin{align*}
\text{either }
    \begin{bmatrix}
1 & x \\
x & x^2
\end{bmatrix}	\text{ or }
    \begin{bmatrix}
0 & 0 \\
0 & 1
\end{bmatrix}
\end{align*}
where $x \ge 0$.
Let $\textbf{0} = \{0\}$ and $\textbf{1} = \{1\}$.
The latter matrix represents the total preorder
\begin{align}
    (\varnothing, \cdot) \sim (\textbf{0}, \textbf{0}) \sim (\textbf{0}, \Omega) \sim (\textbf{0}, \textbf{1}) \prec (\textbf{1}, \textbf{1}) \sim (\textbf{1}, \Omega) \sim (\Omega, \Omega) \label{order:degenorder}
\end{align}
where the $\cdot$ in $(\varnothing, \cdot)$ stands for any (or every) one of the four subsets of $\Omega$.

In the case of the former matrix, for $A, B \subset \Omega$ we find
\begin{align}
    \Phi(\mathbbm{1}_A, \mathbbm{1}_B) = (1+x^2) \sum_{a \in A} \mathfrak{x}(a) \sum_{b \in B} \mathfrak{x}(b) \label{varphipolynomial}
\end{align}
where the term $\mathfrak{x}(\omega)$ for $\omega \in \Omega$ is defined to be $1$ if $\omega = 0$ and $x$ if $\omega = 1$.
Considering the various (unordered) choices of $A, B$, we see that \eqref{varphipolynomial} gives $10$ polynomials in $x$, in fact $7$ if we observe that all corresponding to $(\varnothing, \cdot)$ are identically $0$.
Below we plot these curves as a function of $x$. Without loss we have divided each by the prefactor $1+x^2$ appearing in \eqref{varphipolynomial}.

\includegraphics[scale=0.8]{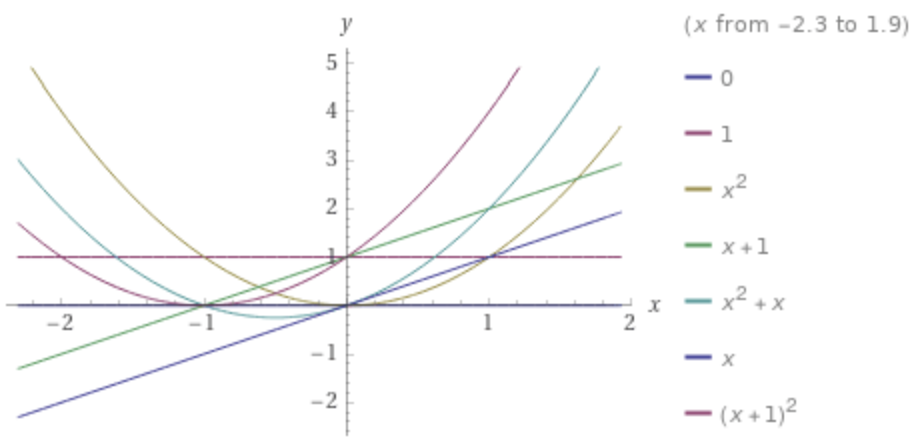}
By inspecting\footnote{This can be made fully rigorous via e.g. Sturm sequences.} the intersections and their induced subdivisions in the first quadrant\footnote{The order \eqref{order:degenorder} arising from the degenerate matrix $\begin{bmatrix}
0 & 0 \\
0 & 1
\end{bmatrix}$ corresponds to $x = +\infty$ on the plot.} we count $8$ total preorders
\begin{align}
    (\varnothing, \cdot) \sim (\textbf{1}, \Omega) \sim (\textbf{0}, \textbf{1}) \sim (\textbf{1}, \textbf{1}) \prec (\textbf{0}, \textbf{0}) \sim (\textbf{0}, \Omega) \sim (\Omega, \Omega) \label{order:nondegenfirst}\\
    (\varnothing, \cdot) \prec (\textbf{1}, \textbf{1}) \prec(\textbf{0}, \textbf{1}) \prec (\textbf{1}, \Omega) \prec (\textbf{0}, \textbf{0}) \prec (\textbf{0}, \Omega) \prec (\Omega, \Omega) \label{order:nondegen2}\\
    (\varnothing, \cdot) \prec (\textbf{1}, \textbf{1}) \prec (\textbf{0}, \textbf{1}) \prec (\textbf{0}, \textbf{0}) \sim (\textbf{1}, \Omega) \prec (\textbf{0}, \Omega) \prec (\Omega, \Omega) \label{order:nondegen3}\\
    (\varnothing, \cdot) \prec (\textbf{1}, \textbf{1}) \prec (\textbf{0}, \textbf{1}) \prec (\textbf{0}, \textbf{0}) \prec (\textbf{1}, \Omega) \prec (\textbf{0}, \Omega) \prec (\Omega, \Omega) \label{order:nondegen4}\\
    (\varnothing, \cdot) \prec (\textbf{0}, \textbf{0}) \sim (\textbf{1}, \textbf{1}) \sim (\textbf{0}, \textbf{1}) \prec (\textbf{0}, \Omega) \sim (\textbf{1}, \Omega) \prec (\Omega, \Omega) \label{order:nondegen5}\\
    (\varnothing, \cdot) \prec (\textbf{0}, \textbf{0}) \prec (\textbf{0}, \textbf{1}) \prec (\textbf{1}, \textbf{1}) \prec (\textbf{0}, \Omega) \prec (\textbf{1}, \Omega) \prec (\Omega, \Omega) \label{order:nondegen6}\\
    (\varnothing, \cdot) \prec (\textbf{0}, \textbf{0}) \prec (\textbf{0}, \textbf{1}) \prec (\textbf{1}, \textbf{1}) \sim (\textbf{0}, \Omega) \prec (\textbf{1}, \Omega) \prec (\Omega, \Omega) \label{order:nondegen7}\\
    (\varnothing, \cdot) \prec (\textbf{0}, \textbf{0}) \prec (\textbf{0}, \textbf{1}) \prec (\textbf{0}, \Omega) \prec (\textbf{1}, \textbf{1}) \prec (\textbf{1}, \Omega) \prec (\Omega, \Omega)  \label{order:nondegenlast}
\end{align}

Thus we have $9$ total preorders \eqref{order:degenorder}, \eqref{order:nondegenfirst}--\eqref{order:nondegenlast} representable by a rank $1$.

We claim these are exactly the ones satisfying the stated Domotor axioms. The soundness direction is straightforward while the completeness direction can be shown via casework. Note that the $n=2$ instance of $\mathsf{Q5}$ is just monotonicity (in both arguments, by symmetry), namely that
\begin{align*}
    (A, B) \preceq (A, C) \Rightarrow (D, B) \preceq (D, C)
\end{align*}
and analogously in the first argument.
This presumes that $(\varnothing, \Omega) \prec (A, B)$ and $(\varnothing, \Omega) \prec (D, C)$.
We also have strict monotonicity.
Below we will generally take $\mathsf{Q1}$--$\mathsf{Q4}$ for granted, which just amount to the axioms of a nondegenerate total preorder, symmetric on the pairs.
Our cases are:
\begin{enumerate}
\item If $(\textbf{0}, \Omega) \prec (\textbf{1}, \Omega)$:
\begin{enumerate}
    \item If $(\varnothing, \Omega) \sim (\textbf{0}, \textbf{0})$: conclude order \eqref{order:degenorder}, by $\mathsf{Q5}$ and $\mathsf{Q6}$.
    \item If $(\varnothing, \Omega) \prec (\textbf{0}, \textbf{0})$: by $\mathsf{Q5}$ conclude $(\textbf{0}, \textbf{0}) \prec (\textbf{0}, \textbf{1}) \prec (\textbf{1}, \textbf{1})$. By $\mathsf{Q6}$ conclude $(\textbf{1}, \Omega) \prec (\Omega, \Omega)$.
\begin{enumerate}
    \item If $(\textbf{1}, \textbf{1}) \prec (\textbf{0}, \Omega)$: order \eqref{order:nondegen6}
    \item If $(\textbf{1}, \textbf{1}) \sim (\textbf{0}, \Omega)$: order \eqref{order:nondegen7}
    \item If $(\textbf{1}, \textbf{1}) \succ (\textbf{0}, \Omega)$: order \eqref{order:nondegenlast}
\end{enumerate}

\end{enumerate}
\item If $(\textbf{0}, \Omega) \sim (\textbf{1}, \Omega)$ conclude order \eqref{order:nondegen5}, by repeated applications of $\mathsf{Q5}$ and $\mathsf{Q6}$.
\item If $(\textbf{0}, \Omega) \succ (\textbf{1}, \Omega)$:
\begin{enumerate}
\item If $(\varnothing, \Omega) \sim (\textbf{1}, \textbf{1})$: conclude order \eqref{order:nondegenfirst} by $\mathsf{Q5}$, $\mathsf{Q6}$.
\item If If $(\varnothing, \Omega) \prec (\textbf{1}, \textbf{1})$: by $\mathsf{Q5}$ conclude $(\textbf{1}, \textbf{1}) \prec (\textbf{0}, \textbf{1}) \prec (\textbf{0}, \textbf{0})$. By $\mathsf{Q6}$ conclude $(\textbf{0}, \Omega) \prec (\Omega, \Omega)$.
\begin{enumerate}
    \item If $(\textbf{0}, \textbf{0}) \prec (\textbf{1}, \Omega)$: order \eqref{order:nondegen4}
    \item If $(\textbf{0}, \textbf{0}) \sim (\textbf{1}, \Omega)$: order \eqref{order:nondegen3}
    \item If $(\textbf{0}, \textbf{0}) \succ (\textbf{1}, \Omega)$: order \eqref{order:nondegen2}
\end{enumerate}
\end{enumerate}
\end{enumerate}
\end{proof}

\end{document}